\documentclass[11pt,a4paper,reqno]{amsart}
\pdfoutput=1
\usepackage{amsmath,amsthm,latexsym,amssymb,wasysym}
\usepackage[margin=1.0in,tmargin=0.9in,bmargin=0.80in]{geometry}
\usepackage{graphicx}
\usepackage{subcaption}
\usepackage{hyperref}
\usepackage{bbm}
\usepackage{xargs}
\usepackage{enumerate}
\usepackage{enumitem}
\setlist[enumerate]{label={\upshape(\roman*)}}
\usepackage[textsize=footnotesize]{todonotes}
\usepackage[flushleft]{threeparttable}
\usepackage{multirow}
\usepackage{longtable}
\usepackage{times}
\usepackage{microtype}
\usepackage[numbers]{natbib}
\usepackage{tabularx}
\usepackage{ltablex}
\usepackage{placeins}

\usepackage[linesnumbered,ruled,vlined]{algorithm2e}
\SetKwInput{KwInput}{Input}               
\SetKwInput{KwOutput}{Output}  

\keepXColumns

\usepackage{listings}

\usepackage{xcolor}
\definecolor{codegreen}{rgb}{0,0.6,0}
\definecolor{codegray}{rgb}{0.5,0.5,0.5}
\definecolor{codepurple}{rgb}{0.58,0,0.82}
\definecolor{backcolour}{rgb}{0.95,0.95,0.92}

\lstdefinestyle{myStyle}{
	language=Python,
	backgroundcolor=\color{white},   
	commentstyle=\color{codegreen},
	keywordstyle=\color{magenta},
	numberstyle=\tiny\color{codegray},
	stringstyle=\color{codegreen},
	basicstyle=\ttfamily\footnotesize,  
	breakatwhitespace=false,         
	breaklines=true,                 
	captionpos=b,                    
	keepspaces=true,                 
	numbers=left,                    
	numbersep=5pt,
	showspaces=false,                
	showstringspaces=false,
	showtabs=false,                  
	tabsize=2,
	frame=single
}

\usepackage{float}
\definecolor{bred}{rgb}{0.8,0,0}
\hypersetup{colorlinks,linkcolor={blue},citecolor={bred},urlcolor={blue}}

\newtheorem{theorem}{Theorem}[section]
\newtheorem{proposition}[theorem]{Proposition}
\newtheorem{lemma}[theorem]{Lemma}
\newtheorem{corollary}[theorem]{Corollary}
\newtheorem{remark}[theorem]{Remark}
\newtheorem{definition}[theorem]{Definition}

\newtheorem{assumption}{Assumption}

\newcommand{\K}{\mathbb{K}}
\newcommand{\E}{\mathbb{E}}
\newcommand{\N}{\mathbb{N}}
\newcommand{\R}{\mathbb{R}}
\newcommand{\1}{\mathbbm{1}}
\newcommand{\dee}{\mathrm{d}}
\newcommand{\mc}{\mathcal}

\newcommand{\lfrf}[1]{\left\lfloor #1 \right\rfloor}
\newcommand{\lcrc}[1]{\left\lceil #1 \right\rceil}
\newcommand{\lara}[1]{\left\langle #1 \right\rangle}
\newcommand{\lbrb}[1]{\left| #1 \right|}
\newcommand{\lsrs}[1]{\left[ #1 \right]}

\makeatletter
\def\moverlay{\mathpalette\mov@rlay}
\def\mov@rlay#1#2{\leavevmode\vtop{%
   \baselineskip\z@skip \lineskiplimit-\maxdimen
   \ialign{\hfil$\m@th#1##$\hfil\cr#2\crcr}}}
\newcommand{\charfusion}[3][\mathord]{
    #1{\ifx#1\mathop\vphantom{#2}\fi
        \mathpalette\mov@rlay{#2\cr#3}
      }
    \ifx#1\mathop\expandafter\displaylimits\fi}
\makeatother

\newcommand{\bigcupdot}{\charfusion[\mathop]{\bigcup}{\cdot}}

\begin{document}
	
\title[]{\vspace*{-2cm}Robust SGLD algorithm for solving \\non-convex distributionally robust optimisation problems}

\author[A. Neufeld]{Ariel Neufeld}
\author[M. Ng]{Matthew Ng Cheng En}
\author[Y. Zhang]{Ying Zhang} 

{\small
\address{Division of Mathematical Sciences, Nanyang Technological University, 21 Nanyang Link, 637371 Singapore}
\email{ariel.neufeld@ntu.edu.sg}

\address{Division of Mathematical Sciences, Nanyang Technological University, 21 Nanyang Link, 637371 Singapore}
\email{matt0037@e.ntu.edu.sg}

\address{Financial Technology Thrust, Society Hub, The Hong Kong University of Science and Technology (Guangzhou), 
Guangzhou, China}
\email{yingzhang@hkust-gz.edu.cn}
}

\date{}
\thanks
{
}

\begin{abstract}
	In this paper we develop a Stochastic Gradient Langevin Dynamics (SGLD) algorithm tailored for solving a certain class of non-convex distributionally robust optimisation (DRO) problems. By deriving non-asymptotic convergence bounds, we build an algorithm which for any prescribed accuracy $\varepsilon>0$ outputs an estimator whose expected excess risk is at most $\varepsilon$. 
	As a concrete application, we consider the problem of identifying the best non-linear estimator of a given regression model involving a neural network using adversarially corrupted samples. We formulate this problem as a DRO problem and demonstrate both theoretically and numerically the applicability of the proposed robust SGLD algorithm. Moreover, numerical experiments show that the robust SGLD estimator outperforms the estimator obtained using vanilla SGLD in terms of test accuracy, which highlights the advantage of incorporating model uncertainty when optimising with perturbed samples.
\end{abstract}

\maketitle

\section{Introduction}
Given $\Xi\subseteq \R^m$, a distance $d_c(\cdot,\cdot)$ on the space of probability measures $\mathcal{P}(\Xi)$, a reference measure $\mu_0\in \mathcal{P}(\Xi)$, parameters $\eta_1, \eta_2>0$, and a possibly non-convex utility function $U:\R^d\times \R^m \to \R$, we consider the following non-convex distributionally robust stochastic optimisation problem 
\begin{equation}
\begin{split}
\label{eqn:dro_problem-Intro}
	\text{minimise } \R^d\ni \theta\mapsto u(\theta):=\left\{\sup_{\mu\in \mc{P}(\Xi)}\left(\int_{\Xi} U(\theta, x)\ \dee\mu(x)-\frac{d_c^2(\mu_0,\mu)}{2\eta_2}\right)+\frac{\eta_1}{2}|\theta|^{2}\right\}.
\end{split}	
\end{equation}
Here, $\mu_0$ represents an estimate for the true but unknown law of the environment of the optimisation problem while $\eta_2>0$ represents the level of model uncertainty an agent has in the environment. Indeed, the smaller $\eta_2$ is chosen the larger the penalty term $\frac{d_c^2(\mu_0,\mu)}{2\eta_2}$ becomes,  hence the more certain the agent believes that his estimated measure $\mu_0$ actually represents the true law of the environment.
\\

The goal of this paper is to \textit{construct} an estimator $\hat{\theta}$ which minimises the expected excess risk associated with \eqref{eqn:dro_problem-Intro}. More precisely, we aim to build an algorithm which for any prescribed accuracy $\varepsilon>0$  outputs a $d$-dimensional estimator $\hat{\theta}_\varepsilon$ defined on a suitable probability space $(\Omega, \mathcal{F},\mathbb{P})$ such that
\begin{equation}\label{excess-risk-intro}
	\mathbb{E}_{\mathbb{P}}[u(\hat{\theta}_\varepsilon)]-\inf_{\theta\in \mathbb{R}^d} u(\theta)< \varepsilon.
	\end{equation}
\vspace{0.1cm}

Already in \cite{ellsberg1961risk,knight1921risk}, Knight and Ellsberg argued that an agent who is making decisions  cannot have the precise knowledge of the true law characterising the environment and hence should take model uncertainty under consideration. 
In distributionally robust optimisation (DRO) problems, there are two approaches to overcome the problem of model uncertainty. In  the first approach, one considers a set of probability measures representing all candidates  for the true but unknown law of the environment and one optimises over the worst-case law among those candidate laws. A typical example for such an ambiguity set of laws would be a Wasserstein-ball of certain radius around a reference measure \cite{blanchet2019quantifying, mohajerin2018data}.  In the second approach, like in our DRO problem \eqref{eqn:dro_problem-Intro}, one starts with a reference measure $\mu_0$ representing the estimated law for the  true but unknown law of the environment and then introduces a penalty function which penalises all probability measures the further they are away from that given reference measure \cite{bartl2017computational}. One then optimises robustly over all possible laws while the penalty function controls how much each law can contribute to the optimisation problem. Typically in both approaches, the corresponding reference measure has been estimated either from historical data by taking the empirical measure, or through experts' insights. 
We refer to, e.g., 
\cite{bayraktar2024data, delage2010distributionally, eckstein2020robust, neufeld2023markov, rujeerapaiboon2016robust} 
for DRO applications in financial engineering, 
to, e.g., \cite{aolaritei2023distributional, chen2022distributionally, gao2023distributionally, kong2020appointment, neufeld2022numerical, taskesen2021sequential} for DRO applications in operations research and artificial intelligence, as well as to, e.g.,  \cite{bartl2021sensitivity,bartlneufeldpark2023sensitivity, bartl2024numerical, gao2022wasserstein,jiang2024sensitivity, nendel2022parametric, obloj2021distributionally, sauldubois2024first} for recent developments on the sensitivity analysis of DRO problems in various fields.
\\

In this paper, we develop a Stochastic Gradient Langevin Dynamics (SGLD) algorithm that can minimise the expected excess risk of a certain class of DRO problems of the form \eqref{eqn:dro_problem-Intro} as described in \eqref{excess-risk-intro}. In Theorem~\ref{theorem:excess risk},  we obtain (under Assumptions~\ref{assumption:1}--\ref{assumption:4}) non-asymptotic convergence bounds for our robust SGLD algorithm \eqref{eqn:sgld}--\eqref{defn:grad_v_delta}. As a consequence of the non-asymptotic convergence bounds, we can indeed develop an algorithm which for every prescribed accuracy~$\varepsilon>0$ outputs a $d$-dimensional estimator which minimises the expected excess risk as defined in~\eqref{excess-risk-intro}. We refer to Algorithm~\ref{Algorithm1} and its theoretical properties stated in Corollary~\ref{corollary:main_result}.
\\

SGLD algorithms are commonly-used methodologies to solve (non-convex) stochastic optimisation problems \cite{chen2020stationary,deng2022adaptively,farghly2021time,kinoshita2022improved,neufeld2022non,raginsky2017non,xu2018global,zhang2017hitting,zhu2024uniform} as well as the sampling problem \cite{convex,ppbdm,chau2021stochastic,pmlr-v65-dalalyan17a,dalalyan2019user,welling2011bayesian,zou2021faster}. Compared to stochastic gradient descent (SGD) algorithms, SGLD algorithms include an additional noise term in each iteration which allows them to better overcome local minima than SGD algorithms. We refer to \cite{li2016preconditioned,lim2021polygonal,lim2022langevin,lim2021non,lovas2023taming,wu2020characterizing} for the development of SGLD based algorithms to solve stochastic optimisation problems involving the training of neural networks, 
	to \cite{chu2021non,sglddiscont} to solve portfolio optimisation problems, 
to \cite{bras2023langevincontrol} for deep hedging,
as well as to \cite{bras2023langevin,durmus2019analysis,kamalaruban2020robust,zhang2022low,zhang2023nonasymptotic} for large-scale Bayesian inference.
\\

However, so far, no SGLD algorithm has been developed tailored to solve general non-convex DRO problems \eqref{eqn:dro_problem-Intro}.
In \cite{li2023policy}, the authors use a \textit{standard} projected SGLD algorithm to solve robust Markov decision problems (MDP) defined on \textit{finite} state and action spaces where the corresponding ambiguity set of probability measures is not required to be rectangular. Since their state space is finite, they can exploit the relation between the value function of the robust MDP and the sampling problem from the Gibbs distribution in order to show that a standard  (i.e.\ not tailored to solve DRO problems)
Langevin dynamics-based algorithm can minimise the expected excess risk arbitrarily well. In \cite{liu2025stochastic}, the authors consider solving min-sum-max optimisation problems which include Wasserstein DRO problems. They propose an iterative smoothing algorithm and provide an almost sure convergence guarantee to a directional stationary point under a certain Clarke regularity condition. However, such a stationary point may not be a global minimiser to the problem under consideration.
\\

As a concrete application of the general framework \eqref{eqn:dro_problem-Intro}, we consider the problem of identifying the best non-linear estimator of a given regression model involving a neural network using training samples that are corrupted with adversarial perturbations. We formulate the aforementioned problem as a DRO problem and solve it using our proposed robust SGLD algorithm so as to potentially mitigate the impact of outlying samples and obtain an estimator that is consistent with the true distribution. We show in Section \ref{sec:Application} that the DRO problem satisfies Assumptions~\ref{assumption:1}-\ref{assumption:4}, hence Theorem~\ref{theorem:excess risk} applies, which provides a theoretical guarantee for robust SGLD to converge. Numerical experiments support our main result, and empirically demonstrate that our robust SGLD algorithm outperforms the vanilla SGLD algorithm \cite{welling2011bayesian} in terms of test accuracy when choosing the penalisation parameter $\eta_2>0$ in a suitable way. 
\\

The rest of this paper is organised as follows. In Section~\ref{sec:AssumpMainRes}, we introduce the setting of our distributionally robust optimisation problem, the assumptions imposed, as well as present the main results of our paper. As a concrete application of our general setting in Section~\ref{sec:AssumpMainRes}, we consider in Section~\ref{sec:Application} the DRO problem of identifying the best non-linear estimator of a given regression model using adversarially corrupted training data. We show that it fits into our general setting with the corresponding assumptions imposed in Section~\ref{sec:AssumpMainRes}. In particular, the main results of our paper can be applied to this concrete DRO problem. We provide numerical results which support our theoretical findings. In Section~\ref{sec:proofOverview} we present an overview of the proofs of our main results, whereas in Sections~\ref{sec:proof_sec:AssumpMainRes}--\ref{sec:proof_sec:proofOverview} we present the remaining proofs of all results and statements presented in Sections~\ref{sec:AssumpMainRes}--\ref{sec:proofOverview}.
\\

\textbf{Notation.} We conclude this section by introducing some notation. Let $\R$ (respectively, $\R_{\geq 0}$) denote the set of (non-negative) real numbers. Let $(\Omega,\mathcal{F})$ be a measurable space. Given a random variable $Z$ and a probability measure $\mathbb{Q}$ on $(\Omega, \mc{F})$, we denote by $\E_\mathbb{Q}[Z]:=\int_\Omega Z\ \dee\mathbb{Q}$ the expectation of $Z$ with respect to $\mathbb{Q}$. For $p\in[1,\infty)$, $L^p(\Omega, \mc{F}, \mathbb{Q})$, or $L^p(\mathbb{Q})$ for short when the measurable space in consideration is clear from the context, is used to denote the space of $p$-integrable real-valued random variables on $\Omega$ with respect to $\mathbb{P}$. Fix integers $d, m \geq 1$. A random vector $\theta\in\R^d$ is always understood to be a column vector unless stated otherwise, with the exception of the gradient $\nabla f$ of a given function $f:\R^d\to \R$ being a row vector as consistent with the interpretation of $\nabla$ acting as a linear operator from $\R^d$ to $\R$ and having matrix representation in $\R^{1\times d}$. For an $\R^d$-valued random variable $Z$, its law on $\mathcal{B}(\R^d)$, i.e. the Borel sigma-algebra of $\R^d$, is denoted by $\mathcal{L}(Z)$. We denote by $I_{d}$ the $d$-dimensional identity matrix and by  $\mc{N}(0,I_{d})$ the $d$-dimensional standard normal distribution.
 For a positive real number $a$, we denote by $\lfrf{a}$ its integer part, and $\lcrc{a} = \lfrf{a}+1 $. The notation $\1_\cdot$ is used to denote indicator functions. For pairwise disjoint sets $\{A_j\}_{j\in\mathbb{N}}$, we denote by $\bigcupdot_{j\in\N}A_j$ their union. Given a normed space $(\Xi, \lVert\cdot\rVert_\Xi)$ and an element $x\in \Xi$, we denote the norm of $x$ by $\lVert x\rVert_\Xi$. In the particular case $\Xi=\R^d$ and $\lVert\cdot\rVert$ is the Euclidean norm, we understand the notation $|x|$ as referring to $|x| = \lVert x\rVert_{\R^d}$ for $x\in\R^d$. Similarly, for a real-valued $m\times d$ matrix $A\in\R^{m\times d}$, we understand $|A|$ as referring to the operator norm $|A| = \sup\{|Ax|: |x|\leq 1, x\in\R^d\}$. The Euclidean scalar product is denoted by $\langle \cdot,\cdot\rangle$. For any normed space $\Xi$, let $\mathcal{P}(\Xi)$ denote the set of probability measures on $\mathcal{B}(\Xi)$. For $\mu,\mu'\in\mathcal{P}(\Xi)$, let $\mathcal{C}(\mu,\mu')$ denote the set of couplings of $\mu,\mu'$, that is, probability measures $\zeta$ on $\mathcal{B}(\Xi\times \Xi)$ such that its respective marginals are $\mu,\mu'$. Given two Borel probability measures $\mu,\mu'\in\mc{P}(\Xi)$ and a cost function $c:\Xi\times\Xi\to [0,\infty]$ in the sense of \cite{bartl2017computational}, the cost of transportation between $\mu$ and $\mu'$ is defined by
\begin{align*}
d_c(\mu,\mu'):=
\inf_{\zeta\in\mathcal{C}(\mu,\mu')}\int_{\Xi\times \Xi}c(\theta,\theta')\ \dee\zeta(\theta,\theta').
\end{align*}

\section{Assumptions and main results}\label{sec:AssumpMainRes}
\subsection{Problem Statement}
Let $\Xi$ be a compact subset of $\R^m$,
let $U:\R^d\times \Xi\to \R$ be a measurable function, let $c:\Xi\times\Xi \to \R_{\geq 0}$ be defined as $c(x,x'):=|x-x'|^p$ for some $p\in[1,\infty)$, and let $\eta_1,\eta_2>0$ be regularisation parameters.
Given a reference probability measure $\mu_0\in\mc{P}(\Xi)$, the main problem of interest is in the form of the following (regularised) distributionally robust stochastic optimisation problem 
\begin{align}
\label{eqn:dro_problem}
\text{minimise } \R^d\ni \theta\mapsto u(\theta):=\left\{\sup_{\mu\in \mc{P}(\Xi)}\left(\int_{\Xi} U(\theta, x)\ \dee\mu(x)-\frac{d_c^2(\mu_0,\mu)}{2\eta_2}\right)+\frac{\eta_1}{2}|\theta|^{2}\right\}.
\end{align}

%
%

\subsection{Assumptions} 
In this section we present the assumptions imposed on the distributionally robust stochastic optimisation problem~\eqref{eqn:dro_problem}. 

\begin{assumption}
\label{assumption:1}
$\Xi$ is a compact subset of $\R^m$.  Denote, henceforth, $M_\Xi:=\max_{x\in\Xi}|x|<\infty$.
\end{assumption}

\begin{assumption}
\label{assumption:3} For every $x\in\Xi$, the mapping $\theta\mapsto U(\theta,x)$ is continuously differentiable. Moreover, there exists a constant $L_\nabla>0$ such that for all $\theta_1,\theta_2\in\R^d$ and $x\in\Xi$,
\[
\lbrb{\nabla_\theta U(\theta_1,x)-\nabla_\theta U(\theta_2, x)} \leq L_{\nabla}|\theta_1-\theta_2|.
\]
In addition, there exists a constant $K_\nabla>1$ such that for all $\theta\in\R^d$ and $x\in\Xi$, 
\[
\lbrb{\nabla_\theta U(\theta,x)} \leq K_{\nabla}.
\]
\end{assumption}

\begin{remark}
\label{remark:u_lipschitz}
Under Assumption \ref{assumption:3}, it holds for all $\theta_1,\theta_2\in\R^d$ and $x\in\Xi$ that
\[
\lbrb{ U(\theta_1,x)- U(\theta_2, x)} \leq K_\nabla |\theta_1-\theta_2|.
\]
Moreover, under Assumptions \ref{assumption:1} and \ref{assumption:3}, it holds for all $\theta\in\R^d$ and $x\in\Xi$ that
\[
|U(\theta,x)|\leq \tilde{K}_\nabla (1+|\theta|),
\]
where $\tilde{K}_\nabla := \max\left\{K_\nabla, \max_{x\in\Xi}|U(0,x)|\right\}$.
\end{remark}

\begin{assumption}
\label{assumption:4}
There exists a constant $J_U>0$ such that for all $\theta\in \R^d$ and $x_1,x_2\in\Xi$,
\[
\lbrb{U(\theta, x_1)-U(\theta,x_2)} \leq J_U (1+|\theta|)|x_1-x_2|.
\]
\end{assumption}

	\begin{remark}
		We comment on our assumptions:
		\begin{enumerate}
			\item Assumption~\ref{assumption:1} is a technical condition which we impose to enable the transformation of DRO problem~\eqref{eqn:dro_problem} to a discretised and smoothed minimisation problem~\eqref{defn:v_delta}-\eqref{eqn:dro_problem_dual_discrete_smoothed}. Then, the SGLD algorithm~\eqref{eqn:sgld}-\eqref{defn:grad_v_delta} can be applied to solve the problem~\eqref{defn:v_delta}-\eqref{eqn:dro_problem_dual_discrete_smoothed} with theoretical guarantees for its performance, see Theorem~\ref{theorem:excess risk}. We refer to Remark~\ref{rmk:sgldexplanation} for detailed explanations. 
			\item In Assumption~\ref{assumption:3}, we assume that $\nabla_\theta U$ is globally Lipschitz continuous and bounded. These conditions ensure that the stochastic gradient of the proposed SGLD algorithm~\eqref{eqn:sgld}-\eqref{defn:grad_v_delta} is globally Lipschitz continuous, see Proposition~\ref{prop:glob_lipschitz}. 
			This allows us to establish an error bound for the SGLD algorithm~\eqref{eqn:sgld}-\eqref{defn:grad_v_delta}, see Proposition~\ref{prop:excess_risk_discrete}.
			\item Assumption~\ref{assumption:4} is essential to establish the quadrature error between the dual problem $z_D$~\eqref{eqn:dro_problem_dual} and its discretised version $z_{D,\ell,\mathfrak{j}}$~\eqref{eqn:dro_problem_discrete} as shown in Proposition~\ref{prop:quadrature}.
		\end{enumerate}
		All together, these assumptions allow us to derive an upper bound with explicit constants for the expected excess risk associated with DRO problem~\eqref{eqn:dro_problem}, 
		which can be made arbitrarily small, see Theorem~\ref{theorem:excess risk} and Corollary~\ref{corollary:main_result}.
	\end{remark}

\subsection{Main Result} 
In this section, we define our robust SGLD algorithm constructed on a suitable probability space $(\Omega,\mathcal{F},\mathbb{P})$ and
state our main result, which is a non-asymptotic upper bound on the excess risk under $\mathbb{P}$  
derived under the stated assumptions.

\begin{definition}\label{def:iota} Let $\iota: \R \to \R_{\geq 0}$ be defined by $\iota(\alpha) = \log(\cosh{\alpha})$.
\end{definition}
\begin{remark}\label{remark:iota_dissipativity}
We note that $\iota$ defined in Definition~\ref{def:iota} is a surjective and continuously differentiable function such that its derivative $\iota'$ is $L_{\iota}$-Lipschitz continuous and bounded by some constant $M_{\iota}>0$, and $\iota\cdot\iota'$ is $\tilde{L}_\iota$-Lipschitz continuous. Moreover, there exist constants $a_\iota,b_\iota>0$ such that for all $\alpha\in\R$, the following dissipativity condition holds:
\begin{align}
\alpha\iota(\alpha)\iota'(\alpha)\geq a_\iota \alpha^2-b_\iota.\label{eqn:dissipativity_cond}
\end{align} 
\end{remark}

Given positive integers $\ell, \mathfrak{j}>0$, we define the set of dyadic rationals
\begin{align}
\K_{\ell,\mathfrak{j}}:= \left\{-2^{\ell-1},-2^{\ell-1}+\frac{1}{2^\mathfrak{j}},\cdots,2^{\ell-1}-\frac{1}{2^\mathfrak{j}} \right\},\label{defn:dyadic}
\end{align}
and fix, henceforth, an $\ell\in\N$ large enough such that $\Xi\subseteq [-2^{\ell-1},2^{\ell-1})^m$. We also denote the finite set
\begin{align*}
\{\xi_j^{\ell,\mathfrak{j}}\}_{j=1,\cdots, N_{\ell,\mathfrak{j}}}:=&\ \Xi\cap \K_{\ell,\mathfrak{j}}^m,\nonumber\\
N_{\ell,\mathfrak{j}}:=&\ 2^{m(\ell+\mathfrak{j})},
\end{align*}
where $\K_{\ell,\mathfrak{j}}^m$ denotes the $m$-th Cartesian power of the set $\K_{\ell,\mathfrak{j}}$. In addition, we denote, for the ease of notation, $\xi_j:= \xi_j^{\ell, \mathfrak{j}}$ and $N:=N_{\ell,\mathfrak{j}}$. That is, the dependence of the quantities $\xi_j$ and $N$ on $\ell$ and $\mathfrak{j}$ are suppressed for the sake of brevity. In addition, we fix a probability space $(\Omega,\mathcal{F},\mathbb{P})$ such that $(X_n)_{n\in\N_0}$, $(Z_n)_{n\in\N_0}$ are i.i.d.\ sequences with $\mathbb{P}\circ X_0^{-1}=\mu_0\in\mc{P}(\Xi)$ and $\mathbb{P}\circ Z_0^{-1}\sim\mc{N}(0,I_{d+1})$. We assume $(X_n)_{n\in\N_0}$ and $(Z_n)_{n\in\N_0}$ are independent. \\

The proposed robust SGLD algorithm yields a sequence of estimators $(\hat{\bar{\theta}}^{\lambda, \delta, \ell,\mathfrak{j}}_n)_{n\in\N_0}$ with $\hat{\bar{\theta}}^{\lambda, \delta, \ell,\mathfrak{j}}_{n} := (\hat{\theta}^{\lambda, \delta, \ell,\mathfrak{j}}_n,\hat{\alpha}^{\lambda, \delta, \ell,\mathfrak{j}}_n)\in \R^d \times \R$, which, for a given $\delta>0$, choice of step size $\lambda\in(0,\lambda_{\max,\delta})$, where the maximum step size restriction $\lambda_{\max,\delta}$ is given explicitly in \eqref{defn:step}, $\mathfrak{j}\in\N$ controlling the grid mesh, and $\beta>0$ representing the inverse temperature parameter \cite{welling2011bayesian} which controls the magnitude of the Gaussian noise $(Z_n)_{n\in\N_0}$, is defined recursively as
\begin{align}
\hat{\bar{\theta}}^{\lambda, \delta, \ell, \mathfrak{j}}_{n+1}
:=&\ \hat{\bar{\theta}}^{\lambda, \delta, \ell,\mathfrak{j}}_n-\lambda H^{\delta, \ell,\mathfrak{j}}(\hat{\bar{\theta}}_n^{\lambda, \delta, \ell,\mathfrak{j}},X_{n+1})+\sqrt{2\lambda\beta^{-1}}Z_{n+1},\qquad
\hat{\bar{\theta}}_0^{\lambda, \delta, \ell,\mathfrak{j}}
= \bar{\theta}_0,
\label{eqn:sgld}
\end{align}
where
\begin{align}
\begin{split}\label{defn:grad_v_delta}
H^{\delta, \ell,\mathfrak{j}}(\bar{\theta},x):=&\ \begin{pmatrix}\eta_1\theta^T+\frac{\sum^N_{j=1}F_j^{\delta,\ell,\mathfrak{j}}(\bar{\theta},x)\nabla_\theta U(\theta, \xi_j)}{\sum^N_{j=1}F_j^{\delta,\ell,\mathfrak{j}}(\bar{\theta}, x)},&\ \eta_2\iota(\alpha)\iota'(\alpha)-\frac{\sum^N_{j=1}F_j^{\delta,\ell,\mathfrak{j}}(\bar{\theta},x)\iota'(\alpha)|x-\xi_j|^p}{\sum^N_{j=1}F_j^{\delta,\ell,\mathfrak{j}}(\bar{\theta}, x)}\end{pmatrix}^T,\\
F_j^{\delta,\ell,\mathfrak{j}}(\bar{\theta}, x):=&\ \exp\left[\frac{1}{\delta}\left(U(\theta,\xi_j)-\iota(\alpha)|x-\xi_j|^p\right)\right],
\end{split}
\end{align}
%
with $\bar{\theta}=:(\theta,\alpha)\in \R^d\times \R$ and $x\in\Xi\subset\R^m$. The following main result gives a non-asymptotic upper bound for the expected excess risk under $\mathbb{P}$ 
of the SGLD algorithm associated with \eqref{eqn:dro_problem}.

\begin{remark}\label{rmk:sgldexplanation}
We note that the proposed algorithm~\eqref{eqn:sgld}-\eqref{defn:grad_v_delta} is the SGLD algorithm with a specifically designed stochastic gradient given in \eqref{defn:grad_v_delta}. The motivation for the design is as follows. Recall that our goal is to use (a variant of) the SGLD algorithm to solve DRO problem~\eqref{eqn:dro_problem} with theoretical guarantees for its performance. To avoid the infinite-dimensional maximisation problem over probability measures, the main idea is to use duality arguments to reformulate problem~\eqref{eqn:dro_problem} as a standard minimisation problem (i.e., without distributional robustness) to which convergence results for the SGLD algorithm, e.g., \cite[Corollary 2.8]{zhang2023nonasymptotic}, can be applied. We achieve this by taking three steps as described in Sections~\ref{subsec:dual_problem_formulation}, \ref{subsec:finite_grid}, and \ref{subsec:nesterov}, respectively. More precisely,
\begin{enumerate}
\item in Section~\ref{subsec:dual_problem_formulation}, we apply the duality result~\cite[Theorem 2.4]{bartl2017computational} to obtain the dual formulation of problem~\eqref{eqn:dro_problem}. By further applying the transformation $\mathfrak{a} = \iota(\alpha)$ with $\iota$ defined in Definition~\ref{def:iota}, and by using the surjectivity of $\iota: \R \to \R_{\geq 0}$, we are able to derive the corresponding optimisation problem of dual problem \eqref{eqn:dro_problem_dual} over the whole domain $\mathbb{R}^{d+1}$ instead of $\mathbb{R}^d\times [0,\infty)$, see \eqref{defn:small_v}-\eqref{eqn:dro_problem_dual_transformed}, which is crucial to apply SGLD algorithms. However, the function~$\tilde{V}$ in \eqref{defn:small_v} is not necessarily differentiable w.r.t.\ $\theta$ under our assumptions, hence, classical convergence results for the SGLD algorithm do not apply.
\item To fix this issue, in Section~\ref{subsec:finite_grid}, we reduce the dual problem~\eqref{eqn:dro_problem_dual} to a discretised version \eqref{eqn:dro_problem_discrete} where the data points are defined on a finite grid $\Xi\cap \K_{\ell,\mathfrak{j}}^m$ with $\K_{\ell,\mathfrak{j}}$ defined in \eqref{defn:dyadic}. Consequently, $\tilde{V}$ in~\eqref{defn:small_v} becomes $\tilde{V}^{\ell,\mathfrak{j}}$ in~\eqref{defn:small_v_discrete} which can be approximated using continuously differentiable functions through certain smoothing methods.
\item Finally, in Section~\ref{subsec:nesterov}, we apply Nesterov's smoothing technique to obtain a continuously differentiable approximation $\tilde{V}^{\delta,\ell,\mathfrak{j}}$ of $\tilde{V}^{\ell,\mathfrak{j}}$ defined in \eqref{defn:small_v_delta_discrete}. With this approximation, we further obtain a discretised and smoothed minimisation problem~\eqref{defn:v_delta}-\eqref{eqn:dro_problem_dual_discrete_smoothed} which can be solved by using the SGLD algorithm with a theoretical guarantee provided by \cite[Corollary 2.8]{zhang2023nonasymptotic}.
\end{enumerate}
Following these steps, we define the stochastic gradient of our proposed algorithm to be the gradient of $\tilde{V}^{\delta,\ell,\mathfrak{j}}$ w.r.t.\ $\theta$, given explicitly in \eqref{defn:grad_v_delta}.

To see that the robust SGLD algorithm \eqref{eqn:sgld}-\eqref{defn:grad_v_delta} solves the DRO problem \eqref{eqn:dro_problem}, note that the approximation error between problem~\eqref{eqn:dro_problem} and its discretised and smoothed version~\eqref{defn:v_delta}-\eqref{eqn:dro_problem_dual_discrete_smoothed} can be made arbitrarily small by choosing appropriately parameters $\delta,\ell,\mathfrak{j}$ as shown in Proposition~\ref{prop:quadrature} and Corollary~\ref{co:smoothing}. Hereby, we employ the regularity assumptions of $U:\R^d\times \Xi\to \R$  and $\iota: \R \to \R_{\geq 0}$, as well as the assumption that $\Xi \subseteq \R^m$ is compact and so $M_\Xi:=\max_{x\in\Xi}|x|$ is finite. These, together with the convergence result in Corollary~\ref{corollary:excess_risk_discrete}, provide a theoretical guarantee for the performance of the proposed algorithm. Crucially, using this approach, we obtain explicit constants on the convergence upper bound as well as the explicit rate of convergence. We refer to Theorem~\ref{theorem:excess risk} and Corollary~\ref{corollary:main_result} below for more details.
\end{remark}

\begin{theorem}
\label{theorem:excess risk}
Let Assumptions \ref{assumption:1}, \ref{assumption:3}, and \ref{assumption:4} hold. Let $\beta,\delta>0$, and let $\bar{\theta}_0\in L^4(\Omega, \mc{F},\mathbb{P};\R^{d+1})$.  Moreover, let $(\hat{{\theta}}^{\lambda, \delta,\ell,\mathfrak{j}}_n)_{n\in\N}$ denote the first $d$ components of the sequence of estimators obtained from the SGLD algorithm in \eqref{eqn:sgld} defined on the probability space $(\Omega,\mc{F},\mathbb{P})$. Then, there exist explicit constants $a, c_{\delta,\beta}, C_{1,\delta,\ell,\mathfrak{j},\beta}, C_{2,\delta,\ell,\mathfrak{j},\beta}, C_{3,\delta,\beta}, C_4, C_{5,\delta,\beta}, C_6>0$, defined in Appendix \ref{appendix:B}, such that for each $n$, step size $\lambda\in (0,\lambda_{\max,\delta})$, and $\mathfrak{j}\in\N$,
\begin{align}
\E_\mathbb{P}\!\left[u(\hat{\theta}^{\lambda, \delta,\ell,\mathfrak{j}}_n)\right]-\inf_{\theta\in\R^d} u(\theta) 
\leq&\ C_{1,\delta,\ell,\mathfrak{j},\beta} e^{-c_{\delta,\beta}\lambda n/4} + C_{2,\delta,\ell,\mathfrak{j},\beta} \lambda^{1/4} + C_{3,\delta,\beta} \nonumber\\
&\ + \delta m(\ell+\mathfrak{j})\log{2}+\frac{\sqrt{m}}{2^\mathfrak{j}}\left(C_4+C_{5,\delta,\beta}+C_6e^{-a\lambda(n+1)/2}\right).\label{eqn:eerub}
\end{align}
The dependence of the constants on key parameters is summarised in Appendix \ref{appendix:A}.
\end{theorem}

\begin{corollary}
\label{corollary:main_result}
Let Assumptions \ref{assumption:1}, \ref{assumption:3}, and \ref{assumption:4} hold, and let $\varepsilon>0$ be given. Then, 
Algorithm~\ref{Algorithm1} outputs the estimator $\hat{{\theta}}^{\lambda, \delta,\ell,\mathfrak{j}}_n$
 which satisfies
\begin{align*}
	\E_\mathbb{P}\!\left[u(\hat{\theta}^{\lambda, \delta,\ell,\mathfrak{j}}_n)\right]-\inf_{\theta\in\R^d} u(\theta) < \varepsilon.
\end{align*}
\end{corollary}

\begin{proof}
The proof of Theorem~\ref{theorem:excess risk} and Corollary~\ref{corollary:main_result} can be found in Section~\ref{sec:proofOverview}.
\end{proof}
\vspace{-0.1cm}
\begin{algorithm}[!ht]
\caption{SGLD Algorithm for DRO problem \eqref{eqn:dro_problem}} \label{Algorithm1}

\KwInput{$\varepsilon>0$, $d\in\N$, $m\in \N$, $p\in[1,\infty)$, $\eta_1>0$, $\eta_2>0$, compact subset $\Xi\subseteq \R^m$, measurable function $U:\R^d\times \R^m\to \R$, i.i.d. data $(X_\mathfrak{n})_{\mathfrak{n}\in\N_0}\subset\R^m$ defined on $(\Omega, \mc{F},\mathbb{P})$ such that $\mathbb{P}\circ X_0^{-1}=\mu_0$, initialisation $\bar{\theta}_0\in\R^{d+1}$
} 
\KwOutput{Estimator $\hat{{\theta}}^{\lambda, \delta,\ell,\mathfrak{j}}_n$}
Set $M_\Xi:=\max_{x\in\Xi}|x|$\;
Set $L_\nabla,K_\nabla$ to be the constants given by Assumption \ref{assumption:3}\;
Set $\tilde{K}_\nabla:=\max\{K_\nabla,\max_{x\in\Xi}|U(0,x)|\}$\; 
Set $J_U$ to be the constant given by Assumption \ref{assumption:4}\;
Set $\iota:\R\to\R$ and $a_\iota,b_\iota$ to be the function and constants given in Definition~\ref{def:iota} and \eqref{eqn:dissipativity_cond}, respectively\;
Set $a:= \frac{\min\{\eta_1,\eta_2a_\iota\}}{2}$,  $b:=\eta_2b_\iota + \frac{2\left(K_\nabla+2^pM_\iota  M_\Xi^p\right)^2}{\min\{\eta_1,\eta_2 a_\iota\}}$\;
Set $c_{\delta,\beta}, C_{1,\delta,\ell,\mathfrak{j},\beta}, C_{2,\delta,\ell,\mathfrak{j},\beta}, C_{3,\delta,\beta}$ to be the constants given in Theorem \ref{theorem:excess risk}\;
Set $\mathfrak{C}_1$, $\mathfrak{C}_2$, $\mathfrak{C}_3$, $\tilde{L}_\delta$ to be the constants defined in \eqref{const:tilde_l_delta}\;
Set $\mathfrak{C}_4$, $\mathfrak{M}_1$, $\tilde{C}_4$, $C_{5,\delta,\beta}$, $C_6$ to be the constants defined in \eqref{const:c4_c5_c6}\;
Set $C_4:=\eqref{const:c4}$\;
Set $\lambda_{\max,\delta}:=\eqref{defn:step}$\;
Fix $\ell$ such that $\Xi\subset [-2^{\ell-1},2^{\ell-1})^m$ \label{line:1stpara}\;
Fix $\mathfrak{j}>\max\left\{\log_2\left(\frac{5\sqrt{m}(C_4+\mathfrak{C}_4(a^{-1}+2b))}{\varepsilon}\right), \log_2\sqrt{m}\right\}$\;
Fix $\delta \in\left(0, \min\left\{\frac{\varepsilon}{10m(\ell+\mathfrak{j})\log{2}},\frac{\mathfrak{C}_2}{\sqrt{a\mathfrak{C}_1}},\mathfrak{C}_2\sqrt{\frac{\varepsilon 2^\mathfrak{j}}{10\mathfrak{C}_1\mathfrak{C}_4(2\mathfrak{M}_1+1)\sqrt{m}}}\right\}\right)$\;
Fix $\beta > \max\left\{\frac{100(d+1)}{\varepsilon^2}, \frac{10(d+1)\left(1+\log\left(\frac{(\tilde{L}_\delta-1)(b+1)}{a}\right)\right)}{\varepsilon}, \frac{10\sqrt{m}\mathfrak{C}_4(d+1)}{\varepsilon 2^\mathfrak{j}}\right\}$\;
Fix $\lambda \in \left(0,\min\left\{\lambda_{\max,\delta}, \frac{\varepsilon^4}{625C_{2,\delta,\ell,\mathfrak{j},\beta}^4}\right\}\right)$\;
Fix $n>\max\left\{\frac{4}{c_{\delta,\beta}\lambda}
\log\left(\frac{10C_{1,\delta,\ell,\mathfrak{j},\beta}}{\varepsilon}\right),\frac{2}{a\lambda}\log\left(\frac{10C_6}{\varepsilon}\right)-1\right\}$\label{line:6thpara}\;
Set $H^{\delta,\ell,\mathfrak{j}}:=\eqref{defn:grad_v_delta}$\label{line:robustsgld1}\;
\For{$\mathfrak{n}=0,\cdots, n-1$} {
    Draw $Z_{\mathfrak{n}+1}\sim\mc{N}(0,I_{d+1})$\;
    Set $\hat{\bar{\theta}}^{\lambda, \delta, \ell, \mathfrak{j}}_{\mathfrak{n}+1}
:=\ \hat{\bar{\theta}}^{\lambda, \delta, \ell,\mathfrak{j}}_\mathfrak{n}-\lambda H^{\delta, \ell,\mathfrak{j}}(\hat{\bar{\theta}}_\mathfrak{n}^{\lambda, \delta, \ell,\mathfrak{j}},X_{\mathfrak{n}+1})+\sqrt{2\lambda\beta^{-1}}Z_{\mathfrak{n}+1}$\label{line:robustsgld2}\;
}
Set $\hat{{\theta}}^{\lambda, \delta, \ell, \mathfrak{j}}_{n}:=\text{first $d$ components of } \hat{\bar{\theta}}^{\lambda, \delta, \ell, \mathfrak{j}}_{n}$.
\end{algorithm}

\begin{remark}
Theorem~\ref{theorem:excess risk} demonstrates that the upper bound for the expected excess risk in \eqref{eqn:eerub} can be made arbitrarily small as shown in Corollary~\ref{corollary:main_result}, which is achieved by choosing sequentially $\ell, \mathfrak{j}, \delta, \beta, \lambda$, and~$n$ appropriately. Algorithm~\ref{Algorithm1} lines~\ref{line:1stpara}-\ref{line:6thpara} illustrate this parameter selection procedure, which ensures the generation of approximate minimisers that make the associated expected excess risk fall below the prescribed precision level $\varepsilon$ by implementing the robust SGLD algorithm as described in lines \ref{line:robustsgld1}-\ref{line:robustsgld2}.
\end{remark}

%
%
%
%
\section{Application}\label{sec:Application}
As a concrete application of the general framework introduced in the previous section, we consider in this section a non-linear regression model involving a neural network. Our aim is to obtain the best mean-square estimator of the model when the training data is adversarially corrupted. We formulate the problem as a DRO problem and solve it using robust SGLD \eqref{eqn:sgld}-\eqref{defn:grad_v_delta}. We show in Proposition \ref{prop:regression_example} that the DRO problem satisfies our Assumptions~\ref{assumption:1}-\ref{assumption:4}, thus Theorem \ref{theorem:excess risk} provides a theoretical guarantee for the convergence of robust SGLD. Moreover, we illustrate the superior performance of our robust SGLD over vanilla SGLD \cite{welling2011bayesian} by comparing the mean squared loss on the test dataset which only consists of clean data samples drawn from the true distribution. This indicates that the distributionally robust formulation of optimisation problems associated with regression tasks can help mitigate the impact of outliers in the training data, see, e.g., \cite{chen2018robust,shafieezadeh2015distributionally} for more details. Finally, we conclude this section by providing further discussions on the numerical results. The code can be found under \href{https://github.com/tracyyingzhang/robust-SGLD}{https://github.com/tracyyingzhang/robust-SGLD}.\\
 
\paragraph{\textbf{DRO problem}} We consider a regression model given by
\begin{equation}\label{eq:regressionmodel}
y = \mathfrak{N}(\theta^*,z)+\hat{\epsilon},
\end{equation}
with the response variable $y\in\R$, the feature variable $z\in\R^{m-1}$, the regression coefficient $\theta^* \in \R^m$, and the error term $\hat{\epsilon}\in\R$, where $\mathfrak{N}:\R^m\times\R^{m-1}\to\R$ is the neural network given by
\[
\mathfrak{N}(\theta,z) =\sigma_1(\langle w,z\rangle+b_0)= \frac{1}{1+e^{-(\langle w,z\rangle+b_0)}}
\]
with $\theta = (w,b_0) \in \R^m, w\in\R^{m-1}, b_0\in \R$, and $\sigma_1$ being the sigmoid activation function. We aim to obtain an estimator of $\theta^*$ when the training data is adversarially perturbed and consists of samples drawn from some outlying distribution. To this end, we consider the DRO problem \eqref{eqn:dro_problem} where $U$ is given by
\begin{equation}\label{eq:DROexample_U}
U(\theta,x) = |y-\mathfrak{N}(\theta,z)|^2
\end{equation}
with $\theta\in\R^m$, $x=(z,y)\in\R^m$.\footnote{In this example, the dimension of the parameter of DRO problem \eqref{eqn:dro_problem} coincides with the dimension of the data points, i.e., $d=m$ in the notation of \eqref{eqn:dro_problem}.}
\begin{proposition}
\label{prop:regression_example}
Let Assumption~\ref{assumption:1} hold. Then, the function $U$ defined in \eqref{eq:DROexample_U} satisfies Assumptions \ref{assumption:3} and \ref{assumption:4}.
\end{proposition}
\begin{proof}
See Section~\ref{sec:proof_sec:Application}. 
\end{proof}
By Proposition \ref{prop:regression_example}, we can use robust SGLD \eqref{eqn:sgld}-\eqref{defn:grad_v_delta} to solve the DRO problem under consideration, and Theorem \ref{theorem:excess risk} provides a theoretical guarantee for the convergence of our robust SGLD algorithm.\\

\begin{figure}[t]
\centering
\includegraphics[width=0.81\textwidth]{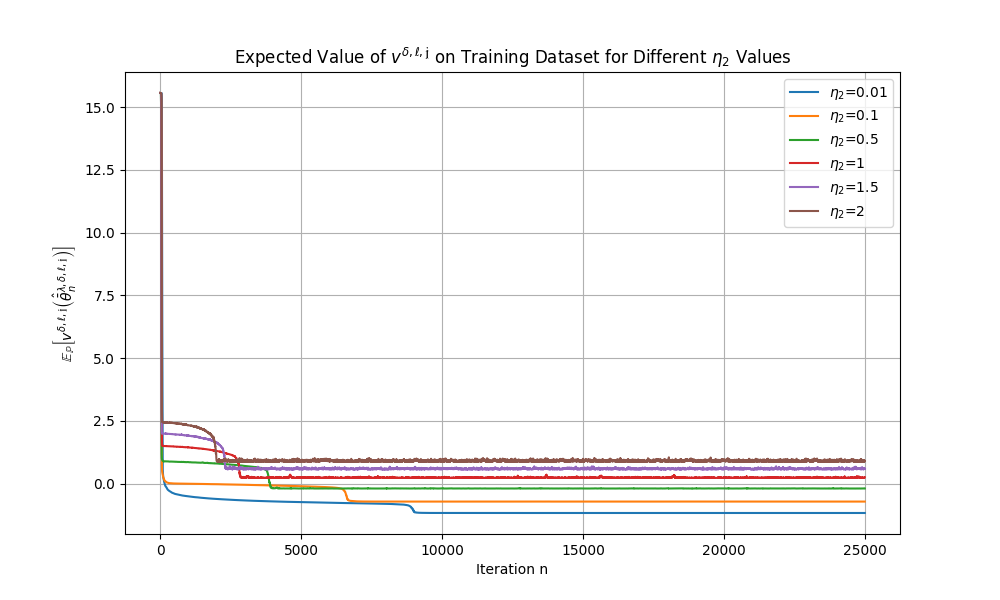}
\caption{Path of robust SGLD for different values of $\eta_2$}
\label{fig:robustsgldobjpath}
\end{figure}
\paragraph{\textbf{Simulation result}} 
Set $m=4$, $\theta^* = (w^*,b_0^*) = (-0.5, 0.5, 0.1, -0.2)$, and $q=0.3$. We consider a training set with $100q\%$ of the samples drawn from some outlying distribution and $100(1-q)\%$ of the samples drawn from a given distribution which we refer here as the ``true distribution''. The training data generation process is similar to that described in \cite{chen2018robust}:
\begin{enumerate}
\item First, we generate a dataset $\{x_i\}_{i=1}^{10000}= \{(y_i,z_i)\}_{i=1}^{10000}$ consisting of 10000 samples drawn from the true distribution. More precisely, for each $i$, $z_i\in\R^3$ is drawn from the uniform distribution on $[-1,1]^3$ and $y_i$ is computed by $y_i = \mathfrak{N}(\theta^*,z_i)+0.1\tilde{\epsilon}$ where $\tilde{\epsilon}\sim \text{Bernoulli}(1/2)$. \label{item:datagenerationstep1}
\item Then, for each $i$, we draw a value from $\text{uniform}[0,1]$. If it is less than $1-q$, we keep $x_i$ generated in step \ref{item:datagenerationstep1}. Otherwise, replace it with $\bar{x}_i = (\bar{y}_i,\bar{z}_i)$ where $\bar{z}_i$ is drawn from a uniform distribution on $[2,2.5]^3$ and $\bar{y}_i = y_i+\tilde{\epsilon}$. \label{item:datagenerationstep2}
\end{enumerate}
The test dataset $\{x_i^{\text{test}}\}_{i=1}^{5000}= \{(y_i^{\text{test}},z_i^{\text{test}})\}_{i=1}^{5000}$ consisting of 5000 samples is generated using the same method as described in step \ref{item:datagenerationstep1}. We note that all the samples included in the test set are drawn from the true distribution. Moreover, all generated training and test samples satisfy Assumption~\ref{assumption:1} with $\Xi = [-3,3]^4$ and $M_\Xi=6$. \\

We use robust SGLD \eqref{eqn:sgld}-\eqref{defn:grad_v_delta} to solve DRO problem \eqref{eqn:dro_problem} with $U$ defined in \eqref{eq:DROexample_U} using the training dataset. More precisely, by using the framework described in Section \ref{sec:proofOverview}, we apply robust SGLD \eqref{eqn:sgld}-\eqref{defn:grad_v_delta} to solve $z_{D,\ell,\mathfrak{j},\delta}$ defined in \eqref{eqn:dro_problem_dual_discrete_smoothed}, which can be viewed as the smoothed and discretised version of the dual of the aforementioned DRO problem \eqref{eqn:dro_problem} (see also Table \ref{tbl:summaryofdef}). Figure \ref{fig:robustsgldobjpath} depicts the path of $\mathbb{E}_\mathbb{P}\left[v^{\delta,\ell,\mathfrak{j}}\left(\hat{\bar{\theta}}^{\lambda, \delta,\ell,\mathfrak{j}}_n\right)\right]$ for different values of $\eta_2$, where $v^{\delta,\ell,\mathfrak{j}}$ in \eqref{defn:small_v_delta_discrete} is the objective function of $z_{D,\ell,\mathfrak{j},\delta}$ and $\hat{\bar{\theta}}^{\lambda, \delta,\ell,\mathfrak{j}}_n$ denotes the $n$-th iteration of our robust SGLD algorithm \eqref{eqn:sgld}-\eqref{defn:grad_v_delta}. We note that the paths in Figure \ref{fig:robustsgldobjpath} stabilise, supporting the result in Proposition \ref{prop:excess_risk_discrete}, hence Theorem \ref{theorem:excess risk}.\\

\begin{table}[t]
\begin{tabularx}{0.99\textwidth}
{ 
	>{\hsize=0.4\hsize\linewidth=\hsize}X|
	>{\hsize=0.8\hsize\linewidth=\hsize}X
	>{\hsize=1.8\hsize\linewidth=\hsize}X
} 
  Parameter      & Value                                                 & Interpretation                                                                                                                                                                                                                                                                           
 \\
   \hline
  $m$            & $4$                                                   & Dimensionality of data points and of the parameters of the DRO problem \eqref{eqn:dro_problem}.\\
    $\mu_0$        & Empirical measure of training data points & Reference probability measure for distribution of training data points \\
  $p$            & $2$                                                   & Controls convexity of cost of transportation between true distribution of data points and given reference distribution.                                                                                                                                                                 \\
  $\eta_1$       & $10^{-3}$                                                 &  Controls regularisation constant in DRO problem. Smaller values of $\eta_1$ impose less regularisation.\\
  $\eta_2$       & Various                                                       &  Controls penalty imposed on the distance between any distribution of data points and given reference distribution. Larger values of $\eta_2$ impose smaller penalty.\\
  $\bar{\theta}_0$     & $(-2,-2,-2,-2,0)$                                                   & Initial condition $\bar{\theta}_0 = (\theta_0,\alpha_0)$ of robust SGLD.                                                                                                                                                                                                                                                            \\
  $n$            & $25000$                                                & Number of algorithm iterations.                                                                                                                                                                                                                                                            \\
  $\lambda$      & $0.01$                                                & Step size of algorithm in time space.                                                                                                                                                                                                                                                      \\
  $\beta$        & $10^9$                                                  &  ``Mixing parameter'' controlling amount of stochasticity in each algorithm iteration. Larger values of $\beta$ generate less randomness at each iteration.\\
  $\delta$       & $0.1$                                                 & Nesterov's smoothing tolerance.                                                                                                                                                                                                                                                           \\  
  $\Xi$          & $[-3, 3]^4$                               & Support of data points in the training and test set.                                                                                                                                                                                                                                                                  \\
  $\ell$         & $3$                                                   & Controls intersection between support of data points traversed in the algorithm and actual support of data points. Must be large enough relative to $\Xi$ to cover entire support. \\
  $\mathfrak{j}$ & $1$                                                   & Controls discretisation of support of data points as a finite grid. Larger values of $\mathfrak{j}$ yield finer meshes of order $\mathcal{O}(2^{-\mathfrak{j}})$.\\
 \caption{Parameter values used in the numerical simulations}\label{tbl:params}
\end{tabularx}
\end{table}

To demonstrate the efficacy of distributionally robust formulation in mitigating the effect of outliers, we compare the performance of robust SGLD against vanilla SGLD\footnote{The vanilla SGLD estimator of $\theta^*$ is obtained by applying SGLD \cite{welling2011bayesian} to the non-robust stochastic optimisation problem:
\[
	\text{minimise}\quad  \R^d\ni \theta\mapsto u(\theta):=\E \left[|Y-\mathfrak{N}(\theta,Z)|^2\right],
\]
where $X = (Y,Z)$ with $Z, Y$ being the $\R^3$-valued input variable and $\R$-valued response variable, respectively. 
} using the mean squared loss computed over the test dataset $\{x_i^{\text{test}}\}_{i=1}^{5000}$. We run SGLD and robust SGLD with different values of $\eta_2$ for $100$ times. For each run, we set the number of iterations to be $25000$ with other hyperparameters specified in Table \ref{tbl:params}. We then compute average training times (in seconds) over $100$ runs for each algorithm. Moreover, for each run, we record the number of iterations required for each algorithm to first reach a value within $1\%$ of the reference value, and then pick the largest one, denoted by $n_{\text{es}}$, over 100 runs. Here, the reference value is the mean squared loss computed using the optimal parameter $\theta^*$ on the test dataset. If $n_{\text{es}}<25000$ for an algorithm, we report the corresponding running time up to $n_{\text{es}}$ and the value of mean squared loss at $n_{\text{es}}$. Otherwise, if an algorithm never reaches a value within $1\%$ of the reference value after 25000 iterations, we report ``NA'' for $n_{\text{es}}$ and for its corresponding running time, moreover, we report the value that is closest to the reference value (among 25000 iterations) as its corresponding mean squared loss. We summarise these values in Table \ref{tbl:mseloss} and draw the path of mean squared loss for SGLD and robust SGLD in Figure \ref{fig:mseloss}. The numerical results indicate the superior performance of robust SGLD over vanilla SGLD for large values of $\eta_2$. This coincides with the fact that larger values of $\eta_2$ impose smaller penalty in view of the formulation of DRO problem \eqref{eqn:dro_problem}, hence allowing optimisation under distributions that deviate from the reference (or empirical) measure, thus reducing the impact of outlying data points.\\

\begin{figure}[t]
\begin{subfigure}{\textwidth}\centering
    \includegraphics[width=0.81\textwidth]{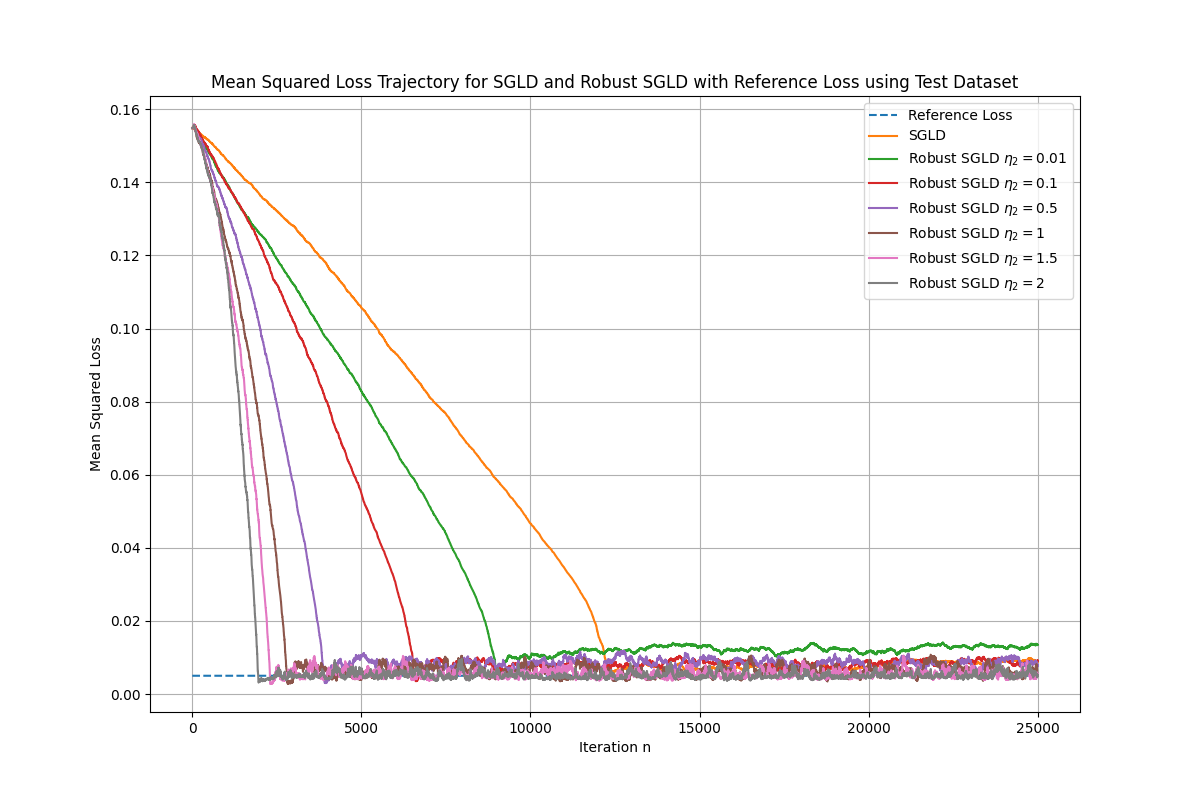}
\end{subfigure}
\begin{subfigure}{\textwidth}\centering
    \includegraphics[width=0.81\textwidth]{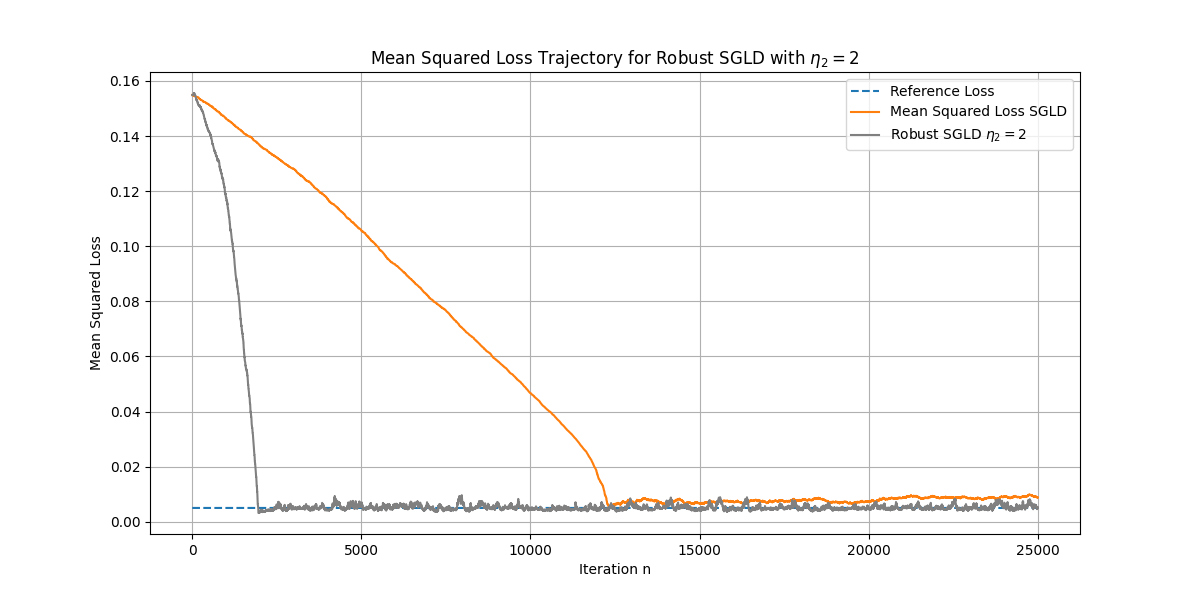}
\end{subfigure}
\caption{Mean squared loss for vanilla SGLD and robust SGLD on test dataset}\label{fig:mseloss}
\end{figure}

\begin{table}[t]
\centering
\begin{tabularx}{\textwidth}{>{\hsize=1.45\hsize}X|
>{\centering\arraybackslash\hsize=1\hsize}X|
*{3}{>{\centering\arraybackslash\hsize=0.85\hsize}X}
}
Method &Average Training Time (s) &Iterations to $1\%$ Precision $n_{\text{es}}$& {Time to $1\%$ Precision (s)} & {Mean Squared Loss} \\
\hline
Robust SGLD ($\eta_2=0.01$) 	& 122.73 	& NA	& NA 	& 0.005441  \\
\hline
Robust SGLD ($\eta_2=0.1$)  	& 116.54 	& NA	& NA 	& 0.004914  \\
\hline
Robust SGLD ($\eta_2=0.5$) 	& 114.80 	& 20353	& 93.63  	& 0.005000  \\
\hline
Robust SGLD ($\eta_2=1.0$)	& 114.28 	& 4044	& 18.54 	& 0.005010  \\
\hline
Robust SGLD ($\eta_2=1.5$)  	& 113.99	& 2795	& 12.99  	& 0.004979  \\
\hline
Robust SGLD ($\eta_2=2.0$)  	& 113.85 	& 2427	& 10.93  	& 0.005031  \\
\hline
Vanilla SGLD             		& 0.45 	& NA	& NA 	& 0.005200  \\
\hline
Reference          			& --- 	& ---	& --- 	& 0.005014 \\
\caption{Average training times for SGLD and robust SGLD to complete 25000 iterations over 100 runs, together with the numbers of iterations and running times required to reach within $1\%$ of the reference value. Mean squared losses are the values at $n_{\text{es}}$ or the closest value to reference computed using test dataset. $1\%$ difference from reference is $[0.004964,0.005064]$.}
\label{tbl:mseloss}
\end{tabularx}
\end{table}

\paragraph{\textbf{Conclusions}} In this section, we consider the problem of identifying the best non-linear mean-square estimator associated with the regression model \eqref{eq:regressionmodel} when the training data is corrupted. We formulate the problem as a DRO problem using the framework described in Section \ref{sec:AssumpMainRes} and solve it using robust SGLD \eqref{eqn:sgld}-\eqref{defn:grad_v_delta}. We demonstrate both theoretically using Proposition \ref{prop:regression_example} (hence Theorem \ref{theorem:excess risk}) and numerically through Figure \ref{fig:robustsgldobjpath} that robust SGLD can be used to solve the DRO problem under consideration. Furthermore, 
we compare the performance of our robust SGLD against vanilla SGLD. As indicated in Figure \ref{fig:mseloss} and Table \ref{tbl:mseloss}, robust SGLD outperforms vanilla SGLD in terms of test accuracy (measured using the mean squared loss), although at the cost of a slower training speed.

\section{Proof Overview of Main Results}\label{sec:proofOverview}
In this section, we present an overview of the proof for obtaining the non-asymptotic upper bound on the excess risk of our proposed robust SGLD algorithm stated in Theorem \ref{theorem:excess risk}. The proof comprises three main steps. First, we make use of a duality result in \cite{bartl2017computational} to express the distributionally robust optimisation problem of \eqref{eqn:dro_problem} in its dual form. After which, we reduce the problem from the compact support $\Xi$ of the observed data $X$ to the finite grid $\Xi\cap \mathbb{K}_{\ell,\mathfrak{j}}^m$. Finally, we obtain the convergence bound of the excess risk of the robust SGLD algorithm defined on this finite grid through Nesterov's smoothing technique and the duality result of Theorem \ref{theorem:wass_duality}.

\subsection{Dual Problem Formulation} \label{subsec:dual_problem_formulation}
Recall that our main problem of interest was defined in \eqref{eqn:dro_problem} and can be stated as
\begin{align}
z_P :=&\ \inf_{\theta\in\R^d} u(\theta)\nonumber\\
=&\ \inf_{\theta\in\R^d}\left\{\sup_{\mu\in \mc{P}(\Xi)}\left(\int_{\Xi} U(\theta, x)\ \dee\mu(x)-\frac{d_c^2(\mu_0,\mu)}{2\eta_2}\right)+\frac{\eta_1}{2}|\theta|^{2}\right\}.\label{defn:zP}
\end{align}

The first step of our proof involves expressing the optimisation problem of \eqref{eqn:dro_problem} in dual form. To this end, we make use of the following duality result which is an immediate consequence of\footnote{We also refer to   \cite{blanchet2019quantifying}, \cite{gao2023distributionally}, \cite{mohajerin2018data}, and \cite{zhao2018data} for similar duality results.} Theorem 2.4 of \cite{bartl2017computational}.

\begin{theorem}[\cite{bartl2017computational}]
\label{theorem:wass_duality}
Let Assumption~\ref{assumption:1} hold and let $c:\Xi\times\Xi\to\R_{\geq 0}$ and $\varphi:\R_{\geq 0}\to\R_{\geq 0}$ be cost and penalty functions in the sense of \cite{bartl2017computational}, respectively. Moreover, let  $\mu_0\in\mc{P}(\Xi)$ and let $U:\R^d\times \Xi\to\R$ be a measurable function such that for every $\theta\in\R^d$, $x\mapsto U(\theta,x)$ is in $L^1(\mu_0)$ and bounded from below. Then the following duality result holds: for every $\theta\in\R^d$:
\begin{align}
\begin{split}\label{eqn:duality_theorem}
&\ \sup_{\mu\in\mc{P}(\Xi)}\left\{\int_{\Xi} U(\theta,x)\ \dee \mu(x)-\varphi(d_c(\mu_0,\mu))\right\} \\
=&\ \inf_{\mathfrak{a}\geq 0}\left\{\varphi^*(\mathfrak{a})+\int_{\Xi}\sup_{y\in\Xi}\left\{U(\theta,y)-\mathfrak{a} c(x,y)\right\}\ \dee\mu_0(x)\right\},
\end{split}
\end{align}
where $\varphi^*$ denotes the convex conjugate of $\varphi$.
\end{theorem}

For our problem of interest as stated in \eqref{eqn:dro_problem}, the choices of cost and penalty functions are $c(x,x'):=|x-x'|^p$ and $\varphi(x):=\frac{x^2}{2\eta_2}$, respectively, where $p\in [1,\infty)$ and $\eta_2>0$. By applying the duality result of Theorem \ref{theorem:wass_duality}, we obtain the equivalent dual formulation of the problem \eqref{eqn:dro_problem} as
\begin{align}
\label{eqn:dro_problem_dual}
z_D
:=&\ \inf_{\theta\in\R^d}\inf_{\mathfrak{a}\geq 0}\left\{\int_{\Xi}\sup_{y\in\Xi}\left\{U(\theta,y)-\mathfrak{a} |x-y|^p\right\}\ \dee\mu_0(x)+\frac{\eta_1}{2}|\theta|^{2}+\frac{\eta_2}{2}|\mathfrak{a}|^2\right\},
\end{align}
such that strong duality $z_P=z_D$ holds. The purpose of obtaining this dual form of the problem is that, after applying the transformation $\mathfrak{a} = \iota(\alpha)$, where $\iota:\R\to\R_{\geq 0}$ is a function given in Definition \ref{def:iota}, the dual problem can be expressed in the form of a standard, i.e. non-distributionally robust, stochastic optimisation problem to which the SGLD algorithm of \cite{zhang2023nonasymptotic} can be directly applied.\\

To see this clearly, we define for every $\bar{\theta}:=(\theta,\alpha)\in\R^d\times \R$ and $x \in \Xi$
\begin{align}\label{defn:small_v}
\begin{split}
v(\bar{\theta}) := &\ \int_\Xi \tilde{V}(\bar{\theta}, x)\ \dee \mu_0(x),\\
\tilde{V}(\bar{\theta}, x) :=&\ \sup_{y\in\Xi}\left\{U(\theta,y)-\iota(\alpha) |x-y|^p\right\}+\frac{\eta_1}{2}|\theta|^{2}+\frac{\eta_2}{2}|\iota(\alpha)|^2,
\end{split}
\end{align}
such that we have
\begin{align}
z_D = \inf_{\bar{\theta}\in\R^{d+1}}v(\bar{\theta}) =z_P\label{eqn:dro_problem_dual_transformed}
\end{align}
by the surjectivity of $\iota$. The optimisation problem \eqref{eqn:dro_problem_dual_transformed} is a standard stochastic optimisation problem over the whole domain $\R^{d+1}$ in $\bar{\theta}:=(\theta,\alpha)$ to which the SGLD algorithm of \cite{zhang2023nonasymptotic} can be applied.

\subsection{Reduction to a Finite Grid} \label{subsec:finite_grid} 
The dual problem $z_D$ as stated in \eqref{eqn:dro_problem_dual} involves an observed data variable $X$ which has compact support $\Xi$ in $\R^m$. The next step of the proof involves reducing the dual problem $z_D$ to a discretised version $z_{D,\ell,\mathfrak{j}}$, to be formulated subsequently, where the observed data has finite support in $\R^m$. The quadrature error $|z_D-z_{D,\ell,\mathfrak{j}}|$ is then controlled. To this end, we recall, given positive integers $\ell,\mathfrak{j}>0$, the definition of the set of dyadic rationals
\begin{align*}
\K_{\ell,\mathfrak{j}}:= \left\{-2^{\ell-1},-2^{\ell-1}+\frac{1}{2^\mathfrak{j}},\cdots,2^{\ell-1}-\frac{1}{2^\mathfrak{j}} \right\}
\end{align*}
stated in \eqref{defn:dyadic}, and that we have previously fixed a $\mathfrak{j}\in\N$ and an $\ell\in\N$ large enough such that $\Xi\subseteq [-2^{\ell-1},2^{\ell-1})^m$. Note that the finite grid $\Xi\cap \K_{\ell,\mathfrak{j}}^m$ is the set on which the robust SGLD algorithm \eqref{eqn:sgld} is defined. In addition, we define, for each $\boldsymbol{i}=(i_1,\cdots, i_m)\in\K_{\ell,\mathfrak{j}}^m$, the set
\begin{align}\label{eqn:qij}
Q_{\boldsymbol{i},\mathfrak{j}}:=\left[i_1,i_1+\frac{1}{2^\mathfrak{j}}\right)\times \left[i_2,i_2+\frac{1}{2^\mathfrak{j}}\right)\times\cdots\times \left[i_m,i_m+\frac{1}{2^\mathfrak{j}}\right)
\end{align}
such that
\begin{align*}
\bigcupdot_{\boldsymbol{i}\in\K_{\ell,\mathfrak{j}}^m}Q_{\boldsymbol{i},\mathfrak{j}}=[-2^{\ell-1},2^{\ell-1})^m\supseteq \Xi.
\end{align*}

The reference probability measure $\mu_0\in\mc{P}(\Xi)$ can then be extended to a probability measure $\mu_{0,\ell}\in\mc{P}([-2^{\ell-1},2^{\ell-1})^m)$, defined by 
\begin{align*}
\mu_{0,\ell}(B):=\mu_0(B\cap \Xi),\qquad B\in\mc{B}([-2^{\ell-1},2^{\ell-1})^m).
\end{align*}

Then, by applying a quadrature procedure, we can discretise $\mu_{0,\ell}$ to the finite grid $\K_{\ell,\mathfrak{j}}^m$, that is, we define the discrete probability measure $\mu_{0,\ell,\mathfrak{j}}\in\mc{P}(\K_{\ell,\mathfrak{j}}^m)$ by
\begin{align}\label{eqn:dprobmeasure}
\mu_{0,\ell,\mathfrak{j}}(\{\boldsymbol{i}\}) := \mu_{0,\ell}(Q_{\boldsymbol{i},\mathfrak{j}}),\qquad \boldsymbol{i}\in \K_{\ell,\mathfrak{j}}^m.
\end{align}
By defining the function $[\cdot]_\mathfrak{j}:\R^m \to (2^{-\mathfrak{j}}\mathbb{Z})^m$ by
\begin{align}\label{eqn:gridfun}
(x_1,\cdots,x_m)=:x\mapsto [x]_\mathfrak{j} = \left(\frac{\lfloor 2^\mathfrak{j} x_1\rfloor}{2^\mathfrak{j}},\cdots,\frac{\lfloor 2^\mathfrak{j} x_m\rfloor}{2^\mathfrak{j}}\right),
\end{align}
one may specify the discretised version of the primal problem \eqref{defn:zP} as
\begin{align}
\text{minimise } \R^d\ni\theta \mapsto  u^{\ell,\mathfrak{j}}(\theta)\nonumber
\end{align}
with
\begin{align}
 u^{\ell,\mathfrak{j}}(\theta) := \sup_{\mu\in \mc{P}(\Xi\cap\mathbb{K}_{\ell,\mathfrak{j}}^m)}\left(\int_{\Xi\cap\mathbb{K}_{\ell,\mathfrak{j}}^m} U(\theta, x)\ \dee\mu(x)-\frac{1}{2\eta_2}d_c^2(\mu_{0,\ell,\mathfrak{j}},\mu)\right)+\frac{\eta_1}{2}|\theta|^{2},\label{defn:u_discrete}
\end{align}
and
\begin{align}
z_{P,\ell,\mathfrak{j}}:=&\ \inf_{\theta\in\R^d}u^{\ell,\mathfrak{j}}(\theta)\nonumber\\
=&\ \inf_{\theta\in\R^d}\sup_{\mu\in \mc{P}(\Xi\cap\mathbb{K}_{\ell,\mathfrak{j}}^m)}\left(\int_{\Xi\cap\mathbb{K}_{\ell,\mathfrak{j}}^m} U(\theta, x)\ \dee\mu(x)-\frac{1}{2\eta_2}d_c^2(\mu_{0,\ell,\mathfrak{j}},\mu)\right)+\frac{\eta_1}{2}|\theta|^{2}.\label{defn:zPlj}
\end{align}

We also define the discretised version of the dual problem \eqref{eqn:dro_problem_dual} as
\begin{align}
z_{D,\ell,\mathfrak{j}}
:=&\ \inf_{\theta\in\R^d}\inf_{\mathfrak{a}\geq 0}\left\{\int_{\Xi}\sup_{y\in\Xi}\left\{U(\theta,[y]_\mathfrak{j})-\mathfrak{a} |[x]_\mathfrak{j}-[y]_\mathfrak{j}|^p\right\}\ \dee\mu_0(x)+\frac{\eta_1}{2}|\theta|^{2}+\frac{\eta_2}{2}|\mathfrak{a}|^2\right\}.\label{eqn:dro_problem_discrete}
\end{align}
The following lemma enables us to explicitly represent $z_{D,\ell,\mathfrak{j}}$ as an optimisation problem that lives on a discrete probability space.
\begin{lemma}
\label{lemma:finite_grid}
The discretised version $z_{D,\ell,\mathfrak{j}}$ of the dual problem $z_D$ in \eqref{eqn:dro_problem_dual}, given by \eqref{eqn:dro_problem_discrete}, has the equivalent representation
\begin{align}
z_{D,\ell,\mathfrak{j}}=\inf_{\theta\in\R^d}\inf_{\mathfrak{a}\geq 0}\left\{\int_{\K_{\ell,\mathfrak{j}}^m}\max_{y\in\Xi\cap\K_{\ell,\mathfrak{j}}^m}\left\{U(\theta,y)-\mathfrak{a} |x-y|^p\right\}\ \dee\mu_{0,\ell,\mathfrak{j}}(x)+\frac{\eta_1}{2}|\theta|^{2}+\frac{\eta_2}{2}|\mathfrak{a}|^2\right\}.\label{eqn:dro_problem_discrete_alt}
\end{align}
\end{lemma}
\begin{proof}
See Section~\ref{sec:proof_sec:proofOverview}. 
\end{proof}

As a discrete analogue of \eqref{defn:small_v}, we define, for any $\bar{\theta}:=(\theta,\alpha)\in\R^d\times \R$ and any $x\in\Xi$, the quantities
\begin{align}
\begin{split}\label{defn:small_v_discrete}
v^{\ell,\mathfrak{j}}(\bar{\theta}) := &\ \int_{\Xi\cap\mathbb{K}_{\ell,\mathfrak{j}}^m} \tilde{V}^{\ell,\mathfrak{j}}(\bar{\theta}, x)\ \dee \mu_{0,\ell,\mathfrak{j}}(x),\\
\tilde{V}^{\ell,\mathfrak{j}}(\bar{\theta}, x) :=&\ \max_{y\in\Xi\cap\mathbb{K}_{\ell,\mathfrak{j}}^m}\left\{U(\theta,y)-\iota(\alpha) |x-y|^p\right\}+\frac{\eta_1}{2}|\theta|^{2}+\frac{\eta_2}{2}|\iota(\alpha)|^2.
\end{split}
\end{align}

Then, by the duality result of Theorem \ref{theorem:wass_duality}, we obtain for every $\theta \in \R^d$ that
\begin{equation}\label{eq:discrete-duality}
u^{\ell,\mathfrak{j}}(\theta)
=\inf_{\alpha \in \R}\bigg(\int_{\Xi\cap\mathbb{K}_{\ell,\mathfrak{j}}^m} \tilde{V}^{\ell,\mathfrak{j}}((\theta,\alpha), x)\ \dee \mu_{0,\ell,\mathfrak{j}}(x)\bigg).
\end{equation}	

This and the representation of $z_{D,\ell,\mathfrak{j}}$ given by \eqref{eqn:dro_problem_discrete_alt} in Lemma \ref{lemma:finite_grid} 
hence imply the relation
\begin{align}
z_{D,\ell,\mathfrak{j}} 
=&\ \inf_{\bar{\theta}\in\R^{d+1}}v^{\ell,\mathfrak{j}}(\bar{\theta})=z_{P,\ell,\mathfrak{j}}.\label{eqn:dro_problem_dual_transformed_discrete}
\end{align}
 Note that \eqref{eqn:dro_problem_dual_transformed_discrete} is a discrete analogue of \eqref{eqn:dro_problem_dual_transformed}.\\

Moreover, the compactness of the support $\Xi$ enables us to reduce the computation of $z_D$ as well as $z_{D,\ell,\mathfrak{j}}$ to optimisation problems over compact subsets of $\R^{d+1}$ which do not depend on $\mathfrak{j}$. This is precisely stated in the next lemma and is paramount to obtaining the upper bound for the quadrature error $|z_D-z_{D,\ell,\mathfrak{j}}|$.

\begin{lemma}
\label{lemma:compactness}
Let Assumptions \ref{assumption:1} and \ref{assumption:3} hold. Define $B_{R_\mc{K}}(0):=\{(\theta,\mathfrak{a})\in\R^d\times [0,\infty):|(\theta,\mathfrak{a})|\leq R_\mc{K} \}$ with
\begin{equation}\label{eq:rk_expression}
	R_\mc{K} := \left(\left(3\tilde{K}_\nabla + \frac{(\tilde{K}_\nabla+2^pM_\Xi^p)^2}{\min\{\eta_1, \eta_2\}}\right)\frac{4}{\min\{\eta_1, \eta_2\}}\right)^{1/2}.
\end{equation}
Then, the following holds: 
\begin{align*}
z_D
:=&\ \inf_{\theta\in\R^d}\inf_{\mathfrak{a}\geq 0}\left\{\int_{\Xi}\sup_{y\in\Xi}\left\{U(\theta,y)-\mathfrak{a} |x-y|^p\right\}\ \dee\mu_0(x)+\frac{\eta_1}{2}|\theta|^{2}+\frac{\eta_2}{2}|\mathfrak{a}|^2\right\}\nonumber\\
=&\ \inf_{(\theta,\mathfrak{a})\in B_{R_\mc{K}}(0)}\left\{\int_{\Xi}\sup_{y\in\Xi}\left\{U(\theta,y)-\mathfrak{a} |x-y|^p\right\}\ \dee\mu_0(x)+\frac{\eta_1}{2}|\theta|^{2}+\frac{\eta_2}{2}|\mathfrak{a}|^2\right\},
\end{align*}
as well as
\begin{align*}
z_{D,\ell,\mathfrak{j}}
:=&\ \inf_{\theta\in\R^d}\inf_{\mathfrak{a}\geq 0}\left\{\int_{\Xi}\sup_{y\in\Xi}\left\{U(\theta,[y]_\mathfrak{j})-\mathfrak{a} |[x]_\mathfrak{j}-[y]_\mathfrak{j}|^p\right\}\ \dee\mu_0(x)+\frac{\eta_1}{2}|\theta|^{2}+\frac{\eta_2}{2}|\mathfrak{a}|^2\right\}\nonumber\\
=&\ \inf_{(\theta,\mathfrak{a})\in B_{R_\mc{K}}(0)}\left\{\int_{\Xi}\sup_{y\in\Xi}\left\{U(\theta,[y]_\mathfrak{j})-\mathfrak{a} |[x]_\mathfrak{j}-[y]_\mathfrak{j}|^p\right\}\ \dee\mu_0(x)+\frac{\eta_1}{2}|\theta|^{2}+\frac{\eta_2}{2}|\mathfrak{a}|^2\right\}.
\end{align*}
\end{lemma}
\begin{proof}
See Section~\ref{sec:proof_sec:proofOverview}. 
\end{proof}

Finally, we state in the following proposition an upper bound on the quadrature error. The implication is that, by varying $\mathfrak{j}$ to control the mesh $\frac{1}{2^\mathfrak{j}}$ of the grid, one can ensure the discretised dual problem $z_{D,\ell,\mathfrak{j}}$ to be as close to the original dual problem $z_D$ as desired.

\begin{proposition}
\label{prop:quadrature} Let Assumptions \ref{assumption:1}, \ref{assumption:3}, and \ref{assumption:4} hold. Moreover, let $R_\mc{K}$ be defined in \eqref{eq:rk_expression}. Then, for any $\ell\in\N$ such that $\Xi\subset [-2^{\ell-1},2^{\ell-1})^m$ and for any given $\mathfrak{j}\in\N$, the following bound for the quadrature error $|z_D-z_{D,\ell,\mathfrak{j}}|$ holds:
\begin{align}
|z_D-z_{D,\ell,\mathfrak{j}}| \leq&\ \frac{\sqrt{m}(J_U+ 2p(1+4M_\Xi)^{p-1})(1+R_\mc{K} )}{2^\mathfrak{j}}.\label{eqn:quadrature_error}
\end{align}
\end{proposition}
\begin{proof}
See Section~\ref{sec:proof_sec:proofOverview}.
\end{proof}
\subsection{Nesterov's Smoothing Technique}\label{subsec:nesterov} 
Having obtained a non-asymptotic upper bound on the quadrature error $|z_D-z_{D,\ell,\mathfrak{j}}|$, the second step of the proof is to obtain a non-asymptotic upper bound on the expected excess risk of the algorithm over the optimal value of the discretised version of the dual problem -- that is the quantity $\E_\mathbb{P}\left[u(\hat{\theta}^{\lambda, \delta,\ell,\mathfrak{j}}_n)\right]-z_{D,\ell,\mathfrak{j}}$. To this end, we make use of the following result which can be obtained by applying Nesterov's smoothing technique to the maximum function, see, for example, Lemma 5 of \cite{an2022efficient}.

\begin{lemma}[\cite{an2022efficient}]
\label{lemma:smooth_max}
Let $N\in\N$. Then, for any $\delta>0$, the following smooth approximation of the maximum function
\begin{align*}
\R^N\ni (x_1,\cdots,x_N)\mapsto \phi_{\delta}(x_1,\cdots,x_N) := \delta \log\left(\frac{1}{N}\sum^N_{j=1}e^{x_j/\delta}\right)
\end{align*}
satisfies, for any $(x_1,\cdots,x_N)\in \R^N$, the inequalities
\begin{align*}
\phi_{\delta}(x_1,\cdots,x_N)\leq \max\{x_j: j=1,\cdots,N\}\leq \phi_{\delta}(x_1,\cdots,x_N) + \delta\log{N}.
\end{align*}
\end{lemma}

Recall that we have fixed $\ell\in\N$ large enough such that $\Xi\subset [-2^{\ell-1},2^{\ell-1})^m$, and that we have also previously denoted
\begin{align*}
\{\xi_j^{\ell,\mathfrak{j}}\}_{j=1,\cdots, N_{\ell,\mathfrak{j}}}:=&\ \Xi\cap \K_{\ell,\mathfrak{j}}^m,\nonumber\\
N_{\ell,\mathfrak{j}}:=&\ 2^{m(\ell+\mathfrak{j})}.
\end{align*}
Fixing also $\mathfrak{j}\in\N$, we denote $\xi_j:= \xi_j^{\ell, \mathfrak{j}}$ and $N:=N_{\ell,\mathfrak{j}}$ thereby suppressing dependence of the quantities $\xi_j$ and $N$ on $\ell$ and $\mathfrak{j}$ for the sake of brevity.\\

%

An application of Lemma \ref{lemma:smooth_max} to the representation of $z_{D,\ell,\mathfrak{j}}$ given in Lemma \ref{lemma:finite_grid}, along with the surjectivity of $\iota:\R\to\R_{\geq 0}$, then yields the following result.
\begin{corollary}\label{co:smoothing}
Let Assumptions \ref{assumption:1} and \ref{assumption:3} hold. For every $\delta>0$, define $V^{\delta,\ell,\mathfrak{j}}:\R^{d+1}\times \Xi\to\R$ as
\begin{align}
V^{\delta,\ell,\mathfrak{j}}(\bar{\theta},x)
:=&\ \delta\log\left(\frac{1}{N}\sum^N_{j=1}\exp\left[\frac{1}{\delta}\left(U(\theta,\xi_j)-\iota(\alpha)|x-\xi_j|^p\right)\right]\right),
\label{defn:v_delta}
\end{align}
where $\bar{\theta}=(\theta,\alpha)\in\R^{d}\times\R$ and $x\in\Xi$. Moreover, define for every $\delta>0$
\begin{align}
	\tilde{V}^{\delta,\ell,\mathfrak{j}}(\bar{\theta},x) := {V}^{\delta,\ell,\mathfrak{j}}(\bar{\theta},x) + \frac{\eta_1}{2}|\theta|^{2}+\frac{\eta_2}{2}|\iota(\alpha)|^2,
	\qquad
	v^{\delta,\ell,\mathfrak{j}}(\bar{\theta})
	:= \int_{\Xi\cap\mathbb{K}_{\ell,\mathfrak{j}}^m} \tilde{V}^{\delta,\ell,\mathfrak{j}}(\bar{\theta}, x)\ \dee\mu_{0,\ell,\mathfrak{j}}(x),
	\label{defn:small_v_delta_discrete}
\end{align}
where $\bar{\theta}=(\theta,\alpha)\in\R^{d}\times\R$ and $x\in\Xi$. Furthermore, define for every $\delta>0$
\begin{align}
	z_{D,\ell,\mathfrak{j},\delta}
	:=&\ \inf_{\bar{\theta}=(\theta, \alpha)\in\R^{d+1}} v^{\delta,\ell,\mathfrak{j}}(\bar{\theta}),\label{eqn:dro_problem_dual_discrete_smoothed}
\end{align}
Then, for every $\delta>0$ and $\bar{\theta}\in\R^{d+1}$ we have that
\begin{equation}
	\begin{split}
v^{\delta,\ell,\mathfrak{j}}(\bar{\theta})
\leq &\ 
v^{\ell,\mathfrak{j}}(\bar{\theta})
\leq
v^{\delta,\ell,\mathfrak{j}}(\bar{\theta})+\delta\log{N},
\ \mbox{ and hence also} \\ 
z_{D,\ell,\mathfrak{j},\delta} \leq &\ z_{D,\ell,\mathfrak{j}}\leq z_{D,\ell,\mathfrak{j},\delta}+\delta\log{N},
\label{eqn:nesterov_ineq}
\end{split}
\end{equation}
where $v^{\ell,\mathfrak{j}}(\bar{\theta})$ is defined in \eqref{defn:small_v_discrete} and $z_{D,\ell,\mathfrak{j}}$ is  defined in \eqref{eqn:dro_problem_discrete}.
\end{corollary}
\begin{proof}
This follows immediately from applying Lemma \ref{lemma:smooth_max} to the definition of $\tilde{V}^{\ell,\mathfrak{j}}$ in \eqref{defn:small_v_discrete}.
\end{proof}\

Table \ref{tbl:summaryofdef} provides a summary of the definitions of the quantities $z_P,z_{P,\ell,\mathfrak{j}},z_D,z_{D,\ell,\mathfrak{j}}, z_{D,\ell,\mathfrak{j},\delta}$ and their relationship with each other.
\begin{table}[t]
\begin{tabularx}{0.99\textwidth}
		{ 
			>{\hsize=0.7\hsize\linewidth=\hsize}X|
			>{\hsize=0.7\hsize\linewidth=\hsize}X
			>{\hsize=0.7\hsize\linewidth=\hsize}X
			>{\hsize=1.9\hsize\linewidth=\hsize}X
		} 
Quantity                 & Primal                                                        & Dual                                                                                                  & Relation                                                                                                                    \\ \hline
Original                 & $z_P:=\eqref{defn:zP}$                      & $z_D:=\eqref{eqn:dro_problem_dual}$ & $z_P=z_D$ by $\eqref{eqn:dro_problem_dual_transformed}$ \\
Discretised              & $z_{P,\ell,\mathfrak{j}}:=\eqref{defn:zPlj}$ & $z_{D,\ell,\mathfrak{j}}:=\eqref{eqn:dro_problem_discrete}$ & $z_{P,\ell,\mathfrak{j}}=z_{D,\ell,\mathfrak{j}}$ by $\eqref{eqn:dro_problem_dual_transformed_discrete}$, $|z_D-z_{D,\ell,\mathfrak{j}}|\leq \eqref{eqn:quadrature_error}$\\
Discretised and Smoothed & -                                                             & $z_{D,\ell,\mathfrak{j},\delta}:=\eqref{eqn:dro_problem_dual_discrete_smoothed}$ &                                                                                                  $z_{D,\ell,\mathfrak{j},\delta} \leq z_{D,\ell,\mathfrak{j}}\leq z_{D,\ell,\mathfrak{j},\delta}+\delta\log{N}$ by $\eqref{eqn:nesterov_ineq}$\\
\caption{Summary of definitions of $z_P,z_{P,\ell,\mathfrak{j}},z_D,z_{D,\ell,\mathfrak{j}}, z_{D,\ell,\mathfrak{j},\delta}$ and their relationship.}
\label{tbl:summaryofdef}
\end{tabularx}
\end{table}

\subsection{Applying the SGLD Algorithm} The final step of the proof is to obtain convergence bounds on the SGLD algorithm of \cite{zhang2023nonasymptotic} applied to the smoothed and discretised version of the dual problem $z_{D,\ell,\mathfrak{j},\delta}$ as defined in \eqref{eqn:dro_problem_dual_discrete_smoothed}. Note that here, we apply the SGLD algorithm of \cite{zhang2023nonasymptotic} to the variable $\bar{\theta}=(\theta,\alpha)\in\R^{d}\times \R$ which lies in the \textit{enlarged} space $\R^{d+1}$, as opposed to the original variable $\theta\in\R^d$. The next two propositions establish the global Lipschitz and dissipativity conditions on the function $\nabla_{\bar{\theta}}\tilde{V}^{\delta,\ell,\mathfrak{j}}$.
\begin{proposition}
\label{prop:glob_lipschitz}
Let Assumptions \ref{assumption:1} and \ref{assumption:3} hold. Then, for every $\delta>0$, there exists $L_\delta>0$ such that for all $\bar{\theta}_1,\bar{\theta}_2\in\R^{d+1}$ and all $x\in\Xi$,
\begin{align*}
|\nabla_{\bar{\theta}}V^{\delta,\ell,\mathfrak{j}}(\bar{\theta}_1,x)-\nabla_{\bar{\theta}}V^{\delta,\ell,\mathfrak{j}}(\bar{\theta}_2,x)|\leq L_\delta(1+|x|)^{2p}|\bar{\theta}_1-\bar{\theta}_2|.
\end{align*}
\end{proposition}
\begin{proof}
See Section~\ref{sec:proof_sec:proofOverview}. 
\end{proof}
\begin{proposition}
\label{prop:dissipativity}
Let Assumptions \ref{assumption:1} and \ref{assumption:3} hold. Then, there exist $a,b>0$ such that for all $\bar{\theta}\in\R^{d+1}$,
\begin{align*}
\lara{\bar{\theta}, \nabla_{\bar{\theta}}\tilde{V}^{\delta.\ell,\mathfrak{j}}(\bar{\theta}, x) }\geq a|\bar{\theta}|^2 - b.
\end{align*}
\end{proposition}
\begin{proof}
See Section~\ref{sec:proof_sec:proofOverview}. 
\end{proof}

Since the global Lipschitz and dissipativity conditions are satisfied by $\nabla_{\bar{\theta}}\tilde{V}^{\delta,\ell,\mathfrak{j}}$, the assumptions of the SGLD algorithm of \cite{zhang2023nonasymptotic}, when applied to the discretised and smoothed version of the dual problem $z_{D,\ell,\mathfrak{j},\delta}$ as defined in \eqref{eqn:dro_problem_dual_discrete_smoothed}, are satisfied. Hence, with the choice of the stochastic gradient $H^{\delta,\ell,\mathfrak{j}}$ of the SGLD algorithm defined in \eqref{eqn:sgld} as
\begin{align*}
H^{\delta,\ell,\mathfrak{j}} := \nabla_{\bar{\theta}}\tilde{V}^{\delta,\ell,\mathfrak{j}},
\end{align*} 

which is consistent with its definition previously given in \eqref{defn:grad_v_delta}, the following convergence bounds on the excess risk of the SGLD algorithm can be obtained.

\begin{proposition}
\label{prop:excess_risk_discrete}
Let Assumptions \ref{assumption:1}, \ref{assumption:3}, and \ref{assumption:4} hold. Let $\beta,\delta>0$ and $\lambda\in (0,\lambda_{\max,\delta})$, and let $\bar{\theta}_0\in L^4(\Omega, \mc{F},\mathbb{P};\R^{d+1})$. Moreover, let $(\bar{\hat{{\theta}}}^{\lambda, \delta, \ell,\mathfrak{j}}_n)_{n\in\N}$ denote 
 the sequence of estimators obtained from the SGLD algorithm in \eqref{eqn:sgld} defined on the probability space $(\Omega,\mc{F},\mathbb{P})$. Then, there exist constants $c_{\delta,\beta},C_{1,\delta,\ell,\mathfrak{j},\beta},C_{2,\delta,\ell,\mathfrak{j},\beta},C_{3,\delta,\beta}>0$ such that for each $n$, step size $\lambda\in(0,\lambda_{\max,\delta})$, and $\mathfrak{j}\in\N$,
\begin{align*}
\E_\mathbb{P}\left[v^{\delta,\ell,\mathfrak{j}}(\bar{\hat{\theta}}^{\lambda, \delta,\ell,\mathfrak{j}}_n)\right]-z_{D,\ell,\mathfrak{j},\delta} \leq C_{1,\delta,\ell,\mathfrak{j},\beta} e^{-c_{\delta,\beta}\lambda n/4} + C_{2,\delta,\ell,\mathfrak{j},\beta} \lambda^{1/4} + C_{3,\delta,\beta}, 
\end{align*}
where $z_{D,\ell,\mathfrak{j},\delta}$ and $v^{\delta,\ell,\mathfrak{j}}$ are defined in \eqref{eqn:dro_problem_dual_discrete_smoothed} and \eqref{defn:small_v_delta_discrete}, respectively. Moreover, the constants $c_{\delta,\beta}$, $C_{1,\delta,\ell,\mathfrak{j},\beta}$, $C_{2,\delta,\ell,\mathfrak{j},\beta}$, $C_{3,\delta,\beta}>0$ do not depend on $n$ or $\lambda$, and their growth orders are specified as
\begin{align}
C_{1,\delta,\ell,\mathfrak{j},\beta} =&\ \mc{O}\left(e^{\tilde{C}_{\delta}(1+d/\beta)(1+\beta)}\left(1+\frac{1}{1-e^{-c_{\delta,\beta}/2}}\right)\right),\nonumber\\
C_{2,\delta,\ell,\mathfrak{j},\beta} =&\ \mc{O}\left(e^{\tilde{C}_{\delta}(1+d/\beta)(1+\beta)}\left(1+\frac{1}{1-e^{-c_{\delta,\beta}/2}}\right)\right),\label{eqn:constants_order}\\
C_{3,\delta,\beta} =&\ \mc{O}\left((d/\beta)\log(\tilde{C}_{\delta}(\beta/d+1))\right),\nonumber
\end{align}
with $\tilde{C}_\delta>0$ being a constant not depending on $\lambda$, $n$, $d$, $\beta$. 
\end{proposition}
\begin{proof}
See Section~\ref{sec:proof_sec:proofOverview}. 
\end{proof}

We note that Proposition \ref{prop:excess_risk_discrete} gives the discretised excess risk of the SGLD algorithm \eqref{eqn:sgld} which is applied on the variable $\bar{\theta}=(\theta,\alpha)$ living in the extended space $\R^d\times \R$. Applying the duality result \eqref{eqn:dro_problem_dual_transformed_discrete} immediately yields the corresponding bound on the excess risk of the discretised primal problem \eqref{defn:u_discrete} which lives in the original space $\R^d$.
 To see this, observe that by \eqref{eq:discrete-duality} and \eqref{eqn:nesterov_ineq} (with $N=2^{m(\ell+\mathfrak{j})}$)
\begin{equation*}
\begin{split}
\E_{\mathbb{P}}[u^{\ell,\mathfrak{j}}(\hat{{\theta}}^{\lambda, \delta,\ell,\mathfrak{j}}_n)]-z_{P,\ell,\mathfrak{j}}
=&\ \E_{\mathbb{P}}[u^{\ell,\mathfrak{j}}(\hat{{\theta}}^{\lambda, \delta,\ell,\mathfrak{j}}_n)]-z_{D,\ell,\mathfrak{j}}
\\
= &\ \E_{\mathbb{P}}\lsrs{\left.\left(\inf_{\alpha\in\R}\int_{\Xi\cap\mathbb{K}_{\ell,\mathfrak{j}}^m}\tilde{V}^{\ell,\mathfrak{j}}(\bar{\theta}, x)\ \dee\mu_{0,\ell,\mathfrak{j}}(x)\right)\right|_{\theta=\hat{\theta}^{\lambda, \delta, \ell,\mathfrak{j}}_n}}
-z_{D,\ell,\mathfrak{j}}
\\
\leq &\ \E_{\mathbb{P}}\lsrs{\left.\left(\int_{\Xi\cap\mathbb{K}_{\ell,\mathfrak{j}}^m}\tilde{V}^{\ell,\mathfrak{j}}(\bar{\theta}, x)\ \dee\mu_{0,\ell,\mathfrak{j}}(x)\right)\right|_{\bar{\theta}=\hat{\bar{\theta}}^{\lambda, \delta, \ell,\mathfrak{j}}_n=(\hat{{\theta}}^{\lambda, \delta, \ell,\mathfrak{j}}_n,\hat{\alpha}^{\lambda, \delta, \ell,\mathfrak{j}}_n)}}-z_{D,\ell,\mathfrak{j}}\\
= & \
\E_{\mathbb{P}}[v^{\ell,\mathfrak{j}}(\hat{\bar{\theta}}^{\lambda, \delta, \ell,\mathfrak{j}}_n)] -z_{D,\ell,\mathfrak{j}}\\
\leq & \ \E_{\mathbb{P}}[v^{\delta,\ell,\mathfrak{j}}(\hat{\bar{\theta}}^{\lambda, \delta, \ell,\mathfrak{j}}_n)] -z_{D,\ell,\mathfrak{j},\delta}
+\delta m(\ell+\mathfrak{j})\log{2}.
\end{split}
\end{equation*}
This hence indeed allows us to bound the excess risk  of the discretised primal problem \eqref{defn:u_discrete}  directly using Proposition \ref{prop:excess_risk_discrete}, as stated in the following corollary.
\begin{corollary}
\label{corollary:excess_risk_discrete}
Let Assumptions \ref{assumption:1}, \ref{assumption:3}, and \ref{assumption:4} hold. Let $\beta,\delta>0$ and $\lambda\in (0,\lambda_{\max,\delta})$, and let $\bar{\theta}_0\in L^4(\Omega, \mc{F},\mathbb{P};\R^{d+1})$. Moreover, let $(\hat{{\theta}}^{\lambda, \delta, \ell,\mathfrak{j}}_n)_{n\in\N}$ denote the first $d$ components of the sequence of estimators obtained from the SGLD algorithm in \eqref{eqn:sgld} defined on the probability space $(\Omega,\mc{F},\mathbb{P})$. Then,
\begin{align*}
\E_\mathbb{P}\left[u^{\ell,\mathfrak{j}}(\hat{\theta}^{\lambda, \delta,\ell,\mathfrak{j}}_n)\right]-z_{P,\ell,\mathfrak{j}} \leq C_{1,\delta,\ell,\mathfrak{j},\beta} e^{-c_{\delta,\beta}\lambda n/4} + C_{2,\delta,\ell,\mathfrak{j},\beta} \lambda^{1/4} + C_{3,\delta,\beta} + \delta m(\ell+\mathfrak{j})\log{2},
\end{align*}
where $c_{\delta,\beta},C_{1,\delta,\ell,\mathfrak{j},\beta},C_{2,\delta,\ell,\mathfrak{j},\beta},C_{3,\delta,\beta}>0$ are the constants given in Proposition \ref{prop:excess_risk_discrete}.
\end{corollary}
\begin{proof}
See Section~\ref{sec:proof_sec:proofOverview}. 
\end{proof}

Finally, the last piece required for the proof of the main results of this paper is an upper bound between the undiscretised and discretised expected risk of the first $d$ components of the SGLD algorithm \eqref{eqn:sgld}, obtained from the primal problems \eqref{eqn:dro_problem} and \eqref{defn:u_discrete}, respectively. We state the bound in the following proposition.
\begin{proposition}
\label{prop:discretisation}
Let Assumptions \ref{assumption:1}, \ref{assumption:3}, and \ref{assumption:4} hold. Let $\beta,\delta>0$ and $\lambda\in (0,\lambda_{\max,\delta})$, and let $\bar{\theta}_0\in L^4(\Omega, \mc{F},\mathbb{P};\R^{d+1})$. Moreover, let $(\hat{{\theta}}^{\lambda, \delta, \ell,\mathfrak{j}}_n)_{n\in\N}$ denote the first $d$ components of the sequence of estimators obtained from the SGLD algorithm in \eqref{eqn:sgld} defined on the probability space $(\Omega,\mc{F},\mathbb{P})$. Then, there exists constants $\tilde{C}_4, C_{5,\delta,\beta}, C_6>0$, which explicit expressions are given in \eqref{const:c4_c5_c6}, such that for each $n$, step size $\lambda\in (0,\lambda_{\max,\delta})$, and $\mathfrak{j}\in\N$,
\begin{align*}
\lbrb{\E_{\mathbb{P}}\left[u(\hat{\theta}_{n}^{\lambda,\delta,\ell,\mathfrak{j}})\right]-\E_{\mathbb{P}}\left[u^{\ell,\mathfrak{j}}(\hat{\theta}_{n}^{\lambda,\delta,\ell,\mathfrak{j}})\right]}
\leq \frac{\sqrt{m}(\tilde{C}_4 + C_{5,\delta,\beta} + C_{6}e^{-a\lambda(n+1)})}{2^\mathfrak{j}}.
\end{align*}
\end{proposition}

\begin{proof}
See Section~\ref{sec:proof_sec:proofOverview}.
\end{proof}
\subsection{Proof of Main Results in Section~\ref{sec:AssumpMainRes}}
We have established sufficient machinery thus far to prove the main results of this paper.\\

\textbf{\textit{Proof of Theorem \ref{theorem:excess risk}}.} By the duality result of Theorem \ref{theorem:wass_duality} and the triangle inequality, one obtains the following decomposition:
\begin{align*}
&\E_{\mathbb{P}}[u(\hat{\theta}_n^{\lambda,\delta,\ell,\mathfrak{j}})] - \inf_{\theta\in\R^d}u(\theta)\\
=&\ \E_{\mathbb{P}}[u(\hat{\theta}_n^{\lambda,\delta,\ell,\mathfrak{j}})] -
 z_P\nonumber\\
 =&\ \E_{\mathbb{P}}[u(\hat{\theta}_n^{\lambda,\delta,\ell,\mathfrak{j}})] - z_D\nonumber\\
\leq&\ |\E_{\mathbb{P}}[u(\hat{\theta}_n^{\lambda,\delta,\ell,\mathfrak{j}})]-\E_{\mathbb{P}}[u^{\ell,\mathfrak{j}}(\hat{\theta}_n^{\lambda,\delta,\ell,\mathfrak{j}})]|+|\E_{\mathbb{P}}[u^{\ell,\mathfrak{j}}(\hat{\theta}_n^{\lambda,\delta,\ell,\mathfrak{j}})]-z_{D,\ell,\mathfrak{j}}| + |z_{D,\ell,\mathfrak{j}} - z_D|\nonumber\\
=&\ |\E_{\mathbb{P}}[u(\hat{\theta}_n^{\lambda,\delta,\ell,\mathfrak{j}})]-\E_{\mathbb{P}}[u^{\ell,\mathfrak{j}}(\hat{\theta}_n^{\lambda,\delta,\ell,\mathfrak{j}})]|+|\E_{\mathbb{P}}[u^{\ell,\mathfrak{j}}(\hat{\theta}_n^{\lambda,\delta,\ell,\mathfrak{j}})]-z_{P,\ell,\mathfrak{j}}| + |z_{D,\ell,\mathfrak{j}} - z_D|.
\end{align*}
Observe that the first term on the RHS of the above decomposition has an upper bound given in Proposition~\ref{prop:discretisation}, the second term has an upper bound given in Corollary \ref{corollary:excess_risk_discrete}, and the third term has an upper bound given in Proposition \ref{prop:quadrature}. It follows that
\begin{align*}
&\ \E_{\mathbb{P}}[u(\hat{\theta}_n^{\lambda,\delta,\ell,\mathfrak{j}})] - \inf_{\theta\in\R^d}u(\theta)\nonumber\\
\leq&\ \frac{\sqrt{m}(\tilde{C}_4 + C_{5,\delta,\beta} + C_{6}e^{-a\lambda(n+1)})}{2^\mathfrak{j}} + C_{1,\delta,\ell,\mathfrak{j},\beta} e^{-c_{\delta,\beta}\lambda n/4} + C_{2,\delta,\ell,\mathfrak{j},\beta} \lambda^{1/4} + C_{3,\delta,\beta} \nonumber\\
&\ + \delta m(\ell+\mathfrak{j})\log{2} +
	 \frac{\sqrt{m}(J_U+ 2p(1+4M_\Xi)^{p-1})(1+R_\mc{K} )}{2^\mathfrak{j}}  \nonumber\\
=&\ C_{1,\delta,\ell,\mathfrak{j},\beta} e^{-c_{\delta,\beta}\lambda n/4} + C_{2,\delta,\ell,\mathfrak{j},\beta} \lambda^{1/4} + C_{3,\delta,\beta} \nonumber\\
&\ + \delta m(\ell+\mathfrak{j})\log{2}+\frac{\sqrt{m}}{2^\mathfrak{j}}\left(C_4+C_{5,\delta,\beta}+C_6e^{-a\lambda(n+1)/2}\right).
\end{align*}
Here, $c_{\delta,\beta}, C_{1,\delta,\ell,\mathfrak{j},\beta}, C_{2,\delta,\ell,\mathfrak{j},\beta}, C_{3,\delta,\beta}$ are as stated in Proposition \ref{prop:excess_risk_discrete} and given explicitly in Table 2 of \cite{zhang2023nonasymptotic}, with
\[
\dot{c}\leftarrow c_{\delta,\beta},\qquad C_1^\#\leftarrow C_{1,\delta,\ell,\mathfrak{j},\beta},\qquad C_2^\#\leftarrow C_{2,\delta,\ell,\mathfrak{j},\beta},\qquad C_3^\#\leftarrow C_{3,\delta,\beta},
\]

in the notation of \cite{zhang2023nonasymptotic}. Moreover, we define
\begin{align}
\begin{split}\label{const:c4}
C_4:=&\ \tilde{C}_4 + (J_U+2p(1+4M_\Xi)^{p-1})(1+R_\mc{K})
\end{split}
\end{align}
with $R_\mc{K}$ defined in \eqref{eq:rk_expression}, and $\tilde{C}_4, C_{5,\delta,\beta}, C_6$ are as specified in the proof of Proposition \ref{prop:discretisation} as in \eqref{const:c4_c5_c6}. 
This completes the proof.\qed\\

\textbf{\textit{Proof of Corollary \ref{corollary:main_result}}.}
Let $\varepsilon>0$ be given. Our goal is to choose sequentially the parameters $\ell,\mathfrak{j},\delta,\beta,\lambda$, and $n$ so that the expected excess risk $\E_\mathbb{P}\!\left[u(\hat{\theta}^{\lambda, \delta,\ell,\mathfrak{j}}_n)\right]-\inf_{\theta\in\R^d} u(\theta)$ is less than the prescribed precision level $\varepsilon$. To this end, we first rewrite the upper bound obtained in Theorem~\ref{theorem:excess risk} in a form that facilitates subsequent calculations. More precisely, we observe, by using the explicit expression of $C_{5,\delta,\beta}$ provided in \eqref{const:c4_c5_c6}, that
\begin{align*}
C_4 + C_{5,\delta,\beta}
=&\ C_4 + \mathfrak{C}_4\left(2\mathfrak{M}_1\lambda_{\max,\delta} + 2b + 2(d+1)/\beta\right)^{1/2}\left(\lambda_{\max,\delta}+a^{-1}\right)^{1/2}\nonumber\\
\leq&\ C_4 + \mathfrak{C}_4\left((2\mathfrak{M}_1+1)\lambda_{\max,\delta} + a^{-1} + 2b + 2(d+1)/\beta\right)\nonumber\\
=&\ (C_4 + \mathfrak{C}_4(a^{-1}+2b)) + \mathfrak{C}_4(2\mathfrak{M}_1+1)\lambda_{\max,\delta} + 2\mathfrak{C}_4(d+1)/\beta.
\end{align*}

Hence, it follows from the result of Theorem \ref{theorem:excess risk} that the upper bound on the excess risk of the algorithm can be decomposed as
\begin{align}
\E_\mathbb{P}\left[u(\hat{\theta}^{\lambda, \delta,\ell,\mathfrak{j}}_n)\right]-\inf_{\theta\in\R^d} u(\theta) 
\leq&\ \frac{\sqrt{m}(C_4 + \mathfrak{C}_4(a^{-1}+2b))}{2^\mathfrak{j}} \nonumber\\
&\ + \delta m(\ell+\mathfrak{j})\log{2} + \frac{\sqrt{m}}{2^\mathfrak{j}}\mathfrak{C}_4(2\mathfrak{M}_1+1)\lambda_{\max,\delta}\nonumber\\
&\ + \frac{\sqrt{m}}{2^\mathfrak{j}}\mathfrak{C}_4(d+1)/\beta + C_{3,\delta,\beta}\nonumber\\
&\ + C_{2,\delta,\ell,\mathfrak{j},\beta}\lambda^{1/4}\nonumber\\
&\ + C_{1,\delta,\ell,\mathfrak{j},\beta}e^{-c_{\delta,\beta}\lambda n/4} + \frac{\sqrt{m}}{2^\mathfrak{j}}C_6e^{-a\lambda(n+1)/2}.\label{eqn:bound_decomp}
\end{align}
Now, we proceed to provide explicit bounds for each parameter. Fixing first $\ell$ such that $\Xi\subset [-2^{\ell-1},2^{\ell-1})^m$, then fixing
\begin{align*}
\mathfrak{j}>\max\left\{\log_2\left(\frac{5\sqrt{m}(C_4+\mathfrak{C}_4(a^{-1}+2b))}{\varepsilon}\right), \log_2\sqrt{m}\right\},
\end{align*} 
one has
\begin{align}
\frac{\sqrt{m}(C_4 + \mathfrak{C}_4(a^{-1}+2b))}{2^\mathfrak{j}}<\frac{\varepsilon}{5}.\label{eqn:bound_decomp_a}
\end{align}

Next, we fix 
\begin{align*}
\delta \in\left(0, \min\left\{\frac{\varepsilon}{10 m(\ell+\mathfrak{j})\log{2}},\frac{\mathfrak{C}_2}{\sqrt{a\mathfrak{C}_1}},\mathfrak{C}_2\sqrt{\frac{\varepsilon 2^\mathfrak{j}}{10\mathfrak{C}_1\mathfrak{C}_4(2\mathfrak{M}_1+1)\sqrt{m}}}\right\}\right),
\end{align*}

so that 
\begin{align}
\delta m(\ell+\mathfrak{j})\log{2} < \frac{\varepsilon}{10}.\label{eqn:bound_decomp_b}
\end{align}

Then, by using the explicit expressions of the constants $\mathfrak{C}_1,\mathfrak{C}_2, \tilde{L}_\delta, a$ provided in \eqref{defn:step} and \eqref{const:tilde_l_delta}, we obtain that $\frac{\mathfrak{C}_1}{\tilde{L}_\delta^2}
<\frac{\delta^2\mathfrak{C}_1}{\mathfrak{C}_2^2}<a^{-1}$, which implies $\lambda_{\max,\delta}\leq \frac{\mathfrak{C}_1}{\tilde{L}_\delta^2}$ by the definition of $\lambda_{\max,\delta}$ provided in \eqref{defn:step}. Hence,
\begin{align}
\frac{\sqrt{m}}{2^\mathfrak{j}}\mathfrak{C}_4(2\mathfrak{M}_1+1)\lambda_{\max,\delta}
=&\ \frac{\sqrt{m}}{2^\mathfrak{j}}\mathfrak{C}_4(2\mathfrak{M}_1+1)\cdot \frac{\mathfrak{C}_1}{\tilde{L}_\delta^2}\nonumber\\
<&\ \frac{\sqrt{m}}{2^\mathfrak{j}}\mathfrak{C}_4(2\mathfrak{M}_1+1)\cdot \frac{\mathfrak{C}_1}{\mathfrak{C}_2^2}\cdot \delta^2\nonumber\\
<&\ \frac{\sqrt{m}}{2^\mathfrak{j}}\mathfrak{C}_4(2\mathfrak{M}_1+1)\cdot \frac{\mathfrak{C}_1}{\mathfrak{C}_2^2}\cdot \left(\mathfrak{C}_2\sqrt{\frac{\varepsilon 2^\mathfrak{j}}{10\mathfrak{C}_1\mathfrak{C}_4(2\mathfrak{M}_1+1)\sqrt{m}}}\right)^2
=\ \frac{\varepsilon}{10}.\label{eqn:bound_decomp_c}
\end{align}

We next fix
\begin{align*}
\beta > \max\left\{\frac{100(d+1)}{\varepsilon^2}, \frac{10(d+1)\left(1+\log\left(\frac{(\tilde{L}_\delta-1)(b+1)}{a}\right)\right)}{\varepsilon}, \frac{10\sqrt{m}\mathfrak{C}_4(d+1)}{\varepsilon 2^\mathfrak{j}}\right\}.
\end{align*}

Then, from the explicit form of $C_{3,\delta,\beta}$ given in \cite[Eq.\ (32)]{zhang2023nonasymptotic} with $C_3^\#\leftarrow C_{3,\delta,\beta}$ in the notation of \cite{zhang2023nonasymptotic} and the fact that $(\log (1+x))/x\leq 1/\sqrt{x}$ for all $x>0$, we obtain that
\begin{align}
C_{3,\delta,\beta}
=&\ \frac{d+1}{2\beta}\log\left(1+\frac{\beta}{d+1}\right)
+ \frac{d+1}{2\beta}\left(1+\log\left(\frac{(\tilde{L}_\delta-1)(b+1)}{a}\right)\right)\nonumber\\
<&\ \frac{\sqrt{d+1}}{2}\cdot\sqrt{\frac{\varepsilon^2}{100(d+1)}}\nonumber\\
&\ + \frac{d+1}{2}\left(1+\log\left(\frac{(\tilde{L}_\delta-1)(b+1)}{a}\right)\right)\cdot \tfrac{\varepsilon}{10(d+1)\left(1+\log\left(\tfrac{(\tilde{L}_\delta-1)(b+1)}{a}\right)\right)}\nonumber\\
<&\ \frac{\varepsilon}{20} + \frac{\varepsilon}{20}
=\ \frac{\varepsilon}{10},\label{eqn:bound_decomp_d}
\end{align}
as well as 
\begin{align}
\frac{\sqrt{m}}{2^\mathfrak{j}}\mathfrak{C}_4(d+1)/\beta<\frac{\varepsilon}{10}.\label{eqn:bound_decomp_e}
\end{align}

Finally, we fix
\begin{align*}
\lambda \in \left(0,\min\left\{\lambda_{\max,\delta}, \frac{\varepsilon^4}{625C_{2,\delta,\ell,\mathfrak{j},\beta}^4}\right\}\right)
\end{align*}

which implies that
\begin{align}
C_{2,\delta,\ell,\mathfrak{j},\beta}\lambda^{1/4}<\frac{\varepsilon}{5},\label{eqn:bound_decomp_f}
\end{align}

and then fix
\begin{align*}
n>\max\left\{\frac{4}{c_{\delta,\beta}\lambda}
\log\left(\frac{10C_{1,\delta,\ell,\mathfrak{j},\beta}}{\varepsilon}\right),\frac{2}{a\lambda}\log\left(\frac{10C_6}{\varepsilon}\right)-1\right\},
\end{align*}

implying that
\begin{align}
C_{1,\delta,\ell,\mathfrak{j},\beta}e^{-c_{\delta,\beta}\lambda n/4} + \frac{\sqrt{m}}{2^\mathfrak{j}}C_6e^{-a\lambda(n+1)/2}
<&\ C_{1,\delta,\ell,\mathfrak{j},\beta}\exp\left\{-\frac{c_{\delta,\beta}\lambda}{4}\cdot \frac{4}{c_{\delta,\beta}\lambda}
\log\left(\frac{10C_{1,\delta,\ell,\mathfrak{j},\beta}}{\varepsilon}\right)\right\}\nonumber\\
&\ + C_6\exp\left\{-\frac{a\lambda}{2}\left(\frac{2}{a\lambda}\log\left(\frac{10C_6}{\varepsilon}\right)-1+1\right)\right\}\nonumber\\
=&\ \frac{\varepsilon}{10}+\frac{\varepsilon}{10}
=\ \frac{\varepsilon}{5}.\label{eqn:bound_decomp_g}
\end{align}

Substituting \eqref{eqn:bound_decomp_a}, \eqref{eqn:bound_decomp_b}, \eqref{eqn:bound_decomp_c}, \eqref{eqn:bound_decomp_d}, \eqref{eqn:bound_decomp_e}, \eqref{eqn:bound_decomp_f}, and \eqref{eqn:bound_decomp_g} into \eqref{eqn:bound_decomp} thus yields
\begin{align*}
\E_\mathbb{P}\left[u(\hat{\theta}^{\lambda, \delta,\ell,\mathfrak{j}}_n)\right]-\inf_{\theta\in\R^d} u(\theta) 
<&\ \frac{\varepsilon}{5} + \frac{\varepsilon}{10}+ \frac{\varepsilon}{10} +\frac{\varepsilon}{10}+\frac{\varepsilon}{10} + \frac{\varepsilon}{5} + \frac{\varepsilon}{5}
=\ \varepsilon,
\end{align*}

which completes the proof.\qed
	\\
	\\
\section{Proof of Statements in Section~\ref{sec:AssumpMainRes}}\label{sec:proof_sec:AssumpMainRes}
\subsection*{Proof of Remark \ref{remark:u_lipschitz}}
\begin{proof}
Fix $\theta_1,\theta_2\in\R^d$, $x\in\Xi$, and denote $f(t):=U(t\theta_1+(1-t)\theta_2,x)$ for any $t\in[0,1]$, such that $f'(t)=\nabla_\theta U(t\theta_1+(1-t)\theta_2,x)(\theta_1-\theta_2)$. By Assumption \ref{assumption:3}, one obtains
\begin{align*}
|U(\theta_1,x)-U(\theta_2,x)|
=\ |f(1)-f(0)|
\leq&\ \int^1_0 |f'(t)|\ \dee t\nonumber\\
\leq&\ \int^1_0 |\nabla_\theta U(t\theta_1+(1-t)\theta_2,x)|\cdot |\theta_1-\theta_2|\ \dee t\nonumber\\
\leq&\ K_\nabla |\theta_1-\theta_2|,
\end{align*}

which establishes the first part of the remark. It then follows for all $\theta\in\R^d$ and $x\in \Xi$ that 
\begin{align*}
|U(\theta,x)|
\leq\ |U(\theta,x)-U(0,x)| + |U(0,x)|
\leq&\ K_\nabla|\theta| + |U(0,x)|
\leq \tilde{K}_\nabla  (1+|\theta|),
\end{align*}
where $\tilde{K}_\nabla:=\max\{K_\nabla,\max_{x\in\Xi}|U(0,x)|\}$. This establishes the second part of the remark.
\end{proof}

\subsection*{Proof of Remark \ref{remark:iota_dissipativity}}
\begin{proof}
With the choice $\iota(\alpha) = \log(\cosh{\alpha})$, one has $L_\iota=M_\iota=1$, since $|\iota'(\alpha)|=|\tanh{\alpha}|\leq 1$ and $|\iota''(\alpha)|=|\operatorname{sech}^2{\alpha}|\leq 1$. Furthermore, note that $\iota(\alpha)-(|\alpha|-\log{2})\to 0$ and $(\iota'(\alpha)- \operatorname{sgn}\alpha)\to 0$ as $|\alpha|\to\infty$. Therefore,
\begin{align*}
\lim_{|\alpha|\to\infty}\lsrs{\alpha \iota(\alpha)\iota'(\alpha) - |\alpha|^2+(\log{2})|\alpha|} = 0.
\end{align*}
This implies that for any $\mathfrak{w} >0$, there exists an $R_\mathfrak{w} > 0$ such that
\begin{align*}
\alpha \iota(\alpha)\iota'(\alpha)
\geq&\ \lsrs{|\alpha|^2 - (\log{2})|\alpha| - \mathfrak{w}}\1_{\{|\alpha|> R_\mathfrak{w}\}} + \alpha \iota(\alpha)\iota'(\alpha)\1_{\{|\alpha|\leq R_\mathfrak{w}\}}\\
\geq&\ \lsrs{\frac{1}{2}|\alpha|^2\1_{\{|\alpha| > 2\log{2}\}} -(2\log^2{2})\1_{\{|\alpha| \leq 2\log{2}\}}- \mathfrak{w}}\1_{\{|\alpha|> R_\mathfrak{w}\}} + \alpha \iota(\alpha)\iota'(\alpha)\1_{\{|\alpha|\leq R_\mathfrak{w}\}}\\
=&\ \lsrs{\frac{1}{2}|\alpha|^2-\left(\frac{1}{2}|\alpha|^2+2\log^2{2}\right)\1_{\{|\alpha| \leq 2\log{2}\}}- \mathfrak{w}}\1_{\{|\alpha|> R_\mathfrak{w}\}} + \alpha \iota(\alpha)\iota'(\alpha)\1_{\{|\alpha|\leq R_\mathfrak{w}\}}\\
\geq&\ \lsrs{\frac{1}{2}|\alpha|^2-4\log^2{2}- \mathfrak{w}}\1_{\{|\alpha|> R_\mathfrak{w}\}} + \alpha \iota(\alpha)\iota'(\alpha)\1_{\{|\alpha|\leq R_\mathfrak{w}\}}\\
\geq&\ \lsrs{\frac{1}{2}|\alpha|^2-4\log^2{2}- \mathfrak{w}} - \lsrs{\frac{1}{2}|\alpha|^2 + M_\mathfrak{w}}\1_{\{|\alpha|\leq R_\mathfrak{w}\}}\\
\geq&\ \frac{1}{2}|\alpha|^2 - \left(4\log^2{2} +\mathfrak{w} + \frac{1}{2}R_\mathfrak{w}^2 + M_\mathfrak{w}\right),
\end{align*}
where $M_\mathfrak{w} := \max_{|\alpha|\leq R_\mathfrak{w}}|\alpha\iota(\alpha)\iota'(\alpha)|$. Therefore, by fixing a particular choice of $\mathfrak{w}>0$, the dissipativity condition holds with $a_\iota=\frac{1}{2}$ and $b_\iota = \left(4\log^2{2} +\mathfrak{w} + \frac{1}{2}R_\mathfrak{w}^2 + M_\mathfrak{w}\right)$, as desired. In particular, one may choose concretely $\mathfrak{w}= R_\mathfrak{w}=M_\mathfrak{w}=1$ to obtain the explicit expression $b_\iota=4\log^2{2}+\frac{5}{2}$.
\end{proof}
%
%
\section{Proof of Statements in Section~\ref{sec:Application}}\label{sec:proof_sec:Application}
\subsection*{Proof of Proposition \ref{prop:regression_example}}
\begin{proof}
To verify Assumption \ref{assumption:3}, note that for any $i=1,\cdots,m-1$, $\theta = (w,b)\in\R^{m-1}\times\R$, and $x=(z,y)\in\Xi\subset \R^m$,
\begin{align}\label{eqn:grad_U}
\begin{split}
\frac{\partial U}{\partial w_i}(\theta, x)& = -2z_i(y-\sigma_1(\langle w,z\rangle+b_0))\sigma_1(\langle w,z\rangle+b_0)(1-\sigma_1(\langle w,z\rangle+b_0)),\\
\frac{\partial U}{\partial b_0}(\theta, x)
&= -2(y-\sigma_1(\langle w,z\rangle+b_0))\sigma_1(\langle w,z\rangle+b_0)(1-\sigma_1(\langle w,z\rangle+b_0)).
\end{split}
\end{align}
We note that for any $\theta, \overline{\theta} = (\overline{w}, \overline{b}_0)\in \R^m$,
\begin{equation}\label{eq:lipschitzsigmoid}
|\sigma_1(\langle w,z\rangle+b_0) - \sigma_1(\langle \overline{w},z\rangle+\overline{b}_0)| \leq |z||w-\overline{w}|+|b_0-\overline{b}_0|\leq (1+|x|)|\theta - \overline{\theta} |.
\end{equation}
Thus, by using \eqref{eqn:grad_U} and \eqref{eq:lipschitzsigmoid}, we obtain, for any $\theta, \overline{\theta} \in\R^m$,
\begin{equation*}
\left|\frac{\partial U}{\partial w_i}(\theta, x) - \frac{\partial U}{\partial w_i}(\overline{\theta}, x)\right|\leq 6(1+|x|)^3|\theta - \overline{\theta} |,\quad
\left|\frac{\partial U}{\partial b_0}(\theta, x) - \frac{\partial U}{\partial b_0}(\overline{\theta}, x)\right|\leq 6(1+|x|)^2|\theta - \overline{\theta} |,
\end{equation*}
which implies that
\[
\lbrb{\nabla_\theta U(\theta,x)-\nabla_\theta U(\overline{\theta}, x)} \leq 6m(1+|x|)^3|\theta - \overline{\theta}| \leq 6m(1+M_\Xi)^3|\theta - \overline{\theta}| .
\]
Moreover, it is easily verifiable that, for any $\theta\in\R^d$ and $x\in\Xi$, 
\[
\lbrb{\nabla_\theta U(\theta,x)} \leq 2m(1+M_\Xi)^2.
\]
Thus, Assumption \ref{assumption:3} holds with $L_{\nabla} = 6m(1+M_\Xi)^3$ and $K_{\nabla} = 2m(1+M_\Xi)^2$.
\\

Lastly, it remains to verify Assumption \ref{assumption:4}. For any $\theta \in \R^m$, $x,\overline{x}=(\overline{z},\overline{y})\in\Xi$, we have
\begin{align*}
&|U(\theta,x)-U(\theta,\overline{x})| \\
& =||y-\sigma_1(\langle w,z\rangle+b_0)|^2 - |\overline{y}-\sigma_1(\langle w,\overline{z}\rangle+b_0)|^2|\\
& = \left|\left(|y-\sigma_1(\langle w,z\rangle+b_0)|-|\overline{y}-\sigma_1(\langle w,\overline{z}\rangle+b_0)|\right) \left(|y-\sigma_1(\langle w,z\rangle+b_0)|+|\overline{y}-\sigma_1(\langle w,\overline{z}\rangle+b_0)|\right)\right|\\
&\leq 2(1+|x|+|\overline{x}|)(|y-\overline{y}|+|\sigma_1(\langle w,\overline{z}\rangle+b_0) - \sigma_1(\langle w,z\rangle+b_0) |)\\
&\leq 2(1+|x|+|\overline{x}|)(|x-\overline{x}|+|w||z -\overline{z}|)\\
&\leq 4(1+M_\Xi)(1+|\theta|)|x-\overline{x}|,
\end{align*}
which implies that Assumption \ref{assumption:4} is satisfied with $J_U=4(1+M_\Xi)$. This completes the proof.

\end{proof}
%
%
\section{Proof of Statements in Section~\ref{sec:proofOverview}}\label{sec:proof_sec:proofOverview}

\subsection*{Proof of Lemma \ref{lemma:finite_grid}}
Recall the definition of $Q_{\boldsymbol{i},j}$ in \eqref{eqn:qij}. One obtains that
\begin{align*}
z_{D,\ell,j}
:=&\ \inf_{\theta\in\R^d}\inf_{\mathfrak{a}\geq 0}\left\{\int_{\Xi}\sup_{y\in\Xi}\left\{U(\theta,[y]_j)-\mathfrak{a} |[x]_j-[y]_j|^p\right\}\ \dee\mu_0(x)+\frac{\eta_1}{2}|\theta|^{2}+\frac{\eta_2}{2}|\mathfrak{a}|^2\right\}\nonumber\\
=&\ \inf_{\theta\in\R^d}\inf_{\mathfrak{a}\geq 0}\left\{\int_{[-2^{\ell-1},2^{\ell-1})^m}\sup_{y\in\Xi\cap\K_{\ell,j}^m}\left\{U(\theta,y)-\mathfrak{a} |[x]_j-y|^p\right\}\ \dee\mu_{0.\ell}(x)+\frac{\eta_1}{2}|\theta|^{2}+\frac{\eta_2}{2}|\mathfrak{a}|^2\right\}\nonumber\\
=&\ \inf_{\theta\in\R^d}\inf_{\mathfrak{a}\geq 0}\left\{\sum_{\boldsymbol{i}\in \K_{\ell,j}^m}\int_{Q_{\boldsymbol{i},j}}\max_{y\in\Xi\cap\K_{\ell,j}^m}\left\{U(\theta,y)-\mathfrak{a} |[x]_j-y|^p\right\}\ \dee\mu_{0.\ell}(x)+\frac{\eta_1}{2}|\theta|^{2}+\frac{\eta_2}{2}|\mathfrak{a}|^2\right\}.
\end{align*}
This further implies, by the definitions of $\mu_{0,\ell,j}$ and $[\cdot]_j$ in \eqref{eqn:dprobmeasure} and \eqref{eqn:gridfun}, that
\begin{align*}
z_{D,\ell,j}=&\ \inf_{\theta\in\R^d}\inf_{\mathfrak{a}\geq 0}\left\{\sum_{\boldsymbol{i}\in \K_{\ell,j}^m}\int_{Q_{\boldsymbol{i},j}}\max_{y\in\Xi\cap\K_{\ell,j}^m}\left\{U(\theta,y)-\mathfrak{a} |\boldsymbol{i}-y|^p\right\}\ \dee\mu_{0.\ell}(x)+\frac{\eta_1}{2}|\theta|^{2}+\frac{\eta_2}{2}|\mathfrak{a}|^2\right\}\nonumber\\
=&\ \inf_{\theta\in\R^d}\inf_{\mathfrak{a}\geq 0}\left\{\sum_{\boldsymbol{i}\in \K_{\ell,j}^m}\max_{y\in\Xi\cap\K_{\ell,j}^m}\left\{U(\theta,y)-\mathfrak{a} |\boldsymbol{i}-y|^p\right\}\ \mu_{0.\ell}(Q_{\boldsymbol{i},j})+\frac{\eta_1}{2}|\theta|^{2}+\frac{\eta_2}{2}|\mathfrak{a}|^2\right\}\nonumber\\
=&\ \inf_{\theta\in\R^d}\inf_{\mathfrak{a}\geq 0}\left\{\sum_{\boldsymbol{i}\in \K_{\ell,j}^m}\max_{y\in\Xi\cap\K_{\ell,j}^m}\left\{U(\theta,y)-\mathfrak{a} |\boldsymbol{i}-y|^p\right\}\ \mu_{0,\ell,j}(\{\boldsymbol{i}\})+\frac{\eta_1}{2}|\theta|^{2}+\frac{\eta_2}{2}|\mathfrak{a}|^2\right\}\nonumber\\
=&\ \inf_{\theta\in\R^d}\inf_{\mathfrak{a}\geq 0}\left\{\int_{\K_{\ell,j}^m}\max_{y\in\Xi\cap\K_{\ell,j}^m}\left\{U(\theta,y)-\mathfrak{a} |x-y|^p\right\}\ \dee\mu_{0,\ell,j}(x)+\frac{\eta_1}{2}|\theta|^{2}+\frac{\eta_2}{2}|\mathfrak{a}|^2\right\},
\end{align*}
as desired. \qed

\subsection*{Proof of Lemma \ref{lemma:compactness}}
\begin{proof}
We recall $B_{R_\mc{K}}(0):=\{(\theta,\mathfrak{a})\in\R^d\times [0,\infty):|(\theta,\mathfrak{a})|\leq R_\mc{K} \}$ with $R_\mc{K}$ given in \eqref{eq:rk_expression}, as well as the definitions
\begin{align*}
z_{D}
:=&\ \inf_{\theta\in\R^d}\inf_{\mathfrak{a}\geq 0}\left\{\int_{\Xi}\sup_{y\in\Xi}\left\{U(\theta,y)-\mathfrak{a} |x-y|^p\right\}\ \dee\mu_0(x)+\frac{\eta_1}{2}|\theta|^{2}+\frac{\eta_2}{2}|\mathfrak{a}|^2\right\},\\
z_{D,\ell,\mathfrak{j}}
:=&\ \inf_{\theta\in\R^d}\inf_{\mathfrak{a}\geq 0}\left\{\int_{\Xi}\sup_{y\in\Xi}\left\{U(\theta,[y]_\mathfrak{j})-\mathfrak{a} |[x]_\mathfrak{j}-[y]_\mathfrak{j}|^p\right\}\ \dee\mu_0(x)+\frac{\eta_1}{2}|\theta|^{2}+\frac{\eta_2}{2}|\mathfrak{a}|^2\right\}.
\end{align*}

To establish the first part of the lemma, we apply 
Assumptions \ref{assumption:1}, \ref{assumption:3}, and Remark \ref{remark:u_lipschitz} to obtain, for every $(\theta,\mathfrak{a})\in\R^d\times [0,\infty)$, that
\begin{align}
\begin{split}\label{eqn:start_arg_1}
&\ \int_\Xi \sup_{y\in\Xi}(U(\theta,y)-\mathfrak{a}|x-y|^p)\ \dee \mu_0(x) +\frac{\eta_1}{2}|\theta|^2+\frac{\eta_2}{2}|\mathfrak{a}|^2\\
\geq&\ \int_\Xi \sup_{y\in\Xi}(U(\theta,y)-\mathfrak{a}|x-y|^p)\ \dee \mu_0(x) +\frac{\min\{\eta_1, \eta_2\}}{2}|(\theta,\mathfrak{a})|^2\\
=&\ \int_\Xi (U(\theta,y^*((\theta,\mathfrak{a}),x))-\mathfrak{a}|x-y^*((\theta,\mathfrak{a}),x)|^p)\ \dee \mu_0(x) +\frac{\min\{\eta_1, \eta_2\}}{2}|(\theta,\mathfrak{a})|^2\\
\geq&\ \int_\Xi (-|U(\theta,y^*((\theta,\mathfrak{a}),x))|-2^p M_\Xi^p |(\theta,\mathfrak{a})|)\ \dee \mu_0(x) +\frac{\min\{\eta_1, \eta_2\}}{2}|(\theta,\mathfrak{a})|^2\\
\geq&\ \int_\Xi (-\tilde{K}_\nabla  (1+|(\theta,\mathfrak{a})|) -2^p M_\Xi^p |(\theta,\mathfrak{a})|)\ \dee \mu_0(x) +\frac{\min\{\eta_1, \eta_2\}}{2}|(\theta,\mathfrak{a})|^2\\
\geq&\  -\tilde{K}_\nabla   (1+|(\theta,\mathfrak{a})|) -2^p M_\Xi^p |(\theta,\mathfrak{a})|+\frac{\min\{\eta_1, \eta_2\}}{2}|(\theta,\mathfrak{a})|^2,
\end{split}
\end{align}
where the inner supremum is attained at $y^*((\theta,\mathfrak{a}),x)\in\Xi$. This further implies, for any $|(\theta,\mathfrak{a})|>R_\mc{K}$ with $R_\mc{K}$ defined in \eqref{eq:rk_expression}, that
\begin{align}\label{eqn:end_arg_1}
	\begin{split}
		&\ \int_\Xi \sup_{y\in\Xi}(U(\theta,y)-\mathfrak{a}|x-y|^p)\ \dee \mu_0(x) +\frac{\eta_1}{2}|\theta|^2+\frac{\eta_2}{2}|\mathfrak{a}|^2\\
		\geq&\ -\tilde{K}_\nabla -\frac{(\tilde{K}_\nabla+2^pM_\Xi^p)^2}{\min\{\eta_1, \eta_2\}} +\frac{\min\{\eta_1, \eta_2\}}{4}|(\theta,\mathfrak{a})|^2\\
		\geq&\ -\tilde{K}_\nabla -\frac{(\tilde{K}_\nabla+2^pM_\Xi^p)^2}{\min\{\eta_1, \eta_2\}} +\frac{\min\{\eta_1, \eta_2\}}{4}\cdot\left(3\tilde{K}_\nabla + \frac{(\tilde{K}_\nabla+2^pM_\Xi^p)^2}{\min\{\eta_1, \eta_2\}}\right)\frac{4}{\min\{\eta_1, \eta_2\}}\\
		=&\ 2\tilde{K}_\nabla > \int_{\Xi}\sup_{y\in\Xi}\left\{U(0,y)\right\}\ \dee\mu_0(x)\geq z_D.
	\end{split}
\end{align}
The above lower bound shows that the infimum of $z_D$ is attained within $B_{R_\mc{K}}(0)$, i.e.,
\[
z_D = \inf_{(\theta,\mathfrak{a})\in B_{R_\mc{K}}(0)}\left\{\int_{\Xi}\sup_{y\in\Xi}\left\{U(\theta,y)-\mathfrak{a} |x-y|^p\right\}\ \dee\mu_0(x)+\frac{\eta_1}{2}|\theta|^{2}+\frac{\eta_2}{2}|\mathfrak{a}|^2\right\}.
\]
This proves the first part of the lemma. \\

To establish the second part of the lemma, we repeat again the argument in \eqref{eqn:start_arg_1} by applying Assumptions~\ref{assumption:1}, \ref{assumption:3}, and Remark \ref{remark:u_lipschitz} to obtain the same lower bound, namely that for every $(\theta,\mathfrak{a})\in\R^d\times [0,\infty)$,
\begin{align*}
&\ \int_{\Xi}\sup_{y\in\Xi}\left\{U(\theta,[y]_\mathfrak{j})-\mathfrak{a} |[x]_\mathfrak{j}-[y]_\mathfrak{j}|^p\right\}\ \dee\mu_0(x)+\frac{\eta_1}{2}|\theta|^{2}+\frac{\eta_2}{2}|\mathfrak{a}|^2\\
\geq&\ \int_{\Xi}\sup_{y\in\Xi}\left\{U(\theta,[y]_\mathfrak{j})-\mathfrak{a} |[x]_\mathfrak{j}-[y]_\mathfrak{j}|^p\right\}\ \dee\mu_0(x)+\frac{\min\{\eta_1, \eta_2\}}{2}|(\theta,\mathfrak{a})|^2\\
\geq&\ \int_\Xi (U(\theta,y^*_\mathfrak{j}((\theta,\mathfrak{a}),x))-\mathfrak{a}|[x]_\mathfrak{j}-y^*_\mathfrak{j}((\theta,\mathfrak{a}),x)|^p)\ \dee \mu_0(x) +\frac{\min\{\eta_1, \eta_2\}}{2}|(\theta,\mathfrak{a})|^2\\
\geq&\ \int_\Xi (-|U(\theta,y^*_\mathfrak{j}((\theta,\mathfrak{a}),x))|-2^p M_\Xi^p |(\theta,\mathfrak{a})|)\ \dee \mu_0(x) +\frac{\min\{\eta_1, \eta_2\}}{2}|(\theta,\mathfrak{a})|^2\\
\geq&\ \int_\Xi (-\tilde{K}_\nabla (1+|(\theta,\mathfrak{a})|) -2^p M_\Xi^p |(\theta,\mathfrak{a})|)\ \dee \mu_0(x) +\frac{\min\{\eta_1, \eta_2\}}{2}|(\theta,\mathfrak{a})|^2\\
\geq&\ -\tilde{K}_\nabla   (1+|(\theta,\mathfrak{a})|) -2^p M_\Xi^p |(\theta,\mathfrak{a})|+\frac{\min\{\eta_1, \eta_2\}}{2}|(\theta,\mathfrak{a})|^2,
\end{align*}
which does not depend on $\mathfrak{j}$. Here, $y^*_\mathfrak{j}((\theta,\mathfrak{a}),x)\in \Xi\cap \mathbb{K}_{\ell,\mathfrak{j}}^m$ denotes, for given $\mathfrak{j}$, an optimiser for the inner supremum. Finally, following the same argument as in \eqref{eqn:end_arg_1} establishes the second part of the lemma, i.e.,
\[
z_{D,\ell,\mathfrak{j}} = \inf_{(\theta,\mathfrak{a})\in B_{R_\mc{K}}(0)}\left\{\int_{\Xi}\sup_{y\in\Xi}\left\{U(\theta,[y]_\mathfrak{j})-\mathfrak{a} |[x]_\mathfrak{j}-[y]_\mathfrak{j}|^p\right\}\ \dee\mu_0(x)+\frac{\eta_1}{2}|\theta|^{2}+\frac{\eta_2}{2}|\mathfrak{a}|^2\right\}.
\]
This completes the proof.
\end{proof}

\subsection*{Proof of Proposition \ref{prop:quadrature}}
By Lemma \ref{lemma:compactness}, the infimums of $z_D$ and $z_{D,\ell,\mathfrak{j}}$ are both attained on $B_{R_\mc{K}}(0):=\{(\theta,\mathfrak{a})\in\R^d\times [0,\infty):|(\theta,\mathfrak{a})|\leq R_\mc{K} \}$ with $R_\mc{K}$ defined in \eqref{eq:rk_expression}.  Applying the inequality
\begin{align*}
\max\left\{\lbrb{\inf_{x\in A}f(x)-\inf_{x\in A}g(x)},\lbrb{\sup_{x\in A}f(x)-\sup_{x\in A}g(x)}\right\}\leq \sup_{x\in A}|f(x)-g(x)|
\end{align*}
for any functions $f,g$ and set $A$ contained within their domains yields 
\begin{align*}
&\ |z_D-z_{D,\ell,\mathfrak{j}}|\nonumber\\
=&\ \Bigg{|}\inf_{(\theta,\mathfrak{a})\in B_{R_\mc{K}}(0)}\left\{\int_{\Xi}\sup_{y\in\Xi}\left\{U(\theta,y)-\mathfrak{a} |x-y|^p\right\}\ \dee\mu_0(x)+\frac{\eta_1}{2}|\theta|^{2}+\frac{\eta_2}{2}|\mathfrak{a}|^2\right\}-\nonumber\\
&\ \ \inf_{(\theta,\mathfrak{a})\in B_{R_\mc{K}}(0)}\left\{\int_{\Xi}\sup_{y\in\Xi}\left\{U(\theta,[y]_\mathfrak{j})-\mathfrak{a} |[x]_\mathfrak{j}-[y]_\mathfrak{j}|^p\right\}\ \dee\mu_0(x)+\frac{\eta_1}{2}|\theta|^{2}+\frac{\eta_2}{2}|\mathfrak{a}|^2\right\}\Bigg{|}\nonumber\\
\leq&\ \sup_{(\theta,\mathfrak{a})\in B_{R_\mc{K}}(0)}\Bigg{|}\int_{\Xi}\sup_{y\in\Xi}\left\{U(\theta,y)-\mathfrak{a} |x-y|^p\right\}-\sup_{y\in\Xi}\left\{U(\theta,[y]_\mathfrak{j})-\mathfrak{a} |[x]_\mathfrak{j}-[y]_\mathfrak{j}|^p\right\}\ \dee\mu_0(x)\Bigg{|}\nonumber\\
\leq&\ \sup_{(\theta,\mathfrak{a})\in B_{R_\mc{K}}(0)}\int_{\Xi}\lbrb{\sup_{y\in\Xi}\left\{U(\theta,y)-\mathfrak{a} |x-y|^p\right\}-\sup_{y\in\Xi}\left\{U(\theta,[y]_\mathfrak{j})-\mathfrak{a} |[x]_\mathfrak{j}-[y]_\mathfrak{j}|^p\right\}}\ \dee\mu_0(x)\nonumber\\
\leq&\ \sup_{(\theta,\mathfrak{a})\in B_{R_\mc{K}}(0)}\int_{\Xi}\sup_{y\in\Xi}\lbrb{U(\theta,y)-U(\theta,[y]_\mathfrak{j})-\mathfrak{a} |x-y|^p+\mathfrak{a} |[x]_\mathfrak{j}-[y]_\mathfrak{j}|^p}\ \dee\mu_0(x),\nonumber
\end{align*}
which implies, by using Assumption~\ref{assumption:4}, that
\begin{align*}
\ |z_D-z_{D,\ell,\mathfrak{j}}| \leq&\ \sup_{(\theta,\mathfrak{a})\in B_{R_\mc{K}}(0)}\int_{\Xi} \Bigg{(}J_U(1+|(\theta,\mathfrak{a})|)\sup_{y\in\Xi} |y-[y]_\mathfrak{j}|+\nonumber\\
&\ \ p|(\theta,\mathfrak{a})|\sup_{y\in\Xi}(1+|x-y|+|[x]_\mathfrak{j}-[y]_\mathfrak{j}|)^{p-1}||x-y|-|[x]_\mathfrak{j}-[y]_\mathfrak{j}|\Bigg{)}\ \dee\mu_0(x)\nonumber\\
\leq&\ \sup_{(\theta,\mathfrak{a})\in B_{R_\mc{K}}(0)}\int_{\Xi} \Bigg{(}J_U (1+|(\theta,\mathfrak{a})|)\frac{\sqrt{m}}{2^\mathfrak{j}}+  p|(\theta,\mathfrak{a})|(1+4M_\Xi)^{p-1}\frac{2\sqrt{m}}{2^\mathfrak{j}}\Bigg{)}\ \dee\mu_0(x)\nonumber\\
\leq&\ \frac{\sqrt{m}(J_U + 2p(1+4M_\Xi)^{p-1})(1+R_\mc{K})}{2^\mathfrak{j}},
\end{align*}
as desired. 
\qed
\subsection*{Proof of Proposition \ref{prop:glob_lipschitz}}
\begin{proof}
Fix any $\delta>0$. For short notation, we denote by ${F}_j:= F_j^{\delta,\ell,\mathfrak{j}}$ with $F_j^{\delta,\ell,\mathfrak{j}}$ defined in \eqref{defn:grad_v_delta}, $j\in\{1,\cdots,N\}$. Then, by the definition of $V^{\delta,\ell,\mathfrak{j}}$ in \eqref{defn:v_delta}, we have, for any $\bar{\theta}=(\theta,\alpha)\in\R^d\times \R$ and $x\in\Xi$, that
\begin{align}
\nabla_{\theta} V^{\delta,\ell,\mathfrak{j}}(\bar{\theta},x)
=\ \tfrac{\sum^N_{j=1}{F}_j(\bar{\theta},x)\nabla_\theta U(\theta, \xi_j)}{\sum^N_{j=1}{F}_j(\bar{\theta}, x)},\qquad
\nabla_{\alpha} V^{\delta,\ell,\mathfrak{j}}(\bar{\theta},x)
=-\tfrac{\sum^N_{j=1}{F}_j(\bar{\theta},x)\iota'(\alpha)|x-\xi_j|^p}{\sum^N_{j=1}{F}_j(\bar{\theta}, x)}\label{eqn:grad_v_delta}.
\end{align}
For all $\bar{\theta}_1, \bar{\theta}_2\in\R^{d+1}$ and $x\in\Xi$, it holds that
\begin{align*}
&\ \lbrb{\nabla_{\theta} V^{\delta,\ell,\mathfrak{j}}(\bar{\theta}_1,x)-\nabla_{\theta} V^{\delta,\ell,\mathfrak{j}}(\bar{\theta}_2,x)}\nonumber\\
=&\ \lbrb{\tfrac{\sum^N_{j=1}{F}_j(\bar{\theta}_1,x)\nabla_\theta U(\theta_1, \xi_j)}{\sum^N_{j=1}{F}_j(\bar{\theta}_1, x)}-\tfrac{\sum^N_{j=1}{F}_j(\bar{\theta}_2,x)\nabla_\theta U(\theta_2, \xi_j)}{\sum^N_{j=1}{F}_j(\bar{\theta}_2, x)}}\nonumber\\
\leq&\ \tfrac{\lbrb{\sum^N_{j,k=1}{F}_j(\bar{\theta}_1,x){F}_k(\bar{\theta}_2,x)\left(\nabla_\theta U(\theta_1, \xi_j)-\nabla_\theta U(\theta_2, \xi_k)\right)}}{\sum^N_{j,k=1}{F}_j(\bar{\theta}_1,x){F}_k(\bar{\theta}_2,x)}\nonumber\\
=&\ \tfrac{1}{\sum^N_{j,k=1}{F}_j(\bar{\theta}_1,x){F}_k(\bar{\theta}_2,x)}\Bigg{|}\sum^N_{j=1}{F}_j(\bar{\theta}_1,x){F}_j(\bar{\theta}_2,x)\left(\nabla_\theta U(\theta_1, \xi_j)-\nabla_\theta U(\theta_2, \xi_j)\right)\nonumber\\
&\ + \sum_{1\leq j< k\leq N}{F}_j(\bar{\theta}_1,x){F}_k(\bar{\theta}_2,x)\left(\nabla_\theta U(\theta_1, \xi_j)-\nabla_\theta U(\theta_2, \xi_k)\right)\nonumber\\
&\ +  \sum_{1\leq j< k\leq N}{F}_j(\bar{\theta}_2,x){F}_k(\bar{\theta}_1,x)\left(\nabla_\theta U(\theta_1, \xi_k)-\nabla_\theta U(\theta_2, \xi_j)\right)\Bigg{|},
\end{align*}
which further implies,
\begin{align}
	&\ \lbrb{\nabla_{\theta} V^{\delta,\ell,\mathfrak{j}}(\bar{\theta}_1,x)-\nabla_{\theta} V^{\delta,\ell,\mathfrak{j}}(\bar{\theta}_2,x)}\nonumber\\
\begin{split}\label{eqn:polylip_split}
\leq &\ \tfrac{1}{\sum^N_{j,k=1}{F}_j(\bar{\theta}_1,x){F}_k(\bar{\theta}_2,x)}\Bigg{(}\sum^N_{j=1}{F}_j(\bar{\theta}_1,x){F}_j(\bar{\theta}_2,x)\left|\nabla_\theta U(\theta_1, \xi_j)-\nabla_\theta U(\theta_2, \xi_j)\right|\\
&\ + \sum_{1\leq j< k\leq N}\left|{F}_j(\bar{\theta}_1,x){F}_k(\bar{\theta}_2,x)-{F}_j(\bar{\theta}_2,x){F}_k(\bar{\theta}_1,x)\right|\cdot\left|\nabla_\theta U(\theta_1, \xi_j)\right|\\
&\ + \sum_{1\leq j< k\leq N}{F}_j(\bar{\theta}_2,x){F}_k(\bar{\theta}_1,x)\cdot \left|\nabla_\theta U(\theta_1, \xi_j)-\nabla_\theta U(\theta_2, \xi_j)\right|\\
&\ +  \sum_{1\leq j< k\leq N}{F}_j(\bar{\theta}_2,x){F}_k(\bar{\theta}_1,x)\cdot \left|\nabla_\theta U(\theta_1, \xi_k)-\nabla_\theta U(\theta_2, \xi_k)\right|\\
&\ +  \sum_{1\leq j< k\leq N}\left|{F}_j(\bar{\theta}_2,x){F}_k(\bar{\theta}_1,x)-{F}_j(\bar{\theta}_1,x){F}_k(\bar{\theta}_2,x)\right|\cdot \left|\nabla_\theta U(\theta_2, \xi_k)\right|\Bigg{)}.
\end{split}
\end{align}
Let $j,k$ be such that $1\leq j<k\leq N$. For $\bar{\theta}_1,\bar{\theta}_2\in\R^{d+1}$ and $x\in\Xi$ such that ${F}_j(\bar{\theta}_1,x){F}_k(\bar{\theta}_2,x)\geq {F}_j(\bar{\theta}_2,x){F}_k(\bar{\theta}_1,x)$, one obtains, by the inequality $1-e^{-y}\leq y$, $y\geq 0$, that
\begin{align}
&\ \tfrac{\left|{F}_j(\bar{\theta}_1,x){F}_k(\bar{\theta}_2,x)-{F}_j(\bar{\theta}_2,x){F}_k(\bar{\theta}_1,x)\right|}{\sum^N_{j',k'=1}{F}_{j'}(\bar{\theta}_1,x){F}_{k'}(\bar{\theta}_2,x)} \cdot\left|\nabla_\theta U(\theta_1, \xi_j)\right|\nonumber\\
=&\ \tfrac{{F}_j(\bar{\theta}_1,x){F}_k(\bar{\theta}_2,x)}{\sum^N_{j',k'=1}{F}_{j'}(\bar{\theta}_1,x){F}_{k'}(\bar{\theta}_2,x)}\cdot \left(1-\tfrac{{F}_j(\bar{\theta}_2,x){F}_k(\bar{\theta}_1,x)}{{F}_j(\bar{\theta}_1,x){F}_k(\bar{\theta}_2,x)}\right)\cdot\left|\nabla_\theta U(\theta_1, \xi_j)\right|\nonumber\\
=&\  \tfrac{{F}_j(\bar{\theta}_1,x){F}_k(\bar{\theta}_2,x)}{\sum^N_{j',k'=1}{F}_{j'}(\bar{\theta}_1,x){F}_{k'}(\bar{\theta}_2,x)}\cdot\Bigg{(}1-\exp\Bigg{[}\tfrac{1}{\delta}\Bigg{(}(U(\theta_2,\xi_j)-U(\theta_1,\xi_j))-(\iota(\alpha_2)-\iota(\alpha_1))|x-\xi_j|^p\nonumber\\
&\ +(U(\theta_1,\xi_k)-U(\theta_2,\xi_k))-(\iota(\alpha_1)-\iota(\alpha_2))|x-\xi_k|^p\Bigg{)}\Bigg{]}\Bigg{)}\cdot\left|\nabla_\theta U(\theta_1, \xi_j)\right|\nonumber\\
\leq&\  \tfrac{{F}_j(\bar{\theta}_1,x){F}_k(\bar{\theta}_2,x)}{\sum^N_{j',k'=1}{F}_{j'}(\bar{\theta}_1,x){F}_{k'}(\bar{\theta}_2,x)}\cdot\tfrac{1}{\delta}\Bigg{[}\lbrb{U(\theta_1,\xi_j)-U(\theta_2,\xi_j)}+\lbrb{U(\theta_1,\xi_k)-U(\theta_2,\xi_k)}\nonumber\\
&\ +\lbrb{\iota(\alpha_1)-\iota(\alpha_2)}(|x-\xi_j|^p+|x-\xi_k|^p)\Bigg{]}\cdot\left|\nabla_\theta U(\theta_1, \xi_j)\right|.\nonumber
\end{align}
Interchanging the roles of $\bar{\theta}_1$ and $\bar{\theta}_2$ in the above argument shows that for all $\bar{\theta}_1,\bar{\theta}_2\in\R^{d+1}$ and $x\in\Xi$,
\begin{align}
&\ \tfrac{\left|{F}_j(\bar{\theta}_1,x){F}_k(\bar{\theta}_2,x)-{F}_j(\bar{\theta}_2,x){F}_k(\bar{\theta}_1,x)\right|}{\sum^N_{j',k'=1}{F}_{j'}(\bar{\theta}_1,x){F}_{k'}(\bar{\theta}_2,x)} \cdot\left|\nabla_\theta U(\theta_1, \xi_j)\right|\nonumber\\
\begin{split}\label{eqn:polylip_term_b}
	\leq&\  \tfrac{\max\{{F}_j(\bar{\theta}_1,x){F}_k(\bar{\theta}_2,x),{F}_j(\bar{\theta}_2,x){F}_k(\bar{\theta}_1,x)\}}{\sum^N_{j',k'=1}{F}_{j'}(\bar{\theta}_1,x){F}_{k'}(\bar{\theta}_2,x)}\cdot\tfrac{1}{\delta}\Bigg{[}\lbrb{U(\theta_1,\xi_j)-U(\theta_2,\xi_j)}\\
	&\ +\lbrb{U(\theta_1,\xi_k)-U(\theta_2,\xi_k)}+\lbrb{\iota(\alpha_1)-\iota(\alpha_2)}(|x-\xi_j|^p+|x-\xi_k|^p)\Bigg{]}\cdot\left|\nabla_\theta U(\theta_1, \xi_j)\right|,
\end{split}
\end{align}

and furthermore, 
\begin{align}
&\ \tfrac{\left|{F}_j(\bar{\theta}_2,x){F}_k(\bar{\theta}_1,x)-{F}_j(\bar{\theta}_1,x){F}_k(\bar{\theta}_2,x)\right|}{\sum^N_{j',k'=1}{F}_{j'}(\bar{\theta}_1,x){F}_{k'}(\bar{\theta}_2,x)} \cdot\left|\nabla_\theta U(\theta_2, \xi_j)\right|\nonumber\\
\begin{split}\label{eqn:polylip_term_e}
\leq&\  \tfrac{\max\{{F}_j(\bar{\theta}_1,x){F}_k(\bar{\theta}_2,x),{F}_j(\bar{\theta}_2,x){F}_k(\bar{\theta}_1,x)\}}{\sum^N_{j',k'=1}{F}_{j'}(\bar{\theta}_1,x){F}_{k'}(\bar{\theta}_2,x)}\cdot\tfrac{1}{\delta}\Bigg{[}\lbrb{U(\theta_1,\xi_j)-U(\theta_2,\xi_j)}\\
&\ +\lbrb{U(\theta_1,\xi_k)-U(\theta_2,\xi_k)}+\lbrb{\iota(\alpha_1)-\iota(\alpha_2)}(|x-\xi_j|^p+|x-\xi_k|^p)\Bigg{]}\cdot\left|\nabla_\theta U(\theta_2, \xi_k)\right|.
\end{split}
\end{align}

By Assumption \ref{assumption:3}, Remark \ref{remark:u_lipschitz}, and the fact that $\iota'$ is bounded by $M_\iota$, it holds that
\begin{align}
&\ \tfrac{1}{\delta}\Bigg{[}\lbrb{U(\theta_1,\xi_j)-U(\theta_2,\xi_j)}+\lbrb{U(\theta_1,\xi_k)-U(\theta_2,\xi_k)} +\lbrb{\iota(\alpha_1)-\iota(\alpha_2)}(|x-\xi_j|^p+|x-\xi_k|^p)\Bigg{]}\nonumber\\
&\ \cdot\left(\left|\nabla_\theta U(\theta_1, \xi_j)\right|+\left|\nabla_\theta U(\theta_2, \xi_k)\right|\right)\nonumber\\
\leq&\ \tfrac{1}{\delta}\Bigg{[}2K_\nabla  |\theta_1-\theta_2|+2^{p+1}M_\Xi^p M_\iota |\alpha_1-\alpha_2|\Bigg{]} \cdot 2K_\nabla \nonumber\\
\leq&\ \tfrac{4K_\nabla \left(K_\nabla + 2^{p} M_\Xi^p M_\iota\right)}{\delta}\cdot  |\bar{\theta}_1-\bar{\theta}_2|.\label{eqn:polylip_term_aux}
\end{align}
That is, combining \eqref{eqn:polylip_term_b}, \eqref{eqn:polylip_term_e}, and \eqref{eqn:polylip_term_aux}, then summing over $1\leq j<k\leq N$ yields
\begin{align}
&\tfrac{\sum_{1\leq j<k\leq N}\left|{F}_j(\bar{\theta}_1,x){F}_k(\bar{\theta}_2,x)-{F}_j(\bar{\theta}_2,x){F}_k(\bar{\theta}_1,x)\right|\cdot\left|\nabla_\theta U(\theta_1, \xi_j)\right| }{\sum^N_{j',k'=1}{F}_{j'}(\bar{\theta}_1,x){F}_{k'}(\bar{\theta}_2,x)}\nonumber\\
&\ + \tfrac{\sum_{1\leq j<k\leq N}\left|{F}_j(\bar{\theta}_2,x){F}_k(\bar{\theta}_1,x)-{F}_j(\bar{\theta}_1,x){F}_k(\bar{\theta}_2,x)\right|\cdot\left|\nabla_\theta U(\theta_2, \xi_k)\right|}{\sum^N_{j',k'=1}{F}_{j'}(\bar{\theta}_1,x){F}_{k'}(\bar{\theta}_2,x)}\nonumber\\
\leq&\ \tfrac{4K_\nabla \left(K_\nabla + 2^{p} M_\Xi^p M_\iota\right)}{\delta}\cdot  |\bar{\theta}_1-\bar{\theta}_2|  \cdot \tfrac{\sum_{1\leq j<k\leq N}\max\{{F}_j(\bar{\theta}_1,x){F}_k(\bar{\theta}_2,x),{F}_j(\bar{\theta}_2,x){F}_k(\bar{\theta}_1,x)\}}{\sum^N_{j',k'=1}{F}_{j'}(\bar{\theta}_1,x){F}_{k'}(\bar{\theta}_2,x)}\nonumber\\
\leq &\ \tfrac{8K_\nabla \left(K_\nabla + 2^{p} M_\Xi^p M_\iota\right)}{\delta}\cdot  |\bar{\theta}_1-\bar{\theta}_2|.\label{eqn:polylip_term_be}
\end{align}
In addition, by Assumption \ref{assumption:3}, for all $j,k\in\{1,\cdots,N\}$, $\bar{\theta}_1,\bar{\theta}_2\in\R^{d+1}$ and $x\in\Xi$, it holds that
\begin{align}
&\ \tfrac{{F}_j(\bar{\theta}_2,x){F}_k(\bar{\theta}_1,x)\cdot \left|\nabla_\theta U(\theta_1, \xi_j)-\nabla_\theta U(\theta_2, \xi_j)\right|}{\sum^N_{j',k'=1}{F}_{j'}(\bar{\theta}_1,x){F}_{k'}(\bar{\theta}_2,x)}  
\leq \tfrac{{F}_j(\bar{\theta}_2,x){F}_k(\bar{\theta}_1,x)}{\sum^N_{j',k'=1}{F}_{j'}(\bar{\theta}_1,x){F}_{k'}(\bar{\theta}_2,x)}\cdot L_\nabla  |\theta_1-\theta_2.|\nonumber
\end{align}
This implies that
\begin{align}
&\ \tfrac{\sum^N_{j=1}{F}_j(\bar{\theta}_1,x){F}_j(\bar{\theta}_2,x)\cdot \left|\nabla_\theta U(\theta_1, \xi_j)-\nabla_\theta U(\theta_2, \xi_j)\right|}{\sum^N_{j',k'=1}{F}_{j'}(\bar{\theta}_1,x){F}_{k'}(\bar{\theta}_2,x)}+\tfrac{\sum_{1\leq j<k<N}{F}_j(\bar{\theta}_2,x){F}_k(\bar{\theta}_1,x)\cdot \left|\nabla_\theta U(\theta_1, \xi_j)-\nabla_\theta U(\theta_2, \xi_j)\right|}{\sum^N_{j',k'=1}{F}_{j'}(\bar{\theta}_1,x){F}_{k'}(\bar{\theta}_2,x)}\nonumber\\
&+\ \tfrac{\sum_{1\leq j<k<N}{F}_j(\bar{\theta}_2,x){F}_k(\bar{\theta}_1,x)\cdot \left|\nabla_\theta U(\theta_1, \xi_k)-\nabla_\theta U(\theta_2, \xi_k)\right|}{\sum^N_{j',k'=1}{F}_{j'}(\bar{\theta}_1,x){F}_{k'}(\bar{\theta}_2,x)}\nonumber\\
\leq&\ 2L_\nabla  |\bar{\theta}_1-\bar{\theta}_2|\label{eqn:polylip_term_acd}.
\end{align}
Therefore, substituting \eqref{eqn:polylip_term_be} and \eqref{eqn:polylip_term_acd} into \eqref{eqn:polylip_split} yields, for all $\bar{\theta}_1,\bar{\theta}_2\in\R^{d+1}$ and $x\in\Xi$,
\begin{align}
&\ \lbrb{\nabla_{\theta} V^{\delta,\ell,\mathfrak{j}}(\bar{\theta}_1,x)-\nabla_{\theta} V^{\delta,\ell,\mathfrak{j}}(\bar{\theta}_2,x)}
\leq  2 \left(\tfrac{4K_\nabla \left(K_\nabla + 2^{p} M_\Xi^p M_\iota\right)}{\delta} +L_\nabla \right)\cdot |\bar{\theta}_1-\bar{\theta}_2|.\label{eqn:polylip_theta}
\end{align}

By a similar argument as in \eqref{eqn:polylip_split}, it holds for all $\bar{\theta}_1, \bar{\theta}_2\in\R^{d+1}$ and $x\in\Xi$ that
\begin{align}
&\ \lbrb{\nabla_{\alpha} V^{\delta,\ell,\mathfrak{j}}(\bar{\theta}_1,x)-\nabla_{\alpha} V^{\delta,\ell,\mathfrak{j}}(\bar{\theta}_2,x)}\nonumber\\
\leq&\ \lbrb{\tfrac{\sum^N_{j=1}{F}_j(\bar{\theta}_1,x)\iota'(\alpha_1)|x-\xi_j|^p}{\sum^N_{j=1}{F}_j(\bar{\theta}_1, x)}-\tfrac{\sum^N_{j=1}{F}_j(\bar{\theta}_2,x)\iota'(\alpha_2)|x-\xi_j|^p}{\sum^N_{j=1}{F}_j(\bar{\theta}_2, x)}}\nonumber\\
\leq &\ \tfrac{1}{\sum^N_{j,k=1}{F}_j(\bar{\theta}_1,x){F}_k(\bar{\theta}_2,x)}\Bigg{(}\sum^N_{j=1}{F}_j(\bar{\theta}_1,x){F}_j(\bar{\theta}_2,x)\cdot|\iota'(\alpha_1)-\iota'(\alpha_2)|\cdot|x-\xi_j|^p\nonumber\\
&\ + \sum_{1\leq j< k\leq N}\left|{F}_j(\bar{\theta}_1,x){F}_k(\bar{\theta}_2,x)-{F}_j(\bar{\theta}_2,x){F}_k(\bar{\theta}_1,x)\right|\cdot|\iota'(\alpha_1)|\cdot |x-\xi_j|^p\nonumber\\
&\ + \sum_{1\leq j< k\leq N}{F}_j(\bar{\theta}_2,x){F}_k(\bar{\theta}_1,x)\cdot|\iota'(\alpha_1)-\iota'(\alpha_2)|\cdot|x-\xi_j|^p\nonumber\\
&\ +  \sum_{1\leq j< k\leq N}{F}_j(\bar{\theta}_2,x){F}_k(\bar{\theta}_1,x)\cdot |\iota'(\alpha_1)-\iota'(\alpha_2)|\cdot|x-\xi_k|^p\nonumber\\
&\ +  \sum_{1\leq j< k\leq N}\left|{F}_j(\bar{\theta}_2,x){F}_k(\bar{\theta}_1,x)-{F}_j(\bar{\theta}_1,x){F}_k(\bar{\theta}_2,x)\right| \cdot|\iota'(\alpha_2)|\cdot|x-\xi_k|^p\Bigg{)},\nonumber
\end{align}
which implies, by applying Remark~\ref{remark:iota_dissipativity}, that
\begin{align}
&\ \lbrb{\nabla_{\alpha} V^{\delta,\ell,\mathfrak{j}}(\bar{\theta}_1,x)-\nabla_{\alpha} V^{\delta,\ell,\mathfrak{j}}(\bar{\theta}_2,x)}\nonumber\\
\leq &\ \tfrac{2^{p}M_\Xi^p}{\sum^N_{j,k=1}{F}_j(\bar{\theta}_1,x){F}_k(\bar{\theta}_2,x)}\Bigg{(}\sum^N_{j=1}{F}_j(\bar{\theta}_1,x){F}_j(\bar{\theta}_2,x)\cdot L_\iota\lbrb{\alpha_1-\alpha_2}\nonumber\\
&\ + \sum_{1\leq j< k\leq N}\left|{F}_j(\bar{\theta}_1,x){F}_k(\bar{\theta}_2,x)-{F}_j(\bar{\theta}_2,x){F}_k(\bar{\theta}_1,x)\right|\cdot M_\iota\nonumber\\
&\ + \sum_{1\leq j< k\leq N}{F}_j(\bar{\theta}_2,x){F}_k(\bar{\theta}_1,x)\cdot L_\iota\lbrb{\alpha_1-\alpha_2} \nonumber\\
&\ +  \sum_{1\leq j< k\leq N}{F}_j(\bar{\theta}_2,x){F}_k(\bar{\theta}_1,x)\cdot L_\iota\lbrb{\alpha_1-\alpha_2}\nonumber\\
&\ +  \sum_{1\leq j< k\leq N}\left|{F}_j(\bar{\theta}_2,x){F}_k(\bar{\theta}_1,x)-{F}_j(\bar{\theta}_1,x){F}_k(\bar{\theta}_2,x)\right| \cdot M_\iota\Bigg{)}\nonumber\\
\leq &\ 2^{p}M_\Xi^p\Bigg{(}2L_\iota |\alpha_1-\alpha_2|+\tfrac{8M_\iota\left(K_\nabla +2^{p}M_\Xi^pM_\iota\right)}{\delta} |\bar{\theta}_1-\bar{\theta}_2|\Bigg{)}\nonumber\\
\leq&\ \left(2^{p+1} L_\iota M_\Xi^p+ \tfrac{2^{p+3} M_\iota M_\Xi^p \left(K_\nabla +2^{p}M_\Xi^pM_\iota\right)}{\delta}\right) \cdot |\bar{\theta}_1-\bar{\theta}_2|,\label{eqn:polylip_alpha}
\end{align}
where the second last inequality is obtained using the same arguments as in \eqref{eqn:polylip_term_b}-\eqref{eqn:polylip_term_be}. Combining \eqref{eqn:polylip_theta} and \eqref{eqn:polylip_alpha} thus yields
\begin{align*}
|\nabla_{\bar{\theta}}V^{\delta,\ell,\mathfrak{j}}(\bar{\theta}_1,x)-\nabla_{\bar{\theta}}V^{\delta,\ell,\mathfrak{j}}(\bar{\theta}_2,x)|\leq L_\delta |\bar{\theta}_1-\bar{\theta}_2|,
\end{align*}
for all $\bar{\theta}_1, \bar{\theta}_2\in\R^{d+1}$ and $x\in\Xi$, where 
\begin{align*}
L_\delta
:=&\ 2 \left(\tfrac{4K_\nabla \left(K_\nabla + 2^{p} M_\Xi^p M_\iota\right)}{\delta} +L_\nabla \right) +\left(2^{p+1} L_\iota M_\Xi^p+ \tfrac{2^{p+3} M_\iota M_\Xi^p \left(K_\nabla +2^{p}M_\Xi^pM_\iota\right)}{\delta}\right).
\end{align*}
This completes the proof.
\end{proof}

\subsection*{Proof of Proposition \ref{prop:dissipativity}}
\begin{proof}
Recall the expressions for $\nabla_{\bar{\theta}}V^{\delta,\ell,\mathfrak{j}}$ given in \eqref{eqn:grad_v_delta}. From Assumption \ref{assumption:3}, we derive the growth condition
\begin{align}
|\nabla_{\bar{\theta}}V^{\delta,\ell,\mathfrak{j}}(\bar{\theta},x)|
\leq\ &\ |\nabla_{\theta}V^{\delta,\ell,\mathfrak{j}}(\bar{\theta},x)|+|\nabla_{\alpha}V^{\delta,\ell,\mathfrak{j}}(\bar{\theta},x)|\nonumber\\
=&\ \frac{\sum^N_{j=1}F_j^{\delta,\ell,\mathfrak{j}}(\bar{\theta},x)\left(|\nabla_\theta U(\theta,\xi_j)|+|\iota'(\alpha)|\cdot|x-\xi_j|^p\right)}{\sum^N_{j=1}F_j^{\delta,\ell,\mathfrak{j}}(\bar{\theta},x)}\nonumber\\
\leq&\ \frac{\sum^N_{j=1}F_j^{\delta,\ell,\mathfrak{j}}(\bar{\theta},x)\left(K_\nabla +2^pM_\iota  M_\Xi^p\right)}{\sum^N_{j=1}F_j^{\delta,\ell,\mathfrak{j}}(\bar{\theta},x)}\nonumber\\
=&\ K_\nabla +2^pM_\iota  M_\Xi^p\label{eqn:grad_v_growth}
\end{align}
which holds for all $\bar{\theta}\in\R^{d+1}$ and $x\in\Xi$. Hence, it follows from \eqref{eqn:dissipativity_cond} that
\begin{align*}
&\lara{\bar{\theta}, \nabla_{\bar{\theta}}\left(V^{\delta,\ell,\mathfrak{j}}(\bar{\theta}, x) + \frac{\eta_1}{2}|\theta|^2+\frac{\eta_2}{2}|\iota(\alpha)|^2\right)}\nonumber\\
\geq&\ - |\bar{\theta}|\cdot \lbrb{\nabla_{\bar{\theta}}V^{\delta,\ell,\mathfrak{j}}(\bar{\theta},x)} + \eta_1 |\theta|^2 + \eta_2\alpha \iota(\alpha)\iota'(\alpha)\nonumber\\
\geq&\ - |\bar{\theta}|\cdot \lbrb{\nabla_{\bar{\theta}}V^{\delta,\ell,\mathfrak{j}}(\bar{\theta},x)} + \eta_1 |\theta|^2 + \eta_2 a_\iota |\alpha|^2 - \eta_2b_\iota\nonumber\\
\geq&\ - |\bar{\theta}|\cdot \lbrb{\nabla_{\bar{\theta}}V^{\delta,\ell,\mathfrak{j}}(\bar{\theta},x)} + \min\{\eta_1,\eta_2 a_\iota\} |\bar{\theta}|^2 - \eta_2 b_\iota\nonumber\\
\geq&\ - \left(K_\nabla +2^pM_\iota  M_\Xi^p\right)|\bar{\theta}|+ \min\{\eta_1,\eta_2 a_\iota\} |\bar{\theta}|^2 - \eta_2 b_\iota\nonumber\\
\geq&\ a|\bar{\theta}|^2-b,
\end{align*}
with $a:=\frac{\min\{\eta_1,\eta_2 a_\iota\}}{2}$ and $b:=\eta_2b_\iota + \frac{2\left(K_\nabla +2^pM_\iota  M_\Xi^p\right)^2}{\min\{\eta_1,\eta_2 a_\iota\}}$. This completes the proof.
\end{proof}

\subsection*{Proof of Proposition \ref{prop:excess_risk_discrete}}
\begin{proof}
Clearly, Assumption 1 of \cite{zhang2023nonasymptotic} holds due to  $\bar{\theta}_0\in L^4(\Omega, \mc{F},\mathbb{P};\R^{d+1})$ and the finiteness of the space $\Xi\cap\mathbb{K}_{\ell,\mathfrak{j}}^m$ which allows for exchange of order of differentiation and Lebesgue integration. Assumption 2 of \cite{zhang2023nonasymptotic} holds for the stochastic gradient $\nabla_{\bar{\theta}}\tilde{V}^{\delta.\ell,\mathfrak{j}}(\bar{\theta},x)$ due to Proposition \ref{prop:glob_lipschitz} and $\iota\cdot\iota'$ being Lipschitz continuous. Specifically, we have for all $\bar{\theta}_1=(\theta_1,\alpha_2),\bar{\theta}_2=(\theta_2,\alpha_2)\in\R^{d}\times \R$ and $x\in\Xi$,
\begin{align*}
&\ |\nabla_{\bar{\theta}}\tilde{V}^{\delta.\ell,\mathfrak{j}}(\bar{\theta}_1,x)-\nabla_{\bar{\theta}}\tilde{V}^{\delta.\ell,\mathfrak{j}}(\bar{\theta}_2,x)|\nonumber\\
\leq&\ |\nabla_{\bar{\theta}}V^{\delta.\ell,\mathfrak{j}}(\bar{\theta}_1,x)-\nabla_{\bar{\theta}}V^{\delta.\ell,\mathfrak{j}}(\bar{\theta}_2,x)| + \eta_1|\theta_1-\theta_2| + \eta_2|\iota(\alpha_1)\iota'(\alpha_1)-\iota(\alpha_2)\iota'(\alpha_2)|\nonumber\\
\leq&\ L_\delta |\bar{\theta}_1-\bar{\theta}_2| + \eta_1|\theta_1-\theta_2| + \eta_2\tilde{L}_\iota|\alpha_1-\alpha_2|\nonumber\\
\leq&\ (L_\delta+\eta_1+\eta_2\tilde{L}_\iota)|\bar{\theta}_1-\bar{\theta}_2|,
\end{align*}

and Assumption 2 of \cite{zhang2023nonasymptotic} with the correspondence of quantities between that in \cite{zhang2023nonasymptotic} and those in this paper being
\begin{align}
d\leftarrow d+1,\qquad \theta \leftarrow\ \bar{\theta},\qquad
H \leftarrow\ \nabla_{\bar{\theta}}\tilde{V}^{\delta.\ell,\mathfrak{j}},\qquad 
L_1 \leftarrow\ L_\delta+\eta_1+\eta_2\tilde{L}_\iota,\qquad
\eta(x) \leftarrow&\ 1,\label{eqn:assignments}
\end{align}

where the LHS of the above assignments are in the notation of \cite{zhang2023nonasymptotic}. Furthermore, Assumption 3 of \cite{zhang2023nonasymptotic} holds for the stochastic gradient $\nabla_{\bar{\theta}}\tilde{V}^{\delta,\ell,\mathfrak{j}}(\bar{\theta},x)$ due to Proposition \ref{prop:dissipativity}, and the constants $a,b$ in Remark 2.2 of \cite{zhang2023nonasymptotic} correspond exactly to the constants $a,b$ in Proposition \ref{prop:dissipativity} of this paper. One may obtain, from Equation (7) of \cite{zhang2023nonasymptotic}, the maximum step size restriction
\begin{align}
\lambda_{\max,\delta} = \min\left\{\frac{\mathfrak{C}_1}{\tilde{L}_\delta^2},\frac{1}{a}\right\}\label{defn:step}
\end{align}

for the algorithm, where the constants $a$, $\mathfrak{C}_1$  and $\tilde{L}_\delta := 1+L_\delta + \eta_1+\eta_2\tilde{L}_\iota$ are given explicitly as
\begin{align}\label{const:tilde_l_delta}
\begin{split}
\mathfrak{C}_1
:=&\ \frac{\min\{a,a^{1/3}\}}{64},
\qquad a:= \frac{\min\{\eta_1,\eta_2a_\iota\}}{2},
\qquad \tilde{L}_\delta:= \frac{\mathfrak{C}_2}{\delta} + \mathfrak{C}_3,\\
\mathfrak{C}_2
:=&\ (8K_\nabla +2^{p+3} M_\iota M_\Xi^p)(K_\nabla  + 2^{p}M_\Xi^pM_\iota),\\
\mathfrak{C}_3
:=&\ 2L_\nabla +2^{p+1} L_\iota M_\Xi^p + \eta_1+\eta_2\tilde{L}_\iota + 1.
\end{split}
\end{align}

(We note that the second condition in Assumption 2 of \cite{zhang2023nonasymptotic} was imposed by the authors to obtain sharper bounds for the Lipschitz constants. However, this second condition is not mandatory for the convergence bounds on the SGLD algorithm to hold, thus we do not verify it here.) Therefore, one may apply Corollary 2.8 of \cite{zhang2023nonasymptotic} with $\nabla_{\bar{\theta}}\tilde{V}^\delta(\bar{\theta},x)$ as the stochastic gradient to obtain constants $c_{\delta,\beta},C_{1,\delta,\ell,\mathfrak{j},\beta},C_{2,\delta,\ell,\mathfrak{j},\beta},C_{3,\delta,\beta}>0$ not depending on $n$ or $\lambda$ and with growth orders as specified in \eqref{eqn:constants_order} such that 
\begin{align}
\E_\mathbb{P}\left[v^{\delta,\ell,\mathfrak{j}}(\hat{\bar{\theta}}^{\lambda, \delta,\ell,\mathfrak{j}}_n)\right]-\inf_{\bar{\theta}\in\R^{d+1}} v^{\delta,\ell,\mathfrak{j}}(\bar{\theta}) \leq C_{1,\delta,\ell,\mathfrak{j},\beta} e^{-c_{\delta,\beta}\lambda n/4} + C_{2,\delta,\ell,\mathfrak{j},\beta} \lambda^{1/4} + C_{3,\delta,\beta},\label{eqn:sgld_bound}
\end{align}

where $v^{\delta,\ell,\mathfrak{j}}$ is defined as in \eqref{defn:small_v_delta_discrete}. Note that $c_{\delta,\beta},C_{1,\delta,\ell,\mathfrak{j},\beta},C_{2,\delta,\ell,\mathfrak{j},\beta},C_{3,\delta,\beta}>0$ correspond to $\dot{c}$, $C_1^\#$, $C_2^\#$, $C_3^\#$ of Corollary 2.8 of \cite{zhang2023nonasymptotic}, respectively, and that $c_{\delta,\beta},C_{3,\delta,\beta}$ do not depend on $\ell$ and $\mathfrak{j}$. This completes the proof.
\end{proof}

\subsection*{Proof of Corollary \ref{corollary:excess_risk_discrete}}
\begin{proof}
Applying the duality result in \eqref{eqn:dro_problem_dual_transformed_discrete}  twice yields
\begin{align*}
\E_{\mathbb{P}}[v^{\ell,\mathfrak{j}}(\hat{\bar{\theta}}^{\lambda, \delta, \ell,\mathfrak{j}}_n)]
=&\ \E_{\mathbb{P}}\lsrs{\left.\left(\int_{\Xi\cap\mathbb{K}_{\ell,\mathfrak{j}}^m}\tilde{V}^{\ell,\mathfrak{j}}(\bar{\theta}, x)\ \dee\mu_{0,\ell,\mathfrak{j}}(x)\right)\right|_{\bar{\theta}=\hat{\bar{\theta}}^{\lambda, \delta, \ell,\mathfrak{j}}_n=(\hat{{\theta}}^{\lambda, \delta, \ell,\mathfrak{j}}_n,\hat{\alpha}^{\lambda, \delta, \ell,\mathfrak{j}}_n)}}\nonumber\\
\geq &\ \E_{\mathbb{P}}\lsrs{\left.\left(\inf_{\alpha\in\R}\int_{\Xi\cap\mathbb{K}_{\ell,\mathfrak{j}}^m}\tilde{V}^{\ell,\mathfrak{j}}(\bar{\theta}, x)\ \dee\mu_{0,\ell,\mathfrak{j}}(x)\right)\right|_{\theta=\hat{\theta}^{\lambda, \delta, \ell,\mathfrak{j}}_n}}\nonumber\\
=&\ \E_{\mathbb{P}}[u^{\ell,\mathfrak{j}}(\hat{{\theta}}^{\lambda, \delta, \ell,\mathfrak{j}}_n)]
\geq\ \inf_{\theta\in\R^d}u^{\ell,\mathfrak{j}}(\theta)
=\ \inf_{\bar{\theta}\in\R^{d+1}}v^{\ell,\mathfrak{j}}(\bar{\theta})
=\ z_{D,\ell,\mathfrak{j}}.
\end{align*}
This, together with \eqref{eqn:sgld_bound} and \eqref{eqn:dro_problem_dual_transformed_discrete} as well as that $N=2^{ m(\ell+\mathfrak{j})}$ implies that
\begin{align*}
\E_{\mathbb{P}}[u^{\ell,\mathfrak{j}}(\hat{{\theta}}^{\lambda, \delta,\ell,\mathfrak{j}}_n)]-z_{P,\ell,\mathfrak{j}}
=&\ \E_{\mathbb{P}}[u^{\ell,\mathfrak{j}}(\hat{{\theta}}^{\lambda, \delta,\ell,\mathfrak{j}}_n)]-z_{D,\ell,\mathfrak{j}}\nonumber\\
\leq&\ \E_{\mathbb{P}}[v^{\ell,\mathfrak{j}}(\hat{\bar{\theta}}^{\lambda, \delta,\ell,\mathfrak{j}}_n)] -z_{D,\ell,\mathfrak{j}}\nonumber\\
\leq&\ \E_{\mathbb{P}}[v^{\delta,\ell,\mathfrak{j}}(\hat{\bar{\theta}}^{\lambda, \delta,\ell,\mathfrak{j}}_n)] -  z_{D,\ell,\mathfrak{j},\delta}+ \delta \log{N}\nonumber\\
\leq&\ C_{1,\delta,\ell,\mathfrak{j},\beta} e^{-c_{\delta,\beta}\lambda n/4} + C_{2,\delta,\ell,\mathfrak{j},\beta} \lambda^{1/4} + C_{3,\delta,\beta} + \delta  m(\ell+\mathfrak{j})\log{2},
\end{align*}
where the second inequality is due to
\begin{align*}
v^{\delta,\ell,\mathfrak{j}}(\bar{\theta}) \leq v^{\ell,\mathfrak{j}}(\bar{\theta}) \leq v^{\delta,\ell,\mathfrak{j}}(\bar{\theta}) + \delta\log{N},\qquad \bar{\theta}\in\R^{d+1},
\end{align*}
which follows from the definitions of $v^{\ell,\mathfrak{j}}$ and $v^{\delta,\ell,\mathfrak{j}}$ in \eqref{defn:small_v_discrete} and \eqref{defn:small_v_delta_discrete}, as well as the smoothing error given in Lemma \ref{lemma:smooth_max}. This completes the proof.
\end{proof}

\subsection*{Proof of Proposition \ref{prop:discretisation}}
To obtain the desired result, we first derive an upper estimate for $|u(\theta)-u^{\ell,\mathfrak{j}}(\theta)|$, for all $\theta\in\mathbb{R}^d$. Then, we use 
\[
\lbrb{\E_{\mathbb{P}}\left[u(\hat{\theta}_{n}^{\lambda,\delta,\ell,\mathfrak{j}})\right]-\E_{\mathbb{P}}\left[u^{\ell,\mathfrak{j}}(\hat{\theta}_{n}^{\lambda,\delta,\ell,\mathfrak{j}})\right]}\leq \E_{\mathbb{P}}\left[\lbrb{u(\hat{\theta}_{n}^{\lambda,\delta,\ell,\mathfrak{j}})-u^{\ell,\mathfrak{j}}(\hat{\theta}_{n}^{\lambda,\delta,\ell,\mathfrak{j}})}\right]
\]
together with the moment estimate for $\hat{\theta}_{n}^{\lambda,\delta,\ell,\mathfrak{j}}$ in Lemma 4.2 of \cite{zhang2023nonasymptotic} to conclude the proof.\\

Fix $\theta\in\R^d$. By the definition of $u^{\ell,\mathfrak{j}}$ in \eqref{defn:u_discrete}, the duality result in \eqref{eqn:duality_theorem} due to Theorem \ref{theorem:wass_duality}, and by the proof of Lemma \ref{lemma:finite_grid}, one obtains the relation
\begin{align}
u^{\ell,\mathfrak{j}}(\theta)
=&\ \inf_{\alpha\in\R}\left\{\int_{\Xi\cap \mathbb{K}_{\ell,\mathfrak{j}}^m} \max_{y\in \Xi\cap\mathbb{K}_{\ell,\mathfrak{j}}^m}\{U(\theta,y)-\iota(\alpha)|x-y|^p\} \dee \mu_{0,\ell,\mathfrak{j}}(x) + \frac{\eta_1}{2}|\theta|^2+\frac{\eta_2}{2}|\iota(\alpha)|^2\right\} \nonumber\\
=&\ \inf_{\mathfrak{a}\geq 0}\left\{\int_{\Xi\cap\mathbb{K}_{\ell,\mathfrak{j}}^m} \max_{y\in \Xi\cap\mathbb{K}_{\ell,\mathfrak{j}}^m}\{U(\theta,y)-\mathfrak{a}|x-y|^p\} \dee \mu_{0,\ell,\mathfrak{j}}(x) + \frac{\eta_1}{2}|\theta|^2+\frac{\eta_2}{2}|\mathfrak{a}|^2\right\} \nonumber\\
=&\ \inf_{\mathfrak{a}\geq 0}\left\{\int_{\Xi}\sup_{y\in\Xi}\left\{U(\theta,[y]_\mathfrak{j})-\mathfrak{a} |[x]_\mathfrak{j}-[y]_\mathfrak{j}|^p\right\}\ \dee\mu_0(x)+\frac{\eta_1}{2}|\theta|^{2}+\frac{\eta_2}{2}|\mathfrak{a}|^2\right\}.\label{eqn:u_discretised_exp}
\end{align}
Similarly, by the duality result of Theorem \ref{theorem:wass_duality}, the relation
\begin{align}
u(\theta)
=&\ \inf_{\mathfrak{a}\geq 0}\left\{\int_{\Xi}\sup_{y\in\Xi}\left\{U(\theta,y)-\mathfrak{a} |x-y|^p\right\}\ \dee\mu_0(x)+\frac{\eta_1}{2}|\theta|^{2}+\frac{\eta_2}{2}|\mathfrak{a}|^2\right\}\label{eqn:u_exp}
\end{align}
holds. Observe that, by following the exact same argument in \eqref{eqn:start_arg_1} from the proof of Lemma \ref{lemma:compactness} , one obtains the following lower bound
\begin{align}
\min\Bigg{\{}&\ \int_{\Xi}\sup_{y\in\Xi}\left\{U(\theta,y)-\mathfrak{a} |x-y|^p\right\}\ \dee\mu_0(x)+\frac{\eta_1}{2}|\theta|^{2}+\frac{\eta_2}{2}|\mathfrak{a}|^2,\nonumber\\
&\ \int_{\Xi}\sup_{y\in\Xi}\left\{U(\theta,[y]_\mathfrak{j})-\mathfrak{a} |[x]_\mathfrak{j}-[y]_\mathfrak{j}|^p\right\}\ \dee\mu_0(x)+\frac{\eta_1}{2}|\theta|^{2}+\frac{\eta_2}{2}|\mathfrak{a}|^2\Bigg{\}}\nonumber\\
&\ \hspace{-50pt} \geq -\tilde{K}_\nabla (1+|\theta|)-\mathfrak{a}2^pM_\Xi^p + \frac{\eta_1}{2}|\theta|^2+\frac{\eta_2}{2}|\mathfrak{a}|^2\label{eqn:u_lower_bound}
\end{align}
uniformly in $\mathfrak{j}$. Denote by $\mathfrak{K}_\theta$ the quantity
\begin{align*}
\mathfrak{K}_\theta :=&\ \frac{2}{\sqrt{\eta}_2} \left(1+\sup_{x\in\Xi}|U(\theta,x)|+\tilde{K}_\nabla (1+|\theta|)\right) + \frac{2^{p+2}M_\Xi^p}{\eta_2}.
\end{align*}

Then, for all $\mathfrak{a}>\mathfrak{K}_\theta$, we obtain the inequality
\begin{align*}
&\ -\tilde{K}_\nabla (1+|\theta|)-\mathfrak{a}2^pM_\Xi + \frac{\eta_1}{2}|\theta|^2+\frac{\eta_2}{2}|\mathfrak{a}|^2\nonumber\\
\geq&\ -\tilde{K}_\nabla (1+|\theta|)-\mathfrak{a}2^pM_\Xi^p\cdot \frac{\mathfrak{a}\eta_2}{2^{p+2}M_\Xi^p} + \frac{\eta_1}{2}|\theta|^2+\frac{\eta_2}{2}|\mathfrak{a}|^2\nonumber\\
>&\ -\tilde{K}_\nabla (1+|\theta|) + \frac{\eta_1}{2}|\theta|^2+\frac{\eta_2}{4}|\mathfrak{K}_\theta|^2\nonumber\\
>&\ -\tilde{K}_\nabla (1+|\theta|) + \frac{\eta_1}{2}|\theta|^2+\frac{\eta_2}{4}\left|\frac{2}{\sqrt{\eta}_2} \left(1+\sup_{x\in\Xi}|U(\theta,x)|+\tilde{K}_\nabla (1+|\theta|)\right)\right|^2\nonumber\\
>&\ 1+ \sup_{x\in\Xi}|U(\theta,x)| + \frac{\eta_1}{2}|\theta|^2\nonumber\\
>&\ \max\{u(\theta),u^{\ell,\mathfrak{j}}(\theta)\}.
\end{align*}

This, in particular, implies from \eqref{eqn:u_lower_bound} that for all $\mathfrak{a}>\mathfrak{K}_\theta$,
\begin{align*}
\int_{\Xi}\sup_{y\in\Xi}\left\{U(\theta,y)-\mathfrak{a} |x-y|^p\right\}\ \dee\mu_0(x)+\frac{\eta_1}{2}|\theta|^{2}+\frac{\eta_2}{2}|\mathfrak{a}|^2 >&\ 
u(\theta),\nonumber\\
\int_{\Xi}\sup_{y\in\Xi}\left\{U(\theta,[y]_\mathfrak{j})-\mathfrak{a} |[x]_\mathfrak{j}-[y]_\mathfrak{j}|^p\right\}\ \dee\mu_0(x)+\frac{\eta_1}{2}|\theta|^{2}+\frac{\eta_2}{2}|\mathfrak{a}|^2 >
&\ 
 u^{\ell,\mathfrak{j}}(\theta).
\end{align*}

Therefore, by the same argument in \eqref{eqn:start_arg_1} to \eqref{eqn:end_arg_1} from the proof of Lemma \ref{lemma:compactness}, the infimum in \eqref{eqn:u_discretised_exp} and \eqref{eqn:u_exp} are both attained in $[0,\mathfrak{K}_\theta]$. It follows by applying the same argument in the proof of Proposition \ref{prop:quadrature} that
\begin{align}
|u(\theta)-u^{\ell,\mathfrak{j}}(\theta)|
\leq&\ \sup_{\mathfrak{a}\in[0,\mathfrak{K}_\theta]}\int_\Xi\sup_{y\in\Xi}\left|U(\theta,y)-\mathfrak{a}|x-y|^p-U(\theta,[y]_j)+\mathfrak{a}|[x]_j-[y]_j|^p\right|\ \dee \mu_0(x)\nonumber\\
\leq&\ J_U(1+|\theta|)  \frac{\sqrt{m}}{2^j}+p\mathfrak{K}_\theta(1+4M_\Xi)^{p-1}\frac{2\sqrt{m}}{2^j}\nonumber\\
=&\ J_U(1+|\theta|)  \frac{\sqrt{m}}{2^j}+
	p(1+4M_\Xi)^{p-1}\frac{2^{p+3}M_\Xi^p\sqrt{m}}{\eta_2 2^j} \nonumber\\
&\ +p\left(1+\sup_{x\in\Xi}|U(\theta,x)|+\tilde{K}_\nabla (1+|\theta|)\right)(1+4M_\Xi)^{p-1}\frac{4\sqrt{m}}{\sqrt{\eta_2}2^j}\nonumber\\
\leq&\ J_U(1+|\theta|)\frac{\sqrt{m}}{2^j}+
	p(1+4M_\Xi)^{p-1}\frac{2^{p+3}M_\Xi^p\sqrt{m}}{\eta_2 2^j} \nonumber\\
&\ +p\left(1+2\tilde{K}_\nabla (1+|\theta|)\right)(1+4M_\Xi)^{p-1}\frac{4\sqrt{m}}{\sqrt{\eta_2}2^j}\nonumber\\
\leq&\ \frac{\sqrt{m}}{2^j}\Bigg{[}J_U +\frac{4p}{\sqrt{\eta_2}}(1+4M_\Xi)^{p-1}(1+2\tilde{K}_\nabla)\nonumber\\
&\ + 
	\frac{2^{p+3}pM_\Xi^p}{\eta_2}(1+4M_\Xi)^{p-1}+\left(J_U +\frac{8p\tilde{K}_\nabla}{\sqrt{\eta_2}}(1+4M_\Xi)^{p-1}\right)|\theta|\Bigg{]}.\label{eqn:primal_diff_fixed_theta}
\end{align}

Next, we proceed to obtain a moment estimate for $\hat{\theta}_{n}^{\lambda,\delta,\ell,\mathfrak{j}}$. An application of Lemma 4.2 of \cite{zhang2023nonasymptotic} yields, for $\lambda\in(0,\lambda_{\max,\delta})$ where $\lambda_{\max,\delta}$ is as defined in \eqref{defn:step}, the second moment bound
\begin{align*}
&\ \E_\mathbb{P}\left[|\hat{\bar{\theta}}_{n}^{\lambda,\delta,\ell,\mathfrak{j}}|^2\right]\nonumber\\
\leq&\ e^{-a\lambda(n+1)}\E_\mathbb{P}\left[|\hat{\bar{\theta}}_{0}|^2\right]+\left(2\lambda_{\max,\delta}\sup_{x\in\Xi}|\nabla_{\bar{\theta}}\tilde{V}^{\delta,\ell,\mathfrak{j}}(0,x)|^2 + 2b + 2(d+1)/\beta\right)(\lambda_{\max,\delta}+a^{-1}).
\end{align*}
Note that by the growth condition of \eqref{eqn:grad_v_growth}, it holds that
\begin{align*}
\sup_{x\in\Xi}|\nabla_{\bar{\theta}}\tilde{V}^{\delta,\ell,\mathfrak{j}}(0,x)|^2
\leq&\ \left(K_\nabla + 2^p M_\iota M_\Xi^p \right)^2,\nonumber
\end{align*}
so that 
\begin{align}
\E_\mathbb{P}\left[|\hat{\bar{\theta}}_{n}^{\lambda,\delta,\ell,\mathfrak{j}}|^2\right]
\leq&\ e^{-a\lambda(n+1)}\E_\mathbb{P}\left[|\hat{\bar{\theta}}_{0}|^2\right]+\mathfrak{c}_{1,\delta,\beta}(\lambda_{\max,\delta}+a^{-1}),\label{eqn:moment_bound}
\end{align}

where
\begin{align*}
\mathfrak{c}_{1,\delta,\beta}
:=\ 2\mathfrak{M}_1\lambda_{\max,\delta}+2b+2(d+1)/\beta,\qquad
\mathfrak{M}_1
:=\ \left(K_\nabla  + 2^p M_\iota M_\Xi^p\right)^2.
\end{align*}
Therefore, substituting \eqref{eqn:moment_bound} into \eqref{eqn:primal_diff_fixed_theta} yields
\begin{align*}
&\ \lbrb{\E_{\mathbb{P}}\left[u(\hat{\theta}_{n}^{\lambda,\delta,\ell,\mathfrak{j}})\right]-\E_{\mathbb{P}}\left[u^{\ell,\mathfrak{j}}(\hat{\theta}_{n}^{\lambda,\delta,\ell,\mathfrak{j}})\right]}\nonumber\\
\leq&\ \E_{\mathbb{P}}\left[\lbrb{u(\hat{\theta}_{n}^{\lambda,\delta,\ell,\mathfrak{j}})-u^{\ell,\mathfrak{j}}(\hat{\theta}_{n}^{\lambda,\delta,\ell,\mathfrak{j}})}\right]\nonumber\\
\leq&\ \frac{\sqrt{m}}{2^\mathfrak{j}}\Bigg{[}J_U +\frac{4p}{\sqrt{\eta_2}}(1+4M_\Xi)^{p-1}(1+2\tilde{K}_\nabla )\nonumber\\
&\ + \frac{2^{p+3}pM_\Xi^p}{\eta_2}(1+4M_\Xi)^{p-1}+\left(J_U +\frac{8p\tilde{K}_\nabla}{\sqrt{\eta_2}}(1+M_\Xi)^{p-1}\right)\E_{\mathbb{P}}\left(\left[|\hat{\bar{\theta}}_{n}^{\lambda,\delta,\ell,\mathfrak{j}}|^2\right]\right)^{1/2}\Bigg{]}\nonumber\\
\leq&\ \frac{\sqrt{m}}{2^j}\Bigg{[}J_U +\frac{4p}{\sqrt{\eta_2}}(1+4M_\Xi)^{p-1}(1+2\tilde{K}_\nabla )\nonumber\\
&\ + \frac{2^{p+3}pM_\Xi^p}{\eta_2}(1+4M_\Xi)^{p-1}+\left(J_U +\frac{8p\tilde{K}_\nabla}{\sqrt{\eta_2}}(1+4M_\Xi)^{p-1}\right)\mathfrak{c}_{1,\delta,\beta}^{1/2}(\lambda_{\max,\delta}+a^{-1})^{1/2}\nonumber\\
&\ + \left(J_U+\frac{8p\tilde{K}_\nabla}{\sqrt{\eta_2}}(1+4M_\Xi)^{p-1}\right)e^{-a\lambda(n+1)/2}\left(\E_\mathbb{P}\left[|\hat{\bar{\theta}}_{0}|^2\right]\right)^{1/2}\Bigg{]}\nonumber\\
=&\ \frac{\sqrt{m}(\tilde{C}_4 + C_{5,\delta,\beta} + C_{6}e^{-a\lambda(n+1)/2})}{2^\mathfrak{j}},
\end{align*}
where
\begin{align}\label{const:c4_c5_c6}
\begin{split}
\tilde{C}_4 :=&\ J_U+\frac{4p}{\sqrt{\eta_2}}(1+4M_\Xi)^{p-1}(1+2\tilde{K}_\nabla)+\frac{2^{p+3}pM_\Xi^p}{\eta_2}(1+4M_\Xi)^{p-1},\\
C_{5,\delta,\beta} :=&\ \mathfrak{C}_4\mathfrak{c}_{1,\delta,\beta}^{1/2}(\lambda_{\max,\delta}+a^{-1})^{1/2},\\
C_6 :=&\ \mathfrak{C}_4\left(\E_\mathbb{P}\left[|\hat{\bar{\theta}}_{0}|^2\right]\right)^{1/2},\\
\mathfrak{C}_4 :=&\ \left(J_U+\frac{8p\tilde{K}_\nabla}{\sqrt{\eta_2}}(1+4M_\Xi)^{p-1}\right),\\
\mathfrak{c}_{1,\delta,\beta}
:=&\ 2\mathfrak{M}_1\lambda_{\max,\delta}+2b+2(d+1)/\beta,\\
\mathfrak{M}_1
:=&\ \left(K_\nabla + 2^p M_\iota M_\Xi^p \right)^2.
\end{split}
\end{align}
This completes the proof.
%
%
\vspace{1cm}
\section*{Acknowledgements}
Financial support by the MOE AcRF Tier 2 Grant MOE-T2EP20222-0013 and the Guangzhou-HKUST(GZ) Joint Funding Programs (No. 2024A03J0630 and No. 2025A03J3322) is gratefully acknowledged.


\newpage

\appendix
\section{Dependence on Key Parameters of Constants}
\label{appendix:A}
\begin{tabularx}{\textwidth}[H]
{ 
	>{\hsize=0.6\hsize\linewidth=\hsize}X
	>{\hsize=0.3\hsize\linewidth=\hsize}X|
	>{\hsize=2.1\hsize\linewidth=\hsize}X
} 
{Constant} & & Dependence on Key Parameters\\
\hline
Proposition \ref{prop:glob_lipschitz} &$L_\delta$ &$\mathcal{O}\left(M_\Xi^{2p}(1+1/\delta)\right)$\\
\hline
Proposition \ref{prop:dissipativity}  & $a$ & --- \\
 & $b$ & $\mathcal{O}\left(M_\Xi^{2p}\right)$\\
\hline

Proposition \ref{prop:excess_risk_discrete}\newline Theorem \ref{theorem:excess risk}  & $\mathfrak{C}_1$ &  ---\\
 & $\mathfrak{C}_2$ & $\mathcal{O}\left(M_\Xi^{2p}\right)$\\
 & $\mathfrak{C}_3$ & $\mathcal{O}\left(M_\Xi^{p}\right)$\\
 & $\tilde{L}_\delta$ & $\mathcal{O}\left(M_\Xi^{2p}(1+1/\delta)\right)$\\
  & $\lambda_{\max,\delta}$ & $\varOmega\left(\frac{1}{M_\Xi^{2p}(1+1/\delta)}\right)$\\
 & $c_{\delta,\beta}$ & $\varOmega\left(\frac{1}{e^{C_\star(1+1/\delta)(\beta+d)M_\Xi^{p}}}\right)$\\
  & $C_{1,\delta,\ell,\mathfrak{j},\beta}$ & $\mathcal{O}\left(e^{C_\star(1+1/\delta)(\beta+d)M_\Xi^{p}}\right)$\\
   & $C_{2,\delta,\ell,\mathfrak{j},\beta}$ & $\mathcal{O}\left(e^{C_\star(1+1/\delta)(\beta+d)M_\Xi^{p}}\right)$\\
    & $C_{3,\delta,\beta}$ & $\mathcal{O}\left((d/\beta)\log\left(C_\star(1+1/\delta)(\beta/d+1)M_\Xi^{p}\right)\right)$\\
\hline

Lemma \ref{lemma:compactness}\newline 
Proposition \ref{prop:discretisation}\newline 
Theorem \ref{theorem:excess risk}   & $R_\mc{K}$ &  $\mathcal{O}\left(M_\Xi^{p}\right)$\\

& $\mathfrak{M}_1$ &  $\mathcal{O}\left(M_\Xi^{2p}\right)$\\

 & $\mathfrak{c}_{1,\delta,\beta}$ &  $\mathcal{O}\left(M_\Xi^{2p}(1+d/\beta)\right)$ \\
 
 & $\mathfrak{C}_4$ & $\mathcal{O}\left(M_\Xi^{p-1}\right)$ \\

 & $\tilde{C}_4 $ & $\mathcal{O}\left(M_\Xi^{2p-1}\right)$ \\
 
  & $C_4 $ & $\mathcal{O}\left(M_\Xi^{2p-1}\right)$ \\
 
 & $C_{5,\delta,\beta}$ & $\mathcal{O}\left(M_\Xi^{2p-1}\sqrt{(1+d/\beta)}\right)$\\
 
 & $C_6$ & $\mathcal{O}\left(M_\Xi^{p-1}\right)$\\
\end{tabularx}

\newpage

\section{Analytic Expression of Constants}
\label{appendix:B}
\begin{tabularx}{\textwidth}[H]
{ 
	>{\hsize=0.5\hsize\linewidth=\hsize}X
	>{\hsize=0.3\hsize\linewidth=\hsize}X|
	>{\hsize=2.2\hsize\linewidth=\hsize}X
} 
{Constant} & & Explicit Expression\\
\hline
Proposition \ref{prop:glob_lipschitz} & $L_\delta$ & $ \tfrac{8K_\nabla \left(K_\nabla + 2^{p} M_\Xi^p M_\iota\right)}{\delta} +2L_\nabla  +2^{p+1} L_\iota M_\Xi^p+ \tfrac{2^{p+3} M_\iota M_\Xi^p \left(K_\nabla +2^{p}M_\Xi^pM_\iota\right)}{\delta}$\\
\hline
Proposition \ref{prop:dissipativity}  & $a$ & $\frac{\min\{\eta_1,\eta_2 a_\iota\}}{2}$ \\
 & $b$ & $\eta_2b_\iota + \frac{2\left(K_\nabla +2^pM_\iota  M_\Xi^p\right)^2}{\min\{\eta_1,\eta_2 a_\iota\}}$\\
\hline

Proposition \ref{prop:excess_risk_discrete}\newline Theorem \ref{theorem:excess risk}  & $\mathfrak{C}_1$ &  $\frac{\min\{a,a^{1/3}\}}{64}$\\
 & $\mathfrak{C}_2$ & $(8K_\nabla +2^{p+3} M_\iota M_\Xi^p)(K_\nabla  + 2^{p}M_\Xi^pM_\iota)$\\
 & $\mathfrak{C}_3$ & $2L_\nabla +2^{p+1} L_\iota M_\Xi^p + \eta_1+\eta_2\tilde{L}_\iota + 1$\\
 & $\tilde{L}_\delta$ & $ \frac{\mathfrak{C}_2}{\delta} + \mathfrak{C}_3$\\
  & $\lambda_{\max,\delta}$ & $\min\left\{\frac{\mathfrak{C}_1}{\tilde{L}_\delta^2},\frac{1}{a}\right\}$\\
 & $c_{\delta,\beta}$ & See \eqref{eqn:assignments} and the explicit expression for $\dot{c}$ in Corollary 2.8 of \cite{zhang2023nonasymptotic}.\\
  & $C_{1,\delta,\ell,\mathfrak{j},\beta}$ & See \eqref{eqn:assignments} and the explicit expression for $C_1^\#$ in Corollary 2.8 of \cite{zhang2023nonasymptotic}.\\
   & $C_{2,\delta,\ell,\mathfrak{j},\beta}$ & See \eqref{eqn:assignments} and the explicit expression for $C_2^\#$ in Corollary 2.8 of \cite{zhang2023nonasymptotic}.\\
    & $C_{3,\delta,\beta}$ & See \eqref{eqn:assignments} and the explicit expression for $C_3^\#$ in Corollary 2.8 of \cite{zhang2023nonasymptotic}.\\
\hline

Lemma \ref{lemma:compactness}\newline 
Proposition \ref{prop:discretisation}\newline 
Theorem \ref{theorem:excess risk}  & $R_\mc{K}$ &  $\left(\left(3\tilde{K}_\nabla + \frac{(\tilde{K}_\nabla+2^pM_\Xi^p)^2}{\min\{\eta_1, \eta_2\}}\right)\frac{4}{\min\{\eta_1, \eta_2\}}\right)^{1/2}$\\

 & $\mathfrak{M}_1$ &  $\left(K_\nabla + 2^p M_\iota M_\Xi^p \right)^2$\\

 & $\mathfrak{c}_{1,\delta,\beta}$ &  $2\mathfrak{M}_1\lambda_{\max,\delta}+2b+2(d+1)/\beta$ \\
 
 & $\mathfrak{C}_4$ & $J_U+\frac{8p\tilde{K}_\nabla}{\sqrt{\eta_2}}(1+4M_\Xi)^{p-1}$ \\

 & $\tilde{C}_4 $ & $J_U+\frac{4p}{\sqrt{\eta_2}}(1+4M_\Xi)^{p-1}(1+2\tilde{K}_\nabla)+\frac{2^{p+3}pM_\Xi^p}{\eta_2}(1+4M_\Xi)^{p-1}$ \\
 
  & $C_4 $ & $\tilde{C}_4 + (J_U+ 2p(1+4M_\Xi)^{p-1})(1+R_\mc{K})$ \\
 
 & $C_{5,\delta,\beta}$ & $\mathfrak{C}_4\mathfrak{c}_{1,\delta,\beta}^{1/2}(\lambda_{\max,\delta}+a^{-1})^{1/2}$\\
 
 & $C_6$ & $\mathfrak{C}_4\left(\E_\mathbb{P}\left[|\hat{\bar{\theta}}_{0}|^2\right]\right)^{1/2}$\\
\end{tabularx}
\newpage

\bibliographystyle{plainnat}
\bibliographystyle{alpha}
\bibliography{references}



\end{document}